\documentclass{amsart}
\usepackage{etex}
 \usepackage[foot]{amsaddr}
\usepackage{graphicx}
\usepackage{epstopdf}

\usepackage{cite}
\usepackage{tikz, tikz-cd, tkz-graph}

\usepackage{xcolor}

\usepackage{hyperref}
\hypersetup{
  colorlinks,
  linkcolor={red!50!black},
  citecolor={blue!50!black},
  urlcolor={blue!80!black}
}

\usepackage[english]{babel}
\usepackage{amsmath}
\usepackage{amssymb}
\usepackage{amsthm}
\usepackage{xypic}
\usepackage{longtable}
\usepackage{array}

 \usepackage[enableskew]{youngtab}

\usepackage{epsfig}
\usepackage{hyperref}
\usepackage{enumitem}
\usepackage{booktabs}
 \usepackage{longtable} 
 \usepackage{pdflscape} 
 \usepackage{colortbl} 
 \usepackage{arydshln} 
 \usepackage{calc} 
\Yboxdim{6pt}

\graphicspath{{./images/}}

\newenvironment{claim}[1]{\par\noindent\underline{Claim:}\space#1}{}

\newsavebox{\cimatrixbox}
\newlength{\cimatrixheight}
\newlength{\cimatrixshift}

\newcommand{\evnrow}{\rowcolor[gray]{0.95}}
\newcommand{\oddrow}{}

\newtheorem{thm}{\bf Theorem}[section]
\newtheorem{eg}[thm]{\bf Example}
\newtheorem{prop}[thm]{\bf Proposition}
\newtheorem{cor}[thm]{\bf Corollary}
\newtheorem{rem}[thm]{\bf Remark}
\newtheorem{mydef}[thm]{\bf Definition}
\newtheorem{lem}[thm]{\bf Lemma}

\newcommand{\CC}{\mathbb{C}}
\newcommand{\RR}{\mathbb{R}}
\newcommand{\ZZ}{\mathbb{Z}}
\newcommand{\QQ}{\mathbb{Q}}
\newcommand{\NN}{\mathbb{N}}
\newcommand{\PP}{\mathbb{P}}
\newcommand{\LL}{\mathbb{L}}

\newcommand{\br}{\mathbf{r}}

\DeclareMathOperator{\Rep}{\mathrm{Rep}}

\DeclareMathOperator{\Hom}{\mathrm{Hom}}
\DeclareMathOperator{\GL}{\mathrm{GL}}
\DeclareMathOperator{\Fl}{\mathrm{Fl}}

\DeclareMathOperator{\Gr}{\mathrm{Gr}}

\DeclareMathOperator{\Proj}{\mathrm{Proj}}
\DeclareMathOperator{\Pl}{\mathrm{Pl}}

\renewcommand{\emptyset}{\varnothing}

\newtheorem{conj}{Conjecture}[section]

\newcolumntype{C}[1]{>{\centering\let\newline\\\arraybackslash\hspace{0pt}}m{#1}}
\title{Laurent polynomial mirrors for quiver flag zero loci}
\author{Elana Kalashnikov}
\address[E.~Kalashnikov]{{Department of Mathematics \\ 
University of Waterloo \\ 
Waterloo, Ontario}}
\email{e2kalash@uwaterloo.ca}

\begin{document}
\begin{abstract}
The classification of Fano varieties is an important open question, motivated in part  by the MMP. Smooth Fano varieties have been classified up to dimension three: one interesting feature of this classification is that they can all be described as certain subvarieties in GIT quotients; in particular, they are all either toric complete intersections (subvarieties of toric varieties) or quiver flag zero loci (subvarieties of quiver flag varieties). There is a program to use mirror symmetry to classify Fano varieties in higher dimensions. Fano varieties are expected to correspond to certain Laurent polynomials under mirror symmetry; given such a Fano toric complete intersections, one can produce a Laurent polynomial via the Hori--Vafa mirror. In this paper, we give a method to find Laurent polynomial mirrors to Fano quiver flag zero loci in $Y$-shaped quiver flag varieties. To do this, we generalize the Gelfand--Cetlin degeneration of type A flag varieties to Fano $Y$-shaped quiver flag varieties, and describe these degenerations as toric quiver moduli spaces.  We find conjectural mirrors to 99 four dimensional Fano quiver flag zero loci, and check them up to 20 terms of the period sequence.
\end{abstract}
\maketitle

\section{Introduction}\label{sec:intro}
The classification of smooth Fano varieties is an important open question, motivated, in part, by the role Fano varieties play in the minimal model program. The classification of Fano varieties is only known up to dimension three: in three dimensions, this was a major undertaking by Iskovskih \cite{Iskovskih1977} and Mori--Mukai \cite{MoriMukai1982,MoriMukai2003}. One of the interesting features of the list of Fano threefolds is that all can be found as subvarieties cut out by sections of representation theoretic bundles on a GIT quotient $V/\!/G$ \cite{CoatesCortiGalkinKasprzyk2016}.   

A representation theoretic bundle on a smooth GIT quotient $V/\!/G$ is a bundle produced from a representation $E$ of $G$: letting $G$ act diagonally on the product, we obtain a bundle  $E_G:=(E \times V)/\!/G \to V/\!/G$. In the Mori--Mukai classification, the ambient GIT quotient $V/\!/G$ is in fact either a toric variety or a quiver flag variety. When the ambient GIT quotient is a toric variety, a subvariety cut out by a section of a representation theoretic bundle is a toric complete intersection. When the ambient GIT quotient is a quiver flag variety, such a subvariety is called a \emph{quiver flag zero locus}. 

The classification of Fano fourfolds in unknown. Nevertheless, because both toric varieties and quiver flag varieties are combinatorially described, the search for all Fano representation-theoretic subvarieties of toric varieties and quiver flag varieties with bounded codimension is a finite problem. The search for such Fano fourfolds with codimension at most 4 was completed in \cite{CoatesKasprzykPrince2015} for toric varieties and \cite{kalashnikov} for quiver flag varieties. 

Under mirror symmetry, Fano varieties are conjectured to correspond to certain Laurent polynomial mirrors. A Fano variety $X$ is mirror to a Laurent polynomial $f$ if two power series agree. Laurent polynomial mirrors are known for all Fano threefolds and the toric Fano fourfolds found by \cite{CoatesKasprzykPrince2015}. In this paper, we find conjectural Laurent polynomial mirrors to $99$ of the $141$ Fano fourfolds found in \cite{kalashnikov}, and verify that the first 20 terms of the period are correct. The theoretical basis of this result is a generalization of the famous Gelfand--Cetlin toric degeneration to $Y$-shaped quiver flag varieties, which is of independent interest. The degenerate toric variety is a quiver moduli space with nice combinatorial properties. 

We now describe the background and results of the paper in more detail. 

\subsection*{The mirror symmetry conjecture}
A program of Coates,  Corti, Galkin, Golyshev, Kasprzyk and others is to use mirror symmetry to classify Fano varieties. Mirror symmetry for Fano varieties is conjectured \cite{kasprzyktveiten} to be of the following form (see \cite[Conjecture 5.1]{kasprzyktveiten} for a more precise statement):
\begin{conj} Fano varieties up to deformation to correspond to rigid maximally mutable Laurent polynomials up to mutation. 
\end{conj}
For the definition of rigid maximally mutable Laurent polynomials, see Definition \ref{defn:rigidmaxmutable}. A Fano variety $X$ is said to be \emph{mirror} to $f$ if the regularised quantum period  of $X$ equals the classical period of $f$. The classical period of a Laurent polynomial (Definition \ref{defn:classicalperiod}) is mutation invariant. 

In any dimension, a mirror of a Fano toric variety can be found via the Hori--Vafa mirror. For a Fano toric complete intersection with a \emph{convex partition} (a technical condition satisfied by all known smooth Fano toric complete intersections, see Definition \ref{defn:convexpartition}), a Laurent polynomial mirror can be found from the Hori--Vafa mirror via certain substitutions; this is called the Przyjalkowski method \cite{laurentinversion}.  In \cite{CoatesKasprzykPrince2015}, they found mirror rigid maximally mutable Laurent polynomial mirrors for all Fano toric complete intersections.  
\subsection*{Mirror symmetry and toric degenerations}
More generally, it is expected that if $f$ is a Laurent polynomial mirror to a Fano variety $Y$, then there is a toric degeneration of $Y$ to the toric variety given by the spanning fan of the Newton polytope of $f$. This was proven for toric complete intersections and mirrors given by the Przyjalkowski method by \cite{doranharder}. One can reverse this to produce conjectural mirrors to Fano varieties: given a toric degeneration to $X$ of $Y$, a candidate mirror to $Y$ is a rigid maximally mutable Laurent polynomial supported on the polytope of $X$. 
 
In this paper, we describe a method of producing Laurent polynomial mirrors to Fano quiver flag zero loci in quiver flag varieties of a restricted type, called $Y$-shaped quiver flag varieties. Since it is unclear how to find a toric degeneration of a quiver flag zero locus directly, we instead degenerate the ambient quiver flag variety to a toric variety, and then apply a generalization of the Przyjalkowski method to the degeneration of the quiver flag zero locus. 

This method depends on having very good combinatorial control over the degeneration of the ambient quiver flag variety; the degeneration we present is a generalization of the beautiful Gelfand--Cetlin toric degeneration of flag varieties,  and we can describe the degenerate toric variety as a quiver moduli space itself. This new combinatorial description allows us to write down a complete intersection in this toric variety corresponding to a zero locus in the quiver flag variety. When this complete intersection has a convex partition, we can then use the Przyjalkowski method to produce a Laurent polynomial, and adjust the coefficients so that it is rigid maximally mutable (confirmed by computer code of Kasprzyk). We conjecture that this Laurent polynomial is mirror to the original Fano quiver flag zero locus. In this way, we construct candidate rigid maximally mutable Laurent polynomial mirrors to 99 of the quiver flag zero loci found in \cite{kalashnikov} (and check this up to 20 terms of the period sequence). These can be found in Appendix \ref{ap:table}. Interestingly, the remaining 42 can all be modeled as subvarieties of $Y$-shaped quiver flag varieties, but the method fails either because of a lack of convex partition or or the toric degeneration has a smaller class group than expected.
 \subsection*{The Gelfand--Cetlin degeneration of a $Y$-shaped quiver flag variety}
The generalization of the Gelfand-Cetlin degeneration is constructed as follows. By re-writing the quiver flag variety using incidence conditions, we describe a quiver analogue of the Pl\"ucker algebra and Pl\"ucker relations of flag varieties; see definitions \ref{def:plalg} and \ref{def:plrel}. The quiver Pl\"ucker algebra gives coordinates on the quiver flag variety, and the quiver Pl\"ucker relations give the ideal cutting out the quiver flag variety from a product of projective spaces. 
\begin{thm} Let $M_Q$ be a $Y$-shaped quiver flag variety. There is a SAGBI basis of the quiver Pl\"ucker algebra. 
\end{thm} 
In the case of flag varieties, initial monomials of polynomials in the quiver Pl\"ucker algebra are indexed by semi-standard tableau. In the generalization to $Y$-shaped quivers, this is still true, albeit with skew tableau.  The restriction to $Y$-shaped quiver makes this possible.  

This determines a flat degeneration of $M_Q$ to a toric variety. In \cite{flagdegenerations}, the combinatorics of the degeneration for the flag case was described via ladder diagrams.  We define ladder diagrams for a $Y$-shaped quiver flag variety with quiver $Q$ and associate to them another quiver called the \emph{ladder quiver} $L(Q)$ with dimension $1$ assigned to each vertex. This quiver gives rise to a GIT quotient which is in this case a toric variety, which we call $X(Q)$. We show that $X(Q)$ is the degenerate toric variety coming from the SAGBI basis degeneration. 
\begin{thm} Let $M_Q$ be a quiver flag variety where $Q$ is a $Y$-shaped quiver. There is a flat toric degeneration of $M_Q$ to $X(Q)$, which is a Gorenstein toric variety.
\end{thm}
This generalizes the Gelfand--Cetlin degeneration along with its ladder diagram to $Y$-shaped quiver flag varieties. Ladder diagrams for Grassmannians also play an important role in their cluster structure (see \cite{RW}), where they can be interpreted as a particularly special plabic graph, defining both a torus chart and a toric degeneration. In this torus chart, the Pl\"ucker coordinate mirror \cite{MarshRietsch} agrees with the mirror of the Gelfin-Cetlin toric degeneration. It would be interesting to see if this richer structure can be generalized to $Y$-shaped quiver flag varieties. 

\subsection*{Finding Laurent polynomial mirrors}
A complete intersection in a flag variety degenerates to a complete intersection in the degenerate toric variety. The situation is more complicated for higher rank bundles, as their sections cannot generally be described via Pl\"ucker coordinates (which is how the degeneration is constructed). However, via the proof of the above theorem, we can see that under this degeneration of a $Y$-shaped quiver, a representation theoretic vector bundle corresponds to a direct sum of rank one reflexive sheaves on the degenerate toric variety (see \S \ref{sec:mirrors} for a more precise statement).  The Pl\"ucker coordinate mirror $W_P$ of \cite{MarshRietsch} for the Grassmannian $\Gr(n,k)$ exhibits analogous behaviour: when expanded in the torus chart corresponding to the ladder diagram, some of the monomials of $W_P$ split into precisely $k$ monomials, and these $k$ monomials correspond to a choice of $k$ rank 1 reflexive sheaves on the degenerate toric variety giving rise to the tautological bundle on $\Gr(n,k)$ (see Appendix \ref{ap:plucker} for more details).

The ladder diagram gives a simple description of the associated Weil divisors. In this way, given a quiver flag zero locus in a $Y$-shaped quiver, we find a `complete intersection' in a toric variety. That is, we find  the GIT data of a singular toric variety, together with weights defining divisors. If the GIT and weight data has a convex partition, we can formally apply the Przyjalkowski method to the data of the weights and divisors, and arrive at a Laurent polynomial. Such a convex partition always exists if the zero locus is cut out by line bundles:
\begin{thm}[see Theorem \ref{extendprz}] For any Fano variety cut out of a Fano quiver flag variety by line bundles, the extended Przyjalkowski method produces a Laurent polynomial.
\end{thm}

When the quiver flag zero locus is not a complete intersection, even when there is a convex partition, the resulting Laurent polynomial is often \emph{not} rigid maximally mutable and is \emph{not} a mirror for the quiver flag zero locus. However, in all examples there is a unique rigid maximally mutable Laurent polynomial with the same Newton polytope whose period sequence is correct, up to 20 terms. This gives significant evidence towards rigid maximally mutable Laurent polynomials being the correct class of Laurent polynomials to consider.

\subsection*{Plan of the paper}
In \S \ref{sec:definitions}, we discuss mirror symmetry for Fano varieties and give the definitions needed to state it precisely and recall basic definitions and constructions for quiver flag varieties. In \S \ref{sec:SAGBI}, we introduce the class of $Y$-shaped quiver flag varieties, and describe a SAGBI basis for sections of certain line bundles on these quiver flag varieties, which produces a toric degeneration. In \S \ref{sec:ladder}, we define ladder diagrams and ladder quivers for $Y$-shaped quiver flag varieties, and prove that the SAGBI degeneration is precisely a degeneration to singular toric variety arising the from the ladder quiver. In the next section, \S \ref{sec:przy}, we introduce a generalization of the Przyjalkowski method and explain how to produce a Laurent polynomial from the data of the weights and divisors associated to a toric complete intersection.  In the final section of the paper \S \ref{sec:mirrors}, we explain how to use this degeneration to find candidate Laurent polynomial mirrors for quiver flag zero loci with a convex partition.  Finally, we give one example of a degeneration and mirror outside of the family of $Y$-shaped quiver flag varieties: here, the degeneration is to a \emph{bound} ladder quiver. In Appendix \ref{ap:plucker}, we discuss a connection to the Pl\"ucker coordinate mirror. In Appendix \ref{ap:table}, we give a table with the mirrors to 99 of the quiver flag zero loci found in \cite{kalashnikov}, produced using the method outlined in the paper. 
\subsection*{Acknowledgments}
The author is very grateful for helpful conversations with Tom Coates, Charles Doran, Alexander Kasprzyk, Thomas Prince, and Lauren Williams. 
\section{Mirror symmetry for Fano varieties}\label{sec:definitions}
In this section, we give the definitions required to more precisely state the conjectural equivalence between Fano varieties up to deformation and rigid maximally mutable Laurent polynomials up to mutation (see, for example \cite{AkhtarCoatesGalkinKasprzyk2012} and \cite{GalkinUsnich2010}). 
\subsection{The quantum period}
On the A side (the Fano side) the main invariant is the quantum period. For a more detailed introduction to Gromov--Witten invariants, quantum cohomology, and the quantum period, see for example \cite{CoatesCortiGalkinKasprzyk2016} (in particular, they record the formula for the quantum period of a Fano toric complete intersection).

Let $Y$ be a smooth projective variety. Given $n \in \ZZ_{\geq 0}$ and $\beta \in H_2(Y)$, let  $M_{0,n}(Y,\beta)$ be the moduli space of genus zero stable maps to $Y$ of class $\beta$, and with $n$ marked points~\cite{Kontsevich95}. While this space may be highly singular and have components of different dimensions, it has a \emph{virtual fundamental class} $[M_{0,n}(Y,\beta)]^{virt}$ of the expected dimension~\cite{BehrendFantechi97,LiTian98}. There are natural evaluation maps $ev_i: M_{0,n}(Y,\beta) \to Y$ taking the class of a stable map $f: C \to Y$ to $f(x_i)$, where $x_i \in C$ is the $i^{th}$ marked point. There is also a line bundle $L_i \to M_{0,n}(Y,\beta)$ whose fiber at $f: C \to Y$ is the cotangent space to $C$ at $x_i$. The first Chern class of this line bundle is denoted $\psi_i$. Define:
\begin{equation}
  \label{eq:GW}
  \langle \tau_{a_1}(\alpha_1),\dots,\tau_{a_n}(\alpha_n) \rangle_{n,\beta} = \int_{[M_{0,n}(Y,\beta)]^{virt}} \prod_{i=1}^n ev_i^*(\alpha_i) \psi_i^{a_i}
\end{equation}
where the integral on the right-hand side denotes cap product with the virtual fundamental class.  If $a_i=0$ for all $i$, this is called a (genus zero) Gromov--Witten invariant and the $\tau$ notation is omitted; otherwise it is called a descendant invariant. It is deformation invariant.

\begin{mydef}\label{defn:quantumperiod} Let $X$ be a smooth Fano variety. The \emph{quantum period} of $X$ is
\[G_X(t)=\sum_{i=0}^\infty a_i t^i, \hspace{5mm} a_i:=\sum_{\beta \in H_2(X,\ZZ) \mid \langle \beta, -K_X \rangle =i}\langle \tau_{i-2}([X])\rangle_{1,\beta}.\]
Here $[X]$ is the cohomology class Poincar\'e dual to the class of a point.

The \emph{regularized quantum period} is
\[\tilde{G}_X(t)=\sum_{i=0}^\infty i! a_i t^i\]
which can be interpreted as a Fourier--Laplace transform of $G_X(t)$. 
\end{mydef}
As descendent invariants are deformation invariant, so is the quantum period. A closed form is known for smooth Fano toric complete intersections. The Abelian/non-Abelian correspondence \cite{CiocanFontanineKimSabbah2008,Webb2018, kalashnikov} allows one to compute any number of terms of the quantum period of quiver flag zero loci. 
\subsection{The classical period and mutations}
\begin{mydef}\label{defn:classicalperiod} Let $f$ be a Laurent polynomial in $\CC[x_1^\pm,\dots,x_n^\pm]$. The \emph{classical period} of $f$, denoted $\pi_t(f)$ is
 \[\frac{1}{(2 \pi i)^n} \int_{(S^1)^n} \frac{1}{1-t f} \frac{d x_1}{x_1} \wedge \cdots \wedge \frac{d x_n}{x_n}. \]
 Let $X$ be a smooth Fano variety. A Laurent polynomial $f \in \CC[x_1^\pm,\dots,x_n^\pm]$ is \emph{mirror} to $X$ if the quantum period $X$ is equal to $\pi_f(t)$, the classical period of $f$. 

\end{mydef} 
Repeated applications of the residue theorem allows one to re-write $\pi_f(t)$ as $\sum_{i=0}^\infty a_i t^i,$ where $a_i$ is the constant term in the expansion of $f^i$. 

The class of Laurent polynomials mirror to terminal $\QQ$-factorial Fano varieties is conjectured to be \emph{rigid maximally mutable} Laurent polynomials (although here we only consider smooth Fano fourfolds). To define this class, we first need to define the notion of a mutation. We follow \cite{AkhtarCoatesGalkinKasprzyk2012}. Mutations are compositions of two types of operations -- $\GL(n,\ZZ)$ equivalences and certain birational transformations -- on a Laurent polynomial $f$. For the first, let $A=[a_{ij}] \in \GL(n;\ZZ).$ Then $A$ defines a $\GL(n,\ZZ)$ equivalence $\phi: (\CC^*)^n \to (\CC^*)^n$ via
 \[(x_1,\dots,x_n) \mapsto (\prod_{i=1}^n x_i^{a_{1i}},\dots,\prod_{i=1}^n x_i^{a_{ni}}). \]
This defines a new Laurent polynomial $\phi^*(f)$. For the second type of map, write
 \[f=\sum_{i} C_i(x_1,\dots,x_{n-1}) x_n^i \]
and suppose that $C_i$ is divisible by $h^{-i}$ for a fixed Laurent polynomial $h(x_1,\dots,x_{n-1})$ for all $i<0$ appearing above. This defines a birational transformation $\phi: (\CC^*)^n \dashrightarrow (\CC^*)^n$ via
 \[(x_1,\dots,x_n) \mapsto (x_1,\dots,x_{n-1}, h x_n). \]
 We obtain a new Laurent polynomial $g$ by pullback where
 \[g=\sum_{i} h^i C_i(x_1,\dots,x_{n-1}) x_n^i. \]
A one-step mutation is defined to be the composition of a $\GL(n,\ZZ)$ equivalence, a birational transformation of this type, and another $\GL(n,\ZZ)$ equivalence. A mutation is a composition of one-step mutations.

The polytope $P'$ is defined to be a mutation of $P$ if there exists $f,f'$ such that $P'$ is the Newton polytope of $f'$, $P$ is the Newton polytope of $f$, and $f'$ is a mutation of $f$. A mutation of $P$ is said to be compatible with a Laurent polynomial $f$ if the mutation of $P$ is induced from a mutation of $f$. 

\begin{mydef}[\cite{kasprzyktveiten}] \label{defn:rigidmaxmutable}
Let $f$ be a Laurent polynomial with Newton polytope $P$. We say that $f$ is \emph{rigid maximally mutable} if there is a set of mutations $S$ on $P$ such that $f$ is compatible with all mutations in $S$ and up to scaling, $f$ is the only Laurent polynomial compatible with all mutations in $S$. 
\end{mydef} 

For smooth Fano toric complete intersections in smooth toric varieties subject to some extra technical conditions (the ability to find a \emph{convex partition}), there is a well understood method of producing a Laurent polynomial mirror, called the Przyjalkowski method. This takes the Hori--Vafa mirror of the toric complete intersection, and after certain substitutions, produces a Laurent polynomial.  The paper \cite{laurentinversion} explains this in detail.  However, it isn't known whether this always produces a rigid maximally mutable Laurent polynomial (but there are no known counterexamples). One can formally follow the same method when the toric variety and the toric complete intersection are singular: below, we show examples when this method does not produce a rigid maximally mutable Laurent polynomial. The correct mirror, in this case, is the rigid maximally mutable Laurent polynomial with the same Newton polytope. In \S \ref{sec:przy}, we give an extension of the Przyjalkowski method. 

\subsection{Quiver flag varieties}
We briefly recall the construction of quiver flag varieties.  Quiver flag varieties are generalizations of Grassmannians and type A flag varieties, introduced by Craw \cite{Craw2011}.  Like flag varieties, they are GIT quotients and fine moduli spaces.  We begin by recalling Craw's construction.  A quiver flag variety $M(Q,\br)$ is determined by a quiver $Q$ and a dimension vector $\br$.  The quiver $Q$ is assumed to be finite and acyclic, with a unique source.  Let $Q_0 = \{0,1,\ldots,\rho\}$ denote the set of vertices of $Q$ and let $Q_1$ denote the set of arrows.  Without loss of generality, after reordering the vertices if necessary, we may assume that $0 \in Q_0$ is the unique source and that the number $n_{ij}$ of arrows from vertex $i$ to vertex $j$ is zero unless $i<j$.  Write $s$,~$t :  Q_1 \to Q_0$ for the source and target maps, so that an arrow $a \in Q_1$ goes from $s(a)$ to $t(a)$.  The dimension vector $\br = (r_0,\ldots,r_\rho)$ lies in $\NN^{\rho+1}$, and we insist that $r_0 = 1$. $M(Q,\br)$ is defined to be the moduli space of $\theta$-stable representations of  the quiver $Q$ with dimension vector $\br$. Here $\theta$ is a fixed stability condition defined below, determined by the dimension vector. 

Consider the vector space
\[
\Rep(Q,\br) =\bigoplus_{a \in Q_1}\Hom(\CC^{r_{s(a)}},\CC^{r_{t(a)}})
\]
and the action of $\GL(\mathbf{r}):=\prod_{i=0}^\rho \GL(r_i)$ on $\Rep(Q,\br)$ by change of basis.  The diagonal copy of $\GL(1)$ in $\GL(\br)$ acts trivially, but the quotient $G := \GL(\br)/\GL(1)$ acts effectively; since $r_0 = 1$, we may identify $G$ with $\prod_{i=1}^\rho \GL(r_i)$.  The quiver flag variety $M(Q,\br)$ is the GIT quotient $\Rep(Q,\br)/\!\!/_\theta\, G$, where the stability condition $\theta$ is the character of $G$ given by 
\begin{align*}
  \theta(g) = \prod_{i=1}^\rho \det(g_i), && g = (g_1,\ldots,g_\rho) \in \prod_{i=1}^\rho \GL(r_i).
\end{align*}
 For the stability condition $\theta$, all semistable points are stable.
 \begin{eg} When $\rho=1$, the quiver flag variety is the Grassmannian $\Gr(n_{01},r_1)$ of $r_1$-dimensional quotients in $\CC^{n_{01}}$. 
 \end{eg}
  \begin{eg} Consider the quiver below.
  \begin{center}
   \includegraphics[scale=0.5]{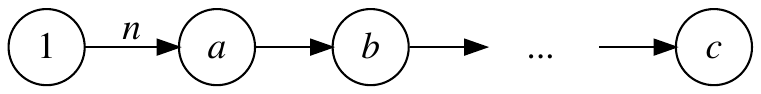}
   \end{center}
  The label in each vertex is the dimension assigned to that vertex. There is a unique choice of labeling for the vertices $\{0,1,\dots,\rho\}$. Here $n_{01}=n$ and $n_{i i+1}=1$. Then the resulting quiver flag variety is the partial flag variety of quotients of $\CC^n$ of dimension $a>b> \cdots>c$. We denote it $\Fl(n;a,b,\dots,c)$. 
   
 \end{eg}

Craw proved that quiver flag varieties are fine moduli spaces for $\theta$-stable representations of $Q$, Mori dream spaces, and towers of Grassmannians. As fine moduli spaces, they are equipped with $\rho$ tautological bundles denoted $W_1,\dots, W_\rho$. The rank of $W_i$ is $r_i$. 

The subvariety construction of flag varieties also generalizes to quiver flag varieties:
\begin{prop}\cite{kalashnikov} \label{prop:zl}
Let $M_Q:=M(Q,\br)$ be a quiver flag variety with $\rho>1.$ Then $M_Q$ is cut out of $P=\prod_{i=1}^\rho \Gr(H^0(M_Q,W_i),r_i)$ by a tautological section of 
\[
E=\bigoplus_{a \in Q_1, s(a) \neq 0} S_{s(a)}^* \otimes F_{t(a)}
\]
where $S_i$ and $F_i$ are the pullbacks to $P$ of the tautological sub-bundle and quotient bundle on the $i^{th}$ factor of $P$. 
\end{prop}
\begin{rem}We avoid the usual notation of $Q$ for the tautological quotient bundle on the Grassmannian as the same notation is already used for the quiver. \end{rem}

Craw also shows that the anti-canonical line bundle of the quiver flag variety $M_\theta(Q,\br)$ is
\[\bigotimes_{i=1}^\rho \det W_i^{\otimes (s_i-s_i')},\]
where
\begin{equation}\label{eq:si}
s_i:=\sum_{a \in Q_1, t(a)=i} r_{s(a)}, \hspace{5mm} s_i':=\sum_{a \in Q_1, s(a)=i} r_{t(a)}.
\end{equation}
\subsection{GIT data for a toric variety} \label{subsection:GIT}
We briefly review the GIT data construction of a toric variety or toric stack. For more details, see \cite[\S 4]{crepantconjecture}. 

The \emph{GIT data} for a toric variety is an $\rho$-dimensional torus $K$ with cocharacter lattice $\LL=\Hom(\CC^*, K)$, and $m$ characters $D_1, \dots, D_m \in \LL^\vee$, together with a stability condition $w \in \LL^\vee \otimes \RR$. We usually assume that the $D_i$ generate the lattice $\LL^\vee$. The characters $D_i$ define an action of $K$ on $\CC^m$, and the toric variety (or stack) associated to this data is the GIT quotient
\[X:=\CC^m/\!/_w K.\]
The \emph{anti-cones} of the GIT data are the sets
\[A_w:=\{ S \subset \{1,\dots,m\} \mid w \in \text{Cone}(D_i: i \in S).\}\]
Here, by cone, we mean the \emph{open} cone generated by the $D_i$: the set of elements that are sums $\sum_{i \in S} a_i D_i,$ where $a_i>0$.  We assume that $\{1,\dots,m\} \in A_w.$

The semi-stable locus $U_w$ of $K$ can be described explicitly as
\[
\Big\{ (z_1,\dots,z_m) \in \CC^m \; \Big|\;  z_i \neq 0 \text{ for } i \in S \text{ and } S \in A_w \Big\},
\]
that is, its elements can have zeroes at $z_i, i \in I$, only if $w$ is in the open cone generated by $D_i$, $i \not \in I$. 

To see fully how the GIT data relates to the fan construction of a toric variety, see \cite{crepantconjecture}. There is an important exact sequence:
$$0 \to M \to \ZZ^{m} \xrightarrow{\pi} \LL^\vee \to 0$$
where $\pi$ is defined by taking the basis element $e_i$ to $D_i$.
 If
\[\bigcap_{S \in A_w}S= \emptyset,\]
then the co-character lattice $\LL^\vee$  can be identified with the class group lattice of the toric variety $Cl(X)$ (otherwise one needs to quotient by the vector space generated by the $D_i, i \in \bigcap_{S \in A_w} S$). 

\section{SAGBI basis degenerations of quiver flag varieties}\label{sec:SAGBI}
In this section, we describe a quiver flag variety using subspaces rather than quotients, leading to a description of the quiver flag variety as a subvariety of a product of projective spaces with a Pl\"ucker relation-type description of the ideal cutting out the quiver flag variety. We give a SAGBI basis for the coordinate ring of this subvariety in the case of a $Y$-shaped quiver flag variety, which gives rise to a toric degeneration. 
\begin{mydef} Let $X$ be a smooth variety. A flat family $\pi: \mathfrak{X} \to \CC$ is a \emph{toric degeneration} of $X$ if the generic fiber is $X$ and the special fiber $X_0$ of $\pi$ is a toric variety. We also require that the family is $\QQ$-Gorenstein.  
\end{mydef}

We first recall the construction for flag varieties. 
\subsection{A degeneration of a flag variety}
We follow Miller and Sturmfels in \cite[Chapter 14]{combinatorial} to present the toric degeneration of the flag variety from \cite{GL}.  When building flag varieties as quiver flag varieties, they arise naturally as flag varieties of \emph{quotients}.  By $\Fl(n;r_1,\dots,r_\rho)$, we denote the (partial) flag variety of quotients of $\CC^n$ of ranks $r_1>r_2>\cdots>r_\rho$ . Points of the flag variety are tuples of surjective maps
\[\CC^n \to H_1 \to H_2 \to \cdots \to H_\rho.\]
Equivalently, however, we can consider the flag variety of subspaces. If we let $V_i$ be the subspace fitting into the exact sequence
\[0 \to V_i \to \CC^n \to H_i \to 0,\]
we can represent a point on the flag variety as a flag
\[V_1 \subset V_2 \subset \cdots \subset V_\rho \subset \CC^n,\]
where $V_i$ has rank $k_i:=n_i-r_i$

We can represent such a flag as a tuple of matrices $A_1,\dots, A_\rho$, where $A_i$ is a $k_i \times (n-k_i)$ matrix, such that the top $k_{i-1}$ rows of $A_i$ is $A_{i+1}$. That is,
 \[A_i:=\begin{bmatrix}
x_{11} & \cdots & x_{1n} \\
\vdots &  & \vdots \\
x_{k_i 1} & \cdots & x_{k_i n} \\
\end{bmatrix}. \]
We can consider the set of $k_i \times k_i$ minors of each $A_i$. Although there are many representations of a flag as a tuple of matrices, the flag is exactly determined by these minors, up to re-scaling each set. In this way, we can view $\Fl(n;r_1,\dots,r_\rho)$ as a subvariety of 
\[\prod_{i=1}^\rho \PP^{\binom{n}{k_i}-1}.\]
It is cut out by the relations between the minors of the $A_i$. The minors are called Pl\"ucker coordinates and the relations are Pl\"ucker relations. 

\begin{rem}
 Let $S_i$ be the vector bundle that fits into the exact sequence
\[0 \to S_i \to \mathcal{O}_{\Fl}^{\oplus n} \to W_i \to 0,\]
where $W_i$ is the tautological quotient bundle on $\Fl(n;r_1,\dots,r_\rho)$. This is the usual tautological sub-bundle on the flag variety. The rank of $S_i$ is $k_i$. There are injective maps
\[0 \to S_1 \to S_2 \to \cdots \to S_\rho.\]
The Cox ring of the flag variety is generated by the sections of each $\det(S_i)^*$, which are generated by the minors of $A_i$. The above description of the flag variety as a subvariety of projective space is exactly the one given by using these line bundles to embed  $\Fl(n;r_1,\dots,r_\rho)$ in $\prod_{i=1}^\rho \PP^{\binom{n}{k_i}-1}.$
\end{rem}

Consider the subalgebra $A \subset \CC[x_{ij}: 1 \leq i \leq r_1, 1 \leq j \leq n]$ generated by these minors. In \cite[\S 14.3]{combinatorial}, they show that the basis given by the minors is a SAGBI basis for $A$ under the monomial order given by the lex ordering on the $x_{ij}$. That is, for any $f$ in the algebra, the initial term of $f$ is a monomial in the initial terms of the basis.   A SAGBI basis defines a flat degeneration of the flag variety in $\prod_{i=1}^\rho \mathbb{P}^{{n \choose k_i}-1}$  to the toric subvariety defined by the monomials which are the initial terms of the basis elements. 

\begin{eg} Consider the flag variety $\Gr(4,2)$. Then the matrix above is
 \[\begin{bmatrix}
x_{11} & x_{12} & x_{13} & x_{14} \\
x_{21} & x_{22} & x_{23} & x_{24}\\
\end{bmatrix}. \]
The algebra is generated by the six $2 \times 2$ minors of this matrix. Their initial terms are
 \[x_{11} x_{22}, x_{11} x_{23},x_{11} x_{24},x_{12} x_{23}, x_{12} x_{24},x_{13} x_{24}. \]
The toric degeneration of $\Gr(4,2)$ is the closure of the image of the map $(\CC^*)^8 \to \mathbb{P}^5$ defined by these monomials. 
\end{eg}
\subsection{Pl\"ucker relations for quiver flag varieties}\label{sec:coords}
The first step towards generalizing this construction to quiver flag varieties is to find appropriate coordinates: that is, we want an analogue of the Pl\"ucker algebra and Pl\"ucker relations for quiver flag varieties. One option would be to consider the Cox ring of the quiver flag variety.  However, generators of the Cox ring of a quiver flag variety are not known in general, so in practice this isn't helpful. Instead, we'll effectively use only a sub-algebra of the Cox ring to find the toric degeneration. The key step is to re-write the universal properties defining the quiver flag variety as incidence conditions. 

Let $M(Q,\br)$ be a quiver flag variety. Use Proposition \ref{prop:zl} to write $M(Q,\br)$ as a subvariety of
 \[P:=\prod_{i=1}^\rho \Gr(\tilde{s}_i,r_i),\]
  \[\tilde{s}_i=\dim H^0(W_i).  \]
    \begin{rem} Craw defines $\tilde{s}_i$ to be the number of paths from $0$ to $i$ in the quiver.  He shows that this agrees with the above definition. 
  \end{rem}
  
The quiver flag variety is cut out by a section of a vector bundle.  We can use the Pl\"ucker embedding of the Grassmannians to view $M(Q,\br)$ as a subvariety of projective space. However, as the section defining the quiver flag variety is not a section of a product of line bundles, it is not easy to find equations defining the quiver flag variety in the projective space. To find such equations, we'll use a slightly different embedding. 

 The stability conditions of a quiver flag variety involve the surjectivity of maps. We want to express them in terms of inclusions, just as we did for the flag variety. Quiver flag varieties do not come in dual pairs, unlike flag varieties and Grassmannians, however, we can dualize the Grassmannian factors of $P$ and then view $M(Q,\br)$ as a subvariety of the result. That is, the Grassmannian $\Gr(n,k)$ is canonically isomorphic to $Gr^{sub}(n-k,n)$, the Grassmannian of subspaces. We can exploit this to view $M(Q,\br)$ as a subvariety of 
\[P^\vee:=\prod_{i=1}^\rho \Gr^{sub}(k_i, \tilde{s}_i),\] 
where 
\[k_i:=\tilde{s}_i-r_i.\]
\begin{rem}We can represent points in the Grassmannian of subvarieties $\Gr^{sub}(n-k,n)$ as full rank $(n-k) \times n$ matrices, where the \emph{span} of the rows give the subspace. These are not quite the coordinates we get from the GIT construction: if we view this as a GIT quotient, we will obtain the Grassmannian $\Gr(n,n-k)$, which is also canonically isomorphic to $\Gr(n,k)$. 
\end{rem}
We'll use the embedding given by the composition
\begin{equation}\label{eq:comp} M(Q,\br) \to P \to P^\vee \to \prod_{i=1}^\rho \PP^{\binom{\tilde{s}_i}{k_i}}.\end{equation}
This is the one given by the line bundles $\det S_i^*.$ In order to find the equations defining $M(Q,\br)$ in the product of projective spaces, however, we'll need to describe it more explicitly.

\begin{lem} Recall that $n_{0i}$ is the number of arrows from $0$ to $i$. A point in $P^\vee$ can be represented as $\rho$ vector subspaces $(V_i \subset \CC^{\tilde{s}_i})_{i=1}^\rho$.  The subvariety $M(Q,\br)$ of $P^\vee$ is characterized the condition that
 \[\bigoplus_{a \in Q_1, s(a) \neq 0, t(a)=i} V_{s(a)} \oplus \CC^{n_{0,i}} \subset V_i\]
 for all $i$.
\end{lem} 
\begin{proof}
Consider $M(Q,\br)$ as a subvariety of $P$. Represent a point $P$ as $\rho$ quotient vector spaces $(\CC^{\tilde{s}_i} \to H_i)_{i=1}^\rho$. 
Note that
\[\CC^{\tilde{s}_i}  =  \bigoplus_{\substack{a \in Q_1, \\ t(a)=i, \\ s(a)>0}} \CC^{\tilde{s}_{s(a)}} \oplus \CC^{n_{0,i}}.\]
By \cite[Proposition 3.1]{kalashnikov}, the point $(H_i)$ is a point in the quiver flag variety exactly when the surjective map
\[ \CC^{\tilde{s}_i} \to H_i\]
factors through
\[\CC^{\tilde{s}_i} \to \bigoplus_{\substack{a \in Q_1, \\ t(a)=i, \\ s(a)>0}} H_{s(a)} \oplus \CC^{n_{0,i}} \to H_{i}.\]
The duality between $P$ and $P^\vee$ takes
\[ V_i \mapsto H_i=\CC^{\tilde{s}_i}/V_i.\]
Therefore when $(V_i)$ represents a point in the quiver flag variety, there is a surjective map
\[\bigoplus_{\substack{a \in Q_1, \\ t(a)=i, \\ s(a)>0}} \CC^{\tilde{s}_{s(a)}}/V_{s(a)} \oplus \CC^{n_{0,i}} \to \CC^{\tilde{s}_i}/V_i.\]
Therefore
 \[\bigoplus_{\substack{a \in Q_1, \\ t(a)=i, \\ s(a)>0}}V_{s(a)} \oplus \CC^{n_{0,i}} \subset V_i\]
 as required. 
\end{proof}
\begin{cor} The quiver flag variety is the fine moduli space for tuples of vector spaces $(V_1,\dots,V_\rho)$ satisfying the incidence conditions
 \[\bigoplus_{\substack{a \in Q_1, \\ t(a)=i, \\ s(a)>0}}V_{s(a)} \oplus \CC^{n_{0,i}} \subset V_i\]
for all $i$.
\end{cor}
\begin{proof} As can be seen from the proof of the lemma, the incidence conditions are simply a reformulation of the conditions defining the quiver flag variety as a fine moduli space using surjective maps. 
\end{proof}

A point in $P^\vee$ can be represented as a tuple of $\rho$ matrices, where the $i^{th}$ matrix is of size $k_i \times \tilde{s}_i$. The row span of each matrix gives the subspace corresponding to the $i^{th}$ factor in $P^\vee$. When a point in $P^\vee$ lies in the quiver flag variety, we can represent it again as a tuple of $\rho$ matrices, $(A_i)_{i=1}^\rho$, where $A_i$ is again of size $k_i \times \tilde{s}_i$, but (like in the flag case) certain submatrices of the $A_i$ are filled in by $A_j, j<i.$

To describe the submatrices, we need to index the columns and rows of the $A_i$ using the quiver. Fix $i \in Q_0$. Each $a \in Q_1, t(a)=i$ corresponds to $r_{s(a)}$ columns; this gives a partition of the columns of $A_i$. We can also partition the rows where there is a subset for each $a \in Q_1, t(a)=i, s(a) \neq 0$ of size $\tilde{s}_{s(a)}-r_{s(a)}$, and one remaining subset of size $s_i-r_i$ as
 \[\tilde{s}_i-\sum_{a \in Q_1, t(a)=i, s(a) \neq 0}(\tilde{s}_{s(a)}-r_{s(a)})-r_i=s_i-r_i. \] 
 See Equation \eqref{eq:si} for the definition of $s_i.$

Using the lemma, we see that we can represent a point of the quiver flag variety by a tuple of matrices $(A_i)_{i=1}^\rho$ such that for $A_i$: the submatrix of $A_i$ corresponding to the rows given for $a \in Q_1, t(a)=i, s(a) \neq 0$ is $0$ except for the sub-submatrix of corresponding to the columns determined by $a$: this sub-submatrix is $A_{s(a)}$.

\begin{eg}\label{eg:coords} Consider the quiver flag variety $M_Q$ given by 
\begin{center}
  \includegraphics[scale=0.5]{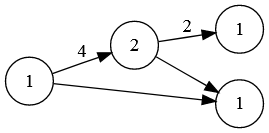}
\end{center}
This quiver flag variety can be seen as a subvariety of $P^\vee=\Gr^{sub}(2,4) \times \Gr^{sub}(7,8) \times \Gr^{sub}(4,5).$ A point of this space given by a triple $(V_1 \subset \CC^4,V_2 \subset \CC^8, V_3 \subset \CC^5)$, of dimension $2,7,$ and $4$ respectively, is in the subvariety $M(Q,\br)$ if $\{0\}\oplus V_1  \subset V_3, V_1 \oplus V_1 \subset V_2$. 
The matrices $A_1, A_2,$ and $A_3$ in this example are
 \[\begin{bmatrix}
x_{11} & x_{12} & x_{13} & x_{14} \\
x_{21} & x_{22} & x_{23} & x_{24}\\
\end{bmatrix},\]
 \[\begin{bmatrix}
z_{11} & z_{12} & z_{13} & z_{14} & z_{15} & z_{16} & z_{17} & z_{18} \\
z_{21} & z_{22} & z_{23} & z_{24} & z_{25} & z_{26} & z_{27} & z_{28} \\
z_{31} & z_{32} & z_{33} & z_{34} & z_{35} & z_{36} & z_{37} & z_{38}\\
x_{11} & x_{12} & x_{13} & x_{14} & 0 & 0 & 0 &0 \\
x_{21} & x_{22} & x_{23} & x_{24} & 0 & 0 & 0 &0\\
0 & 0 & 0 & 0 &x_{11} & x_{12} & x_{13} & x_{14} \\
0 & 0 & 0 & 0 & x_{21} & x_{22} & x_{23} & x_{24} \\

\end{bmatrix}, \]
and
 \[\begin{bmatrix}
y_{11} & y_{12} & y_{13} & y_{14} & y_{15} \\
y_{21} & y_{22} & y_{23} & y_{24} & y_{25} \\
0 & x_{11} & x_{12} & x_{13} & x_{14} \\
0 & x_{21} & x_{22} & x_{23} & x_{24} \\

\end{bmatrix} \]

\end{eg}

\begin{mydef}\label{def:plalg} The \emph{quiver Pl\"ucker algebra} is the algebra $\Pl(Q,\br)$ over $\CC$ generated by all the $k_i \times k_i$ minors of $A_i$, $i=1,\dots,\rho$. 
\end{mydef}

\begin{eg} Continuing Example \ref{eg:coords}, the quiver Pl\"ucker algebra $\Pl(Q,\br)$ is the sub-algebra of $\CC[x_{ij},y_{kl},z_{mn}]$ generated by the size 2 minors of $A_1$, the size $7$ minors of $A_2$, and the size $4$ minors of $A_3.$
\end{eg}

\begin{mydef}\label{def:plrel} Let $\CC[p]:=\CC[p^i_{J}: i=1,\dots,\rho, J \subset \{1,\dots,\tilde{s}_i\}]$, which is naturally identified with the coordinate ring of $\prod_{i=1}^\rho \PP^{\binom{\tilde{s}_i}{k_i}-1}$. We define $\phi$ to be the map
\[\phi: \CC[p] \to \Pl(Q,\br), \hspace{3mm} p^i_J \mapsto m_J(A_i),\]
where $m_J(A_i)$ is the minor of $A_i$ given by taking the columns indexed by $J$. If a polynomial $p$ is an element of $\ker(\phi)$, we say it is a \emph{quiver Pl\"ucker relation}. 
\end{mydef}

\begin{thm} The image of the quiver flag variety $M_\theta(Q,\br)$ under the embedding \eqref{eq:comp} is the subvariety cut out of $\prod_{i=1}^\rho \PP^{\binom{\tilde{s}_i}{k_i}-1}$ by the kernel of $\phi$. 
\end{thm}
\begin{proof}
A subspace represented by matrix $k_i \times \tilde{s}_i$ is precisely characterized (up to a scalar multiple) by all the $k_i \times k_i$ minors of the matrix. Using these minors, we can thus represent a point in $P^\vee$ precisely as a point in $\prod_{i=1}^\rho \PP^{\binom{\tilde{s}_i}{k_i}-1}$. If the point was a point in the image of the quiver flag variety, then we can represent it as a tuple of matrices of the form of the $A_i$. It follows that $\ker(\phi)$ cuts out the closure of the image of the injective map taking a tuple $(V_i)_{i=1}^\rho$ to a point in $\prod_{i=1}^\rho \PP^{\binom{\tilde{s}_i}{k_i}-1}$. However, this map is precisely the composition from \eqref{eq:comp}, so the image is closed. 
\end{proof}

\subsection{Toric degenerations of $Y$-shaped quivers}
There is a class of quivers for which we show minors of the $A_i$ are a SAGBI basis of the quiver Pl\"ucker algebra (it then follows easily for products of such such quiver flag varieties). We call these \emph{$Y$-shaped} quivers, and they are characterized as follows. 
\begin{mydef} Let $Q=(Q_0,Q_1)$ be a quiver, and $\{0,\dots,\rho\}$ a labelling of the vertices such that $n_{ij}=0$ if $i \geq j$. The quiver $Q$ is \emph{$Y$-shaped} if:
\begin{itemize}
\item Vertex $1$ has at most $2$ out-going arrows. 
\item For any vertex $i>1$, there is at most one arrow with $s(a)=i$.
\item  For any vertex $i>0$, there is at most one arrow such that $t(a)=i$ and $s(a) \neq 0.$ 
 \end{itemize}
\end{mydef}
 We can assume that for all $j>1$, there is a path $1 \to j$, as otherwise the associated quiver flag variety is a product of two $Y$-shaped quiver flag varieties.  The quiver in Example \ref{eg:coords} is not a $Y$-shaped quiver because of the double arrow. 
\begin{eg}\label{eg:Yshaped1}The following is an example of a $Y$-shaped quiver:
\begin{center}
  \includegraphics[scale=0.5]{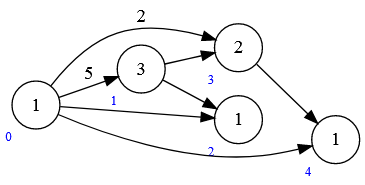}
\end{center}
\end{eg}

By definition, there are at most $2$ arrows out of vertex $1$ in a $Y$-shaped quiver. If there are two, call them $a_1$ and $a_2$, and define a partition $S_1' \sqcup S_2' =\{2,\dots,\rho\}$ by
$i \in S_j'$ if the path from $1 \to i$ contains $a_j$. We write $S_i:=S_i' \cup \{1\}.$ Essentially, we just subdivide the two branches of the quiver. Denote the last vertex of the $S_i$ branch $\rho_i$. 

The ample cone of a quiver flag variety is a cone in $\RR^\rho$, where the $i^{th}$ generator is associated to the line bundle $\det(W_i)$. See \cite{kalashnikov} for more details. 
\begin{lem} The ample cone of a $Y$-shaped quiver is the positive orthant. 
\end{lem}
\begin{proof}
Let $(c_1,\dots,c_\rho)$ be point in the ample cone of a $Y$-shaped quiver flag variety. Suppose $i$ is a vertex in the quiver. If $r_i>1,$ then by \cite[Corollary 5.12]{kalashnikov}, $c_i>0$. By Proposition 5.11 in \cite{kalashnikov}, if $r_i=1$, then $c_i>0$ if for every vertex $j>0$ in the quiver there is a path $0 \to j$ not passing through $i$. If there are no outgoing arrows from $i$, this is automatically satisfied. Suppose there is an outgoing arrow to vertex $k$. Then $\tilde{s}_k=r_i+n_{0k}=1+n_{0k}$ and by assumption $\tilde{s}_k>r_k \geq 1.$ So $n_{0,k}>0.$ So there is a path from $0$ to $k$ not passing through $i$. Now let $j$ be any vertex. There's always some path $0 \to j$. If this path does not go through $i$ we are done, otherwise, it contains an arrow $i \to k$ for some vertex $k$. So this path can be re-routed to avoid $i$ by taking an arrow $0 \to k$. So $c_i>0$. \end{proof}

Let $Q$ be a $Y$-shaped quiver. Let $(A_i)_{i=1}^\rho$ be the $\rho$ matrices with coordinate entries defined by the previous subsection for a $Y$-shaped quiver $Q$. Recall that $A_i$ is a $k_i \times \tilde{s}_i$ matrix, where $k_i:=\tilde{s}_i-r_i.$ For each $i$, there are $(s_i-r_i)(\tilde{s}_i)$ new variables appearing in $A_i$ as the entries of $s_i-r_i$ rows. We use the partition $S_1' \sqcup S_2'$ to define an order on these coordinates: variables introduced in $A_i$ take priority over variables introduced in $A_1$ which take priority over variables introduced in $A_j$, for $i \in S_1'$ and $j \in S_2'$. If $i,j \in S_1, j>i$, variables in $A_i$ take priority over variables in $A_j$. The reverse is true for $i,j \in S_2.$ Otherwise, the order on variables follows the lexicographical order on matrix entries. That is, denote the new variables in $A_i$ as $x^{(i)}_{j,k}, 1 \leq j \leq s_i-r_i, 1 \leq k \leq \tilde{s}_i.$  For a given $i$, 
 \[x^{(i)}_{11}>x^{(i)}_{12} > \cdots >x^{(i)}_{(s_i-r_i)\tilde{s}_i}. \]
 
Our goal is now to show the following theorem. 
\begin{thm}\label{thm:SAGBI}
The $k_i \times k_i$ minors of the $A_i$ for all $i$ form a SAGBI basis under the above defined order for the quiver Pl\"ucker algebra. 
\end{thm}
The argument is very similar to the proof of \cite[Theorem 14.11]{combinatorial}, which is the same statement in the case of flag varieties. We first find a Gr\"obner basis for the ideal of Pl\"ucker relations, i.e. $\ker(\phi)$. 

We can write all the $A_i, i \in S_1$ ($i \in S_2$) as sub-matrices in one matrix $M_1$  ($M_2$) such that the bottom-aligned (top-aligned) minors of sizes $k_i$  of $M_1$ ($M_2$) are the maximal minors of all the $A_i$. Up to a permutation of the rows and columns, $M_j$ is equivalent to $A_{\rho_j}$. Recall the construction of the $A_i$: the set of columns is partitioned so that each subset corresponds to  $a \in Q_1, t(a)=i$. The set of rows is partitioned into sets correspond to $a \in Q_1, t(a)=i, s(a)\neq 0$ and one more set $U$ of size $s_i-r_i$. For $j=1$, and $i \in S_j$, we order the rows so that rows from $U$ come first, and the rest are ordered from greatest to smallest $s(a)$. Columns are ordered similarly. For $j=2$, we use the reverse formation.  Then $M_j=A_{\rho_j}$ with this ordering.  

The matrix $M_j$ is a $k_{\rho_j} \times \tilde{s}_{\rho_j}$ matrix. Recall that $k_{\rho_j}=\tilde{s}_{\rho_j}-r_{\rho_j}.$

\begin{eg} Let $Q$ be the quiver from Example \ref{eg:Yshaped1}. Let $S_1'=\{2\}$ and $S_2'=\{3,4\}$. Then 
\[M_1=\begin{bmatrix}
x^{(2)}_{11} & x^{(2)}_{12} &x^{(2)}_{13} &x^{(2)}_{14} &x^{(2)}_{15} &x^{(2)}_{16}\\
x^{(2)}_{21} & x^{(2)}_{22} &x^{(2)}_{23} &x^{(2)}_{24} &x^{(2)}_{25} &x^{(2)}_{26}\\
x^{(2)}_{31} & x^{(2)}_{32} &x^{(2)}_{33} &x^{(2)}_{34} &x^{(2)}_{35} &x^{(2)}_{36}\\
0 & x^{(1)}_{11} &x^{(1)}_{12} &x^{(1)}_{13} &x^{(1)}_{14} &x^{(1)}_{15}  \\
0&x^{(1)}_{21} &x^{(1)}_{22} &x^{(1)}_{23} &x^{(1)}_{24} &x^{(1)}_{25} \\
\end{bmatrix}.
\]
Similarly, 
\[M_2=
\begin{bmatrix}
x^{(1)}_{11} &x^{(1)}_{12} &x^{(1)}_{13} &x^{(1)}_{14} &x^{(1)}_{15} &0& 0&0\\
x^{(1)}_{21} &x^{(1)}_{22} &x^{(1)}_{23} &x^{(1)}_{24} &x^{(1)}_{25}&0& 0&0 \\
 x^{(3)}_{11} & x^{(3)}_{12} & x^{(3)}_{13} & x^{(3)}_{14}& x^{(3)}_{15}& x^{(3)}_{16}& x^{(3)}_{17} & 0 \\
 x^{(3)}_{21} & x^{(3)}_{22} & x^{(3)}_{23} & x^{(3)}_{24}& x^{(3)}_{25}& x^{(3)}_{26}& x^{(3)}_{27}&0  \\
x^{(3)}_{31} & x^{(3)}_{32} & x^{(3)}_{33} & x^{(3)}_{34}& x^{(3)}_{35}& x^{(3)}_{36}& x^{(3)}_{37}&0  \\
x^{(4)}_{11} &x^{(4)}_{12} &x^{(4)}_{13} &x^{(4)}_{14} &x^{(4)}_{15} &x^{(4)}_{16} &x^{(4)}_{17} &x^{(4)}_{18} \\
x^{(4)}_{21} &x^{(4)}_{22} &x^{(4)}_{23} &x^{(4)}_{24} &x^{(4)}_{25} &x^{(4)}_{26} &x^{(4)}_{27} &x^{(4)}_{28} \\
\end{bmatrix}
.\]
\end{eg}
We index the columns of $M_1$ from $\{1,\dots,\tilde{s}_{\rho_1}\}$, but we shift the index of $M_2$ so that the columns are numbered $\{\tilde{s}_{\rho_1}-\tilde{s}_1+1,\dots,\tilde{s}_{\rho_1}-\tilde{s}_1+\tilde{s}_{\rho_2}\}$. This choice is to ensure that $x^{1}_{ij}$ appears in the same column in both $M_1$ and $M_2$. Use $M_1$ and $M_2$ to relabel the $x^{(i)}_{jk}$ by setting the entries of $M_1$ to be $[m^{(1)}_{ij}]$ and the entries of $M_2$ to be $[m^{(2)}_{ij}]$ where $j$ runs through the column sets described above. Some of the $m_{ij}$ are zero.

The Pl\"ucker algebra is generated by the bottom--aligned minors of $M_1$ of sizes $k_i$, $i \in S_1$ and the upper--aligned minors of $M_2$ of sizes $k_i, i \in S_2$. 
\begin{rem}
From now on, by the minors of $M_1$, we mean \emph{only} the bottom-aligned minors of size $k_i, i \in S_1$, and by the minors of $M_2$, we mean \emph{only} the upper-aligned minors of size $k_i, i \in S_2.$
\end{rem}

The diagonal term of a minor is the monomial given by taking the product of the diagonal entries of the associated sub-matrix.
\begin{lem} The initial term of a non-zero minor of $M_1$ or $M_2$ is the diagonal monomial.
\end{lem}
\begin{proof} It is clear from the term order that the only way the statement of the lemma can fail to be true is if the diagonal monomial of a non-zero minor is $0$. It therefore suffices to show the following claim.
\begin{claim} If the diagonal monomial is zero for a minor of $M_i$, then the minor is zero.
\end{claim}

We'll show the claim for $M_1$. The proof for $M_2$ is similar, and is left to the reader. Consider a $k_i \times k_i$ bottom-aligned minor of $M_1$, given by columns $c_1<\cdots<c_{k_i}$,  $i \in S_1$. If the diagonal monomial vanishes, then for some $t$, $m^{(1)}_{t c_t}=0$. Note that whenever there is a zero in $M_1$, all entries in the column below that zero are also zero. So in fact $m^{(1)}_{v c_t}=0$ for all $v \geq t$. So the first $t$ columns of the minor span a subspace of dimension at most $t-1$, and thus are linearly dependent.  The minor therefore vanishes. 
\end{proof}
\begin{rem} This lemma is not true for general quiver flag varieties. See for example the minor corresponding to the first $7$ columns of $A_2$ in Example \ref{eg:coords}. Although the minor is non-zero, the diagonal monomial vanishes. This can occur when $n_{ij}>1$ for $i>0.$ This is one difficulty in generalizing the discussion below to other types of quiver flag varieties. 
\end{rem}

Consider a monomial $m$ in the $x^{(i)}_{jk}$. We want to write it as $m_1 m_2$, where $m_i$ is a monomial coming from variables in $M_i$. Since the variables $M_1$ and $M_2$  overlap at the $x^{(1)}_{ij}$, we'll put all of these into $M_2$. So $m_1$ contains variables that come from the first $a_1:=k_{\rho_1}-k_1$ rows of $M_1$. Gather the terms in $m_1$ that appear in the $c^{th}$ row, and suppose these appear in columns $i_{c1} \leq \cdots \leq i_{c l_c}$ of $M_1$. That is, we can write $m_1$ as a product of factors of the form:
\[m^{(1)}_{c i_{c1}} m^{(1)}_{c i_{c 2}}\cdots m^{(1)}_{c i_{c l_c}}\]
for some non-negative integer $l_c.$ In other words, as an ordered product, $m_1$ can be written uniquely as
\[m^{(1)}_{1i_{11}} m^{(1)}_{1i_{12}}\cdots m^{(1)}_{1i_{1l_1}} m^{(1)}_{2i_{21}} \cdots m^{(1)}_{2i_{2l_2}} \cdots m^{(1)}_{a_1 i_{a_1 1}} \cdots m^{(1)}_{a_1 i_{a_1 l_{a_1}}}.\]
Note that by assumption, $i_{s t} \leq i_{s t+1}$.

The factor $m_2$ contains variables that come from any of the $a_2:=k_{\rho_2}$ rows of $M_2$. So $m_2$ can be written uniquely as an ordered product as:
 \[m^{(2)}_{1 j_{11}} m^{(2)}_{1j_{12}}\cdots m^{(2)}_{1j_{1 t_1}} m^{(2)}_{2j_{21}} \cdots m^{(2)}_{2j_{2 t_2}} \cdots m^{(2)}_{a_2 j_{a_2 1}} \cdots m^{(2)}_{a_2  j_{a_2 t_{a_2}}}.\]
Here $j_{st} \leq j_{s t+1}$.

We are interested in the case when $m_1 m_2$ is a monomial appearing in an element of the quiver Pl\"ucker algebra. If it is, it is easy to see  that
\begin{itemize}
\item $l_1 \leq l_2 \leq \cdots \leq l_{a_1}.$
\item $t_2 \geq \cdots \geq t_{a_2}.$
\item $t_1 \geq l_{a_1}+t_{k_1+1}.$
\item Suppose rows $p$ and $p+1$ of $M_1$ both have entries from the set of variables (i.e. of the form $x^{(i)}_{jk}$ for fixed $i$). Then $l_p=l_{p+1}$. 
\item Suppose rows $p$ and $p+1$ of $M_2$ both have entries from the set of variables (i.e. of the form $x^{(i)}_{jk}$ for fixed $i$). Then $t_p=t_{p+1}$. 
\end{itemize}

Given such a monomial $m_1 m_2$, we can form a skew tableau filled with entries $i_{pq}$ and $j_{pq}$. The skew tableau has $a_1+a_2$ rows. The first $k_{\rho_1}$ rows are right-aligned, of lengths $l_1,\dots,l_{a_1},t_1,\dots,t_{k_1}$. The last $k_{\rho_2}$ rows are left-aligned, of lengths $t_1, \dots,t_{a_2}.$ The filling on row $p$ given by $i_{p,1},\dots,i_{p,l_p}$ for the first $a_1$ rows and by $j_{p,1},\dots,j_{p,k_p}$ for the rest. The shape of the tableau is illustrated in the following diagram:
\[\begin{tikzpicture}
\draw (0,0)--(0,-4)--(2,-4)--(2,-3)--(3,-3)--(3,-2)--(4,-2)--(4,-1)--(7,-1)--(7,2)--(6,2)--(6,1)--(5,1)--(5,0)--(0,0);
\draw[fill=gray!42] (0,0) -- (0,-1) -- (7,-1)--(7,0)--(0,0);
\end{tikzpicture}\]
The middle gray area is of width $t_1$ and height $k_1$. 
\begin{eg} Continuing the previous example, the monomial 
\[x^{(2)}_{11} x^{(2)}_{22} x^{(2)}_{33} x^{(1)}_{11}(x^{(1)}_{13})^2 x^{(1)}_{22} x^{(1)}_{24} x^{(1)}_{25} x^{(3)}_{13} x^{(3)}_{24} x^{(3)}_{35} x^{(4)}_{16} x^{(4)}_{27}  \]
can be re-written as
\[m^{(1)}_{11} m^{(1)}_{22} m^{(1)}_{33} m^{(2)}_{12} m^{(2)}_{14} m^{(2)}_{14} m^{(2)}_{23} m^{(2)}_{25}m^{(2)}_{26} m^{(2)}_{34}m^{(2)}_{45} m^{(2)}_{56}m^{(2)}_{67}m^{(2)}_{78}.\] 
This corresponds to the filled skew Young tableau
\[\scriptsize \young(::1,::2,::3,244,356,4,5,6,7,8)
\]
\end{eg}

\begin{lem}\label{lem:semis}
The monomial is the initial term of a product of minors of the $M_i$ if and only if the associated filled skew tableau is semi-standard. 
\end{lem}
\begin{proof}
Suppose $m_1 m_2$ is the initial term of a product of minors of the $M_i$. For a minor $k_j \times k_j$ minor of $M_i$ appearing in the product, draw a $k_j$-length column and fill it with the indices of the columns appearing in the minor. We can use these columns to build a filled skew tableau of the shape above. The labels of each column are strictly decreasing. We can then permute the blocks in each row, so that they are weakly increasing. This preserves the property of the columns being strictly decreasing, and so the resulting tableau is semi-standard. It isn't hard to see that this filled tableau is exactly the filled tableau required. 

For the converse, suppose the skew tableau is semi-standard. Each column of the tableau gives a list of columns in either $M_i$ of size $k_j$ for some $j \in S_i.$ The initial term of the product of minors given by each column of the tableau is exactly $m$. 
\end{proof}
\begin{rem} Lemma \ref{lem:semis} is an essential step in the proof that the minors of $A_i$ are a SAGBI basis. This is why the proof works only for $Y$-shaped quiver flag varieties. The fact that there are only two legs in the $Y$-shaped quiver means that we can fit $M_1$  and $M_2$ into a single matrix, and all minors are either top or bottom aligned. This in turn means that we can index initial terms of these minors in a single skew-tableau which looks like the union of a tableau and the second transposed tableau. Any additional `legs' makes this impossible.
\end{rem}

For $Y$-shaped quiver flag varieties, we change our notation for the quiver Pl\"ucker algebra slightly. We relabel the generators of $\CC[p]$ to be $\CC[p^1_\sigma,p^2_\sigma]$ where $p^1_\sigma$ runs over strictly increasing sequences $\sigma$ with entries from $\{1,\dots,\tilde{s}_{\rho_1}\}$ of sizes $k_i, i \in S_1'$ and $p^2_\sigma$ runs over strictly increasing sequences $\sigma$ with entries from $$\{\tilde{s}_{\rho_1}-\tilde{s}_1+1,\dots,\tilde{s}_{\rho_1}-\tilde{s}_1+\tilde{s}_{\rho_2}\}$$ of sizes $k_i, i \in S_2$. Moreover, we only consider $\sigma$ such that $p^i_\sigma$ does not define an identically zero minor on $M_i$.

In this notation, the map $\phi:\CC[p] \to \Pl(Q,\br)$ takes $p^1_\sigma$ to the minor of $M_1$ indexed by columns $\sigma$ and rows $\{\tilde{s}_{\rho_1}-|\sigma|+1, \dots,\tilde{s}_{\rho_1}\}$, and $p^2_\sigma$ to the minor of $M_2$ indexed by columns $\sigma$ and rows $\{1, \dots,|\sigma|\}.$

We now define a partial order, $\leq$, on the variables of $\CC[p]$. 
\begin{mydef} Let $\sigma=\{\sigma_1 < \sigma_2 < \cdots < \sigma_k\}$ and $\tau=\{\tau_1 < \cdots < \tau_l\}.$ Let $\sigma'$ be the partition formed by taking the last $k_1$ entries of $\sigma$, and $\sigma''$ the partition formed by taking the first $k_1$ entries.  Then $p^i_\sigma \leq p^j_\tau$ if one of the following cases holds:
\begin{itemize}
\item $i=j=2$ and $|\sigma| \geq |\tau|$ and $\sigma_s \leq \tau_s$ for all $s \leq l,$

\item $i=j=1$ and $|\sigma| \leq |\tau|$ and $\sigma_s \leq \tau_s$ for all $s \leq k,$

\item $i=2$ and $j=1$ and $\sigma'_s \leq \tau''_s$ for all $s$.
\end{itemize}  
\end{mydef}
We can extend this to a total order, $\prec$, where we say that $p^i_\sigma \prec p^j_\tau$ if one of the following cases holds:
\begin{itemize}
\item $i=j=2$ and $|\sigma| \geq |\tau|$ and if $|\sigma|=|\tau|$ then $\sigma$ comes before $\tau$ in the lexicographic order,

\item $i=j=1$ and $|\sigma| \leq |\tau|$ and if $|\sigma|=|\tau|$ then $\sigma$ comes before $\tau$ in the lexicographic order.
\item $i=2$ and $j=1$.
\end{itemize} 
Extend $\prec$ to an ordering on the monomial of $\CC[p]$ using the reverse lexicographic ordering. 

Suppose we are given a monomial in $\CC[x^{(i)}_{jk}]$ which is the initial term of a polynomial in the minors of the $A_i$. Consider the columns of the associated semi-standard skew Young tableau. The set of decorations of a column is either given by $\sigma=\{i_{pq}\}_{p}$ or $\tau=\{j_{pq}\}_{p}$, which define $p^1_\sigma$ or $p^2_\tau$ respectively. It is clear semi-standard skew Young tableau of these shapes correspond to monomials in $\CC[p]$ supported on chains (under the partial order $\leq$).  Call such monomials \emph{semi-standard}. 

\begin{prop} Semi-standard monomials are a basis for $\Pl(Q,\br) \cong \CC[p]/\ker(\phi)$ as a $\CC$-vector space.
\end{prop}
\begin{proof}
We show the statement that a monomial is not in the initial ideal of $\ker(\phi)$ under $\prec$ if and only if it is semi-standard. This implies the proposition.

First, we show that for any incomparable pair $p^i_\sigma,p^j_\tau$, the product $p^i_\sigma p^j_\tau$ is the initial term  (under $\prec$) of some polynomial in the kernel of $\phi$. If $i=j$, then  the proof is identical (or symmetrical, if $i=1$) to that in \cite[Theorem 14.6]{combinatorial}. The zero matrix entries do not pose a problem as any relation which holds in the larger ring generated by minors of the matrix without zero entries certainly holds in the rings we are considering.  

Suppose $i=1$ and $j=2$. Since $p^1_\sigma$ and $p^2_\tau$ are incomparable, there exists $k$ minimal such that $\sigma'_k<\tau''_k$.  If $l$ is the index of $\sigma'_k$ in $\sigma$, and $m$ the index of $\tau''_k$ is $\tau$, consider
\begin{equation}\label{permutation}
 \sigma_1 < \cdots < \sigma_l < \tau_m < \cdots < \tau_{|\tau|}.
 \end{equation}
This sequence has length $p=|\sigma|+|\tau|-b+1$. As in the proof of Theorem 14.6 of \cite{combinatorial}, for any permutation $\pi$ of $\{1,\dots,p\}$, define new partitions $\pi(\sigma)$ and $\pi(\tau)$ such that elements not from $\{\sigma_1,\dots,\sigma_l\}$ or $\{\tau_m, \dots, \tau_{|\tau|}\}$ are left untouched, and these are permuted via $\pi$. If $p^i_\alpha$ is not in $\CC[p]$, we set it to zero. We claim that the sum
$$f=\sum_{\pi} \text{sign}(\pi) p^1_{\pi(\sigma)} p^2_{pi(\tau)}$$
 is in the kernel of $\phi$. To see this, note that $f$ also defines a function on the set of  $p-1 \times\tilde{s}_{\rho_1}+\tilde{s}_{\rho_2}-\tilde{s}_1$ matrices. This is done by associating to $p^1_\alpha$ the top-aligned minor with columns chosen by $\alpha$; similarly, $p^2_\alpha$ 
is associated with the bottom-aligned minor. Then $f$ is multi-linear and alternating on the $p$ columns indexed by \eqref{permutation}. It must therefore vanish as the space of columns is at most $p-1$ dimensional. Let $M_1'$ be the sub-matrix of $M_1$ given by taking the bottom $|\sigma|$ rows, and $M_2'$ the sub-matrix of $M_2$ given by taking the top $|\tau|$ rows. Consider a matrix of size  $p-1 \times\tilde{s}_{\rho_1}+\tilde{s}_{\rho_2}-\tilde{s}_1$ with $M_1$ as a top right minor, and $M_2$ as a bottom left minor, and all other entries zero. $M_1$ and $M_2$ will overlap, but they agree on the overlap by construction. Then $\phi(f)$ agrees with $f$ thought of as a function on this large matrix, and so $\phi(f)$ vanishes. 

Note that $p^1_\sigma p^2_\tau$ is the initial term of $f$, as for all $\pi$,
\[p^2_{\pi(\tau)} \prec p^2_{\tau} \prec p^1_\sigma \prec p^1_{\pi(\sigma)}. \]
We have now shown that any non-semi-standard monomial is in $in(\ker(\phi))$. The converse -- that no semi-standard monomial is in the initial ideal of $\ker(\phi)$ -- can be argued precisely in the same way as in the second part of the proof of \cite[Theorem 14.6]{combinatorial}, as we have shown that each initial monomial has a unique semi-standard polynomial with this initial monomial.

\end{proof}
We can now complete the proof of Theorem \ref{thm:SAGBI}.
\begin{proof}
We have shown that a monomial $m$ in $\CC[x^{(k)}_{ij}]$ is the initial term of $\phi(f)$, where $f$ is a monomial in $\CC[p]$, if and and only if it defines a semi-standard skew Young tableau of an appropriate shape. Now suppose $m$ is the initial term of some polynomial in $im(\phi)$. By the above proposition, we can write the polynomial as $im(g)$ where $g$ is a sum of semi-standard monomials. Therefore $m$ must again define a semi-standard skew Young tableau of the appropriate shape. 

This implies that the initial algebra is the vector space spanned by all monomials given by semi-standard skew Young tableaux. Since every such monomial is the product of the monomials given by each of the columns in the tableau, and the monomial associated to a column is precisely the initial term of a minor of some $A_i$, this concludes the proof.\end{proof}

In the next section, we show that the degenerate toric variety associated to this SAGBI basis is a singular toric quiver variety and describe the quiver. 
\section{Ladder diagrams for certain degenerations}\label{sec:ladder}
In \cite{flagdegenerations}, the authors Batyrev, Ciocan-Fontanine, Kim, and van Straten  use \emph{ladder diagrams} to give a concrete description of the toric variety to which the flag variety degenerates. In this section, we give a new description of the degenerate toric variety by considering the ladder diagram as a quiver. We then generalize this construction to the degenerations of the $Y$-shaped quiver described in the previous section. 

For a general definition of a ladder diagram of a flag variety, see Definition 2.1.1 in \cite{flagdegenerations}. It can also be described as follows: the ladder diagram of $\Gr(n,r)$ is an $r \times (n-r)$ grid of unit squares such that the bottom left corner is at $(0,0)$. Let $O$ denote this vertex. For example, the ladder diagram of $\Gr(5,2)$ is
\[\begin{tikzpicture}[scale=0.6]
\draw (0,0) rectangle (1,1);
\draw (1,1) rectangle (2,2);
\draw (0,1) rectangle (1,2);
\draw (1,0) rectangle (2,1);
\draw (2,1) rectangle (3,2);
\draw (2,0) rectangle (3,1);
\draw[fill] (0,0) circle (3pt);
\end{tikzpicture}\]
where $O$ is marked. The ladder diagram of $\Fl(n,r_1,\dots,r_\rho)$ is the union of the ladder diagrams of $\Gr(n,r_i)$ for all $i$: for example, the ladder diagram of $\Fl(5,3,2,1)$ is the union of
\[\begin{tikzpicture}[scale=0.4]
\draw (0,0) rectangle (1,1);
\draw (1,1) rectangle (2,2);
\draw (0,1) rectangle (1,2);
\draw (1,0) rectangle (2,1);
\draw (0,2) rectangle (1,3);
\draw (1,2) rectangle (2,3);

\draw[fill] (0,0) circle (3pt);
\draw[line width=0.6mm] (0,0)--(2,0)--(2,3)--(0,3)--(0,0);

\end{tikzpicture} \hspace{2mm} \cup \hspace{2mm}
\begin{tikzpicture}[scale=0.4]
\draw (0,0) rectangle (1,1);
\draw (1,1) rectangle (2,2);
\draw (0,1) rectangle (1,2);
\draw (1,0) rectangle (2,1);
\draw (2,1) rectangle (3,2);
\draw (2,0) rectangle (3,1);
\draw[fill] (0,0) circle (3pt);
\draw[line width=0.6mm] (0,0)--(3,0)--(3,2)--(0,2)--(0,0);
\end{tikzpicture}  \hspace{2mm} \cup \hspace{2mm}
\begin{tikzpicture}[scale=0.4]
\draw (0,0) rectangle (1,1);

\draw (1,0) rectangle (2,1);

\draw (2,0) rectangle (3,1);
\draw (3,0) rectangle (4,1);
\draw[fill] (0,0) circle (3pt);

\draw[line width=0.6mm] (0,0)--(4,0)--(4,1)--(0,1)--(0,0);
\end{tikzpicture}  \hspace{2mm} \rightarrow \hspace{2mm}
\begin{tikzpicture}[scale=0.4]
\draw (0,0) rectangle (1,1);
\draw (1,1) rectangle (2,2);
\draw (0,1) rectangle (1,2);
\draw (1,0) rectangle (2,1);
\draw (2,1) rectangle (3,2);
\draw (0,2) rectangle (1,3);
\draw (1,2) rectangle (2,3);
\draw (2,0) rectangle (3,1);
\draw (3,0) rectangle (4,1);
\draw[fill] (0,0) circle (3pt);
\draw[line width=0.6mm] (0,0)--(2,0)--(2,3)--(0,3)--(0,0);
\draw[line width=0.6mm] (0,0)--(3,0)--(3,2)--(0,2)--(0,0);
\draw[line width=0.6mm] (0,0)--(4,0)--(4,1)--(0,1)--(0,0);
\end{tikzpicture}.\]
Removing the outlines of each component, the we see that the ladder diagram of the flag variety is 
\[\begin{tikzpicture}[scale=0.6]
\draw (0,0) rectangle (1,1);
\draw (1,1) rectangle (2,2);
\draw (0,1) rectangle (1,2);
\draw (1,0) rectangle (2,1);
\draw (2,1) rectangle (3,2);
\draw (0,2) rectangle (1,3);
\draw (1,2) rectangle (2,3);
\draw (2,0) rectangle (3,1);
\draw (3,0) rectangle (4,1);
\draw[fill] (0,0) circle (3pt);
\end{tikzpicture}.\]
The authors in \cite{flagdegenerations} associate to the ladder diagram another graph, and then describe the polytope of the degeneration of the flag variety given above by paths in this graph. Instead, we associate to the ladder diagram a quiver. 

The first step is to add more vertices to the ladder diagram. For $\Gr(n,r)$, add vertices at $(i,j)$ for $1 \leq j < r, 1 \leq i <n-r$ and at $(n-r,r)$. So for $\Gr(5,2)$, the new diagram is\[\begin{tikzpicture}[scale=0.6]
\draw (0,0) rectangle (1,1);
\draw (1,1) rectangle (2,2);
\draw (0,1) rectangle (1,2);
\draw (1,0) rectangle (2,1);
\draw (2,1) rectangle (3,2);
\draw (2,0) rectangle (3,1);
\draw[fill] (0,0) circle (3pt);
\draw[fill] (1,1) circle (3pt);
\draw[fill] (2,1) circle (3pt);
\draw[fill] (3,2) circle (3pt);
\end{tikzpicture}\]
For a flag variety, the new diagram is again the union of the diagrams for each $\Gr(n,r_i)$, with an extra vertex at $(n_{i-1}-r_{i-1},r_{i})$ for each $i>1.$ So the ladder diagram for $\Fl(5,3,2,1)$ is
\[\begin{tikzpicture}[scale=0.6]
\draw (0,0) rectangle (1,1);
\draw (1,1) rectangle (2,2);
\draw (0,1) rectangle (1,2);
\draw (1,0) rectangle (2,1);
\draw (2,1) rectangle (3,2);
\draw (0,2) rectangle (1,3);
\draw (1,2) rectangle (2,3);
\draw (2,0) rectangle (3,1);
\draw (3,0) rectangle (4,1);
\draw[fill] (0,0) circle (3pt);
\draw[fill] (1,1) circle (3pt);
\draw[fill] (2,1) circle (3pt);
\draw[fill] (3,2) circle (3pt);
\draw[fill] (3,1) circle (3pt);
\draw[fill] (4,1) circle (3pt);
\draw[fill] (2,2) circle (3pt);
\draw[fill] (1,2) circle (3pt);
\draw[fill] (2,3) circle (3pt);
\end{tikzpicture}.\]

To make this a quiver, we consider paths between vertices where one is allowed to travel up and to the right only.  Then this defines a quiver where the vertices are the vertices in the ladder diagram, and the number of arrows between two vertices is the number of paths in the diagram between them. Call this quiver the \emph{ladder quiver} $L(n,r_1,\dots,r_\rho)$. Define a dimension vector by setting all vertices to have dimension 1. 

For example, the quiver associated to $\Gr(5,2)$ is
\begin{center}
\includegraphics[scale=0.5]{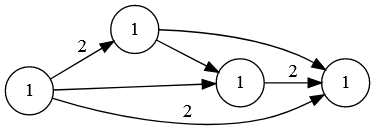}
\end{center}

Just as in the construction of quiver flag varieties, this quiver and dimension vector determines a vector space $V$ and a torus $T$ acting on $V$. The torus $T$ acts with weights $D_a$, where $a$ ranges over the arrows in the ladder quiver. 

Call the vertices at $(r_i,n-r_i)$ the \emph{external vertices}. Let $w$ be the character of $T$ such that the component at a vertex $v$ is $1$ if $v$ is an external vertex, and $0$ otherwise. 
Instead of taking the GIT quotient with stability condition $(1,\dots,1)$, we instead define the toric variety $X(n,r_1,\dots,r_\rho)$ to be the GIT quotient with stability condition $w$.  

The path from the source to an external vertex $(r_i,n-r_i)$ corresponds to a Weil divisor which is in fact Cartier. Call the associated line bundle $L_i$. These line bundles are globally generated and define an embedding into a product of projective spaces which is also the target of the Pl\"ucker embedding of the flag variety. 

We now define ladder quivers for $Y$-shaped quivers. First, suppose $Q$ is a $Y$-shaped quiver such that there is only one arrow out of vertex 1. An example of such a quiver is
 \begin{center}
  \includegraphics[scale=0.5]{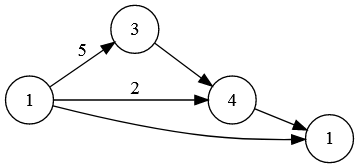}
\end{center}
Build the ladder diagram almost exactly as for the flag case. First, take the union of the ladder diagrams of $\Gr(\tilde{s}_i,r_i)$ for $i=1,\dots,\rho$. This is the \emph{extended} ladder diagram. In this example, it is

\[\begin{tikzpicture}[scale=0.6]
\draw (0,0) rectangle (1,1);
\draw (1,1) rectangle (2,2);
\draw (0,1) rectangle (1,2);
\draw (1,0) rectangle (2,1);
\draw (2,1) rectangle (3,2);
\draw (2,0) rectangle (3,1);
\draw (3,0) rectangle (4,1);
\draw (4,0) rectangle (5,1);
\draw (5,0) rectangle (6,1);
\draw (6,0) rectangle (7,1);
\draw (0,2) rectangle (1,3);
\draw (1,2) rectangle (2,3);
\draw (2,2) rectangle (3,3);
\draw(0,3) rectangle (1,4);
\draw (1,3) rectangle (2,4);

\draw (2,3) rectangle (3,4);
\draw[line width=0.6mm] (0,0)--(2,0)--(2,3)--(0,3)--(0,0);
\draw[line width=0.6mm] (0,0)--(3,0)--(3,4)--(0,4)--(0,0);
\draw[line width=0.6mm] (0,0)--(7,0)--(7,1)--(0,1)--(0,0);

\end{tikzpicture}.\]
As in the flag case, we've outlined the ladder diagram of each $\Gr(\tilde{s}_i,r_i)$. Unlike in the flag case, the sequence $(r_1,\dots,r_\rho)$ is not necessarily decreasing, and the diagram of $\Gr(\tilde{s}_i,r_i)$ may be taller than the diagram of $\Gr(\tilde{s}_{i-1},r_{i-1})$. To obtain the ladder diagram of the full quiver, we truncate this extended ladder diagram, and insists that the maximum height of the diagram between $x=\tilde{s}_{i-1}-r_{i-1}$ and $x=\tilde{s}_i-r_i$ is at most $r_i$. We add vertices as in the flag case: at interior points and at the intersection points $(\tilde{s}_{i-1}-r_{i-1},r_{i})$ for each $i>1.$
The ladder diagram of the above example is then:
\[\begin{tikzpicture}[scale=0.6]
\draw (0,0) rectangle (1,1);
\draw (1,1) rectangle (2,2);
\draw (0,1) rectangle (1,2);
\draw (1,0) rectangle (2,1);
\draw (2,1) rectangle (3,2);
\draw (2,0) rectangle (3,1);
\draw (3,0) rectangle (4,1);
\draw (4,0) rectangle (5,1);
\draw (5,0) rectangle (6,1);
\draw (6,0) rectangle (7,1);
\draw (0,2) rectangle (1,3);
\draw (1,2) rectangle (2,3);
\draw (2,2) rectangle (3,3);

\draw (2,3) rectangle (3,4);
\draw[fill] (0,0) circle (3pt);
\draw[fill] (1,1) circle (3pt);
\draw[fill] (2,1) circle (3pt);
\draw[fill] (3,1) circle (3pt);
\draw[fill] (7,1) circle (3pt);
\draw[fill] (1,2) circle (3pt);
\draw[fill] (2,2) circle (3pt);
\draw[fill] (2,3) circle (3pt);
\draw[fill] (3,4) circle (3pt);
\end{tikzpicture}.\]

We can now describe the proposed ladder diagram of a general $Y$-shaped quiver $Q$. Assume that there are exactly two arrows out of vertex $1$.  Recall that we have partitioned the vertices $\{2,\dots,\rho\}=S'_1 \sqcup S'_2$ according to which of the two branches of the quiver the vertex is on. Consider subquiver of $Q$ with vertices $S_2 \cup \{0\}$: this is a $Y$-shaped quiver for which we know how to build a ladder diagram.  Take this ladder diagram, and reflect it across the $y=-x$ axis, and then translate it so that what was the origin is at $(\tilde{s}_1-r_1,r_1)$. The ladder diagram of $Q$ is the union of this ladder diagram with the ladder diagram of the second subquiver with vertices $S_1 \cup \{0\}$. 
\begin{eg}\label{eg:Yshaped2} We draw the ladder diagram for
\begin{center}
\includegraphics[scale=0.5]{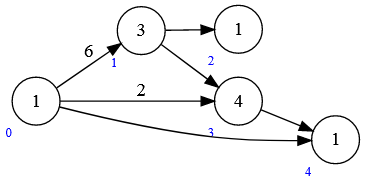}
\end{center}
Set $S_1:=\{1,3,4\}, S_2:=\{1,2\}.$ The reflected diagram of the quiver with vertices $\{0,1,2\}$ and the ladder diagram of the quiver with vertices $\{0,1,3,4\}$ are pictured below:
\[\begin{tikzpicture}[scale=0.6]
\draw (0,0) rectangle (1,1);
\draw (1,0) rectangle (2,1);
\draw (2,0) rectangle (3,1);
\draw (0,1) rectangle (1,2);
\draw (1,1) rectangle (2,2);
\draw (2,1) rectangle (3,2);
\draw (0,2) rectangle (1,3);
\draw (1,2) rectangle (2,3);
\draw (2,2) rectangle (3,3);
\draw (-2,2) rectangle (-1,3);
\draw (-1,2) rectangle (0,3);

\draw[fill] (0,0) circle (3pt);
\draw[fill] (0,2) circle (3pt);
\draw[fill] (1,1) circle (3pt);
\draw[fill] (1,2) circle (3pt);
\draw[fill] (2,1) circle (3pt);
\draw[fill] (2,2) circle (3pt);
\draw[fill] (-2,2) circle (3pt);
\draw[fill,green] (3,3) circle (3pt);

\end{tikzpicture} \hspace{5mm}
\begin{tikzpicture}[scale=0.6]
\draw (-1,0) rectangle (0,1);
\draw (-1,1) rectangle (0,2);
\draw (-1,2) rectangle (0,3);
\draw (0,0) rectangle (1,1);
\draw (1,1) rectangle (2,2);
\draw (0,1) rectangle (1,2);
\draw (1,0) rectangle (2,1);
\draw (2,1) rectangle (3,2);
\draw (2,0) rectangle (3,1);
\draw (3,0) rectangle (4,1);
\draw (4,0) rectangle (5,1);
\draw (5,0) rectangle (6,1);
\draw (6,0) rectangle (7,1);
\draw (0,2) rectangle (1,3);
\draw (1,2) rectangle (2,3);
\draw (2,2) rectangle (3,3);

\draw (2,3) rectangle (3,4);
\draw[fill,blue] (-1,0) circle (3pt);
\draw[fill] (0,1) circle (3pt);
\draw[fill] (0,2) circle (3pt);
\draw[fill] (1,1) circle (3pt);
\draw[fill] (2,1) circle (3pt);
\draw[fill] (3,1) circle (3pt);
\draw[fill] (7,1) circle (3pt);
\draw[fill] (1,2) circle (3pt);
\draw[fill] (2,2) circle (3pt);
\draw[fill] (2,3) circle (3pt);
\draw[fill] (3,4) circle (3pt);

\end{tikzpicture}.
\]
The ladder diagram for the entire quiver is
\[\begin{tikzpicture}[scale=0.6]
\draw (-1,0) rectangle (0,1);
\draw (-1,1) rectangle (0,2);
\draw (-1,2) rectangle (0,3);
\draw (0,0) rectangle (1,1);
\draw (1,1) rectangle (2,2);
\draw (0,1) rectangle (1,2);
\draw (1,0) rectangle (2,1);
\draw (2,1) rectangle (3,2);
\draw (2,0) rectangle (3,1);
\draw (3,0) rectangle (4,1);
\draw (4,0) rectangle (5,1);
\draw (5,0) rectangle (6,1);
\draw (6,0) rectangle (7,1);
\draw (0,2) rectangle (1,3);
\draw (1,2) rectangle (2,3);
\draw (2,2) rectangle (3,3);
\draw (2,3) rectangle (3,4);
\draw (-2,2) rectangle (-1,3);
\draw (-3,2) rectangle (-2,3);

\draw[fill,blue] (-1,0) circle (3pt);
\draw[fill,green] (2,3) circle (3pt);
\draw[fill] (-1,2) circle (3pt);
\draw[fill] (-3,2) circle (3pt);
\draw[fill] (0,1) circle (3pt);
\draw[fill] (0,2) circle (3pt);
\draw[fill] (1,1) circle (3pt);
\draw[fill] (2,1) circle (3pt);
\draw[fill] (3,1) circle (3pt);
\draw[fill] (7,1) circle (3pt);
\draw[fill] (1,2) circle (3pt);
\draw[fill] (2,2) circle (3pt);
\draw[fill] (3,4) circle (3pt);
\end{tikzpicture}.\]
\end{eg}
As in the flag case, we can define a toric variety using the ladder diagram of a $Y$-shaped quiver, by interpreting the ladder diagram as a quiver. Just as in the flag case, we consider paths in the ladder diagram that travel to the right and up. The resulting quiver is called the ladder quiver, $L(Q)$. The vertex at $(0,0)$ is taken to be the source (in that we quotient by the $\CC^*$ to obtain an effective action, as for quiver flag varieties). 

Denote the set of the arrows of $L(Q)$ as $L(Q)_1$, and the set of vertices as $L(Q)_0$. The divisor data for $L(Q)$ is given as follows. Label the vertices $x_0,\dots,x_l$ for some $l$, where $x_0$ is the vertex at $(0,0)$. Let $K=(\CC^*)^l$ be the torus. Let $e_{x_i}, i=1,\dots, l$ be the standard basis of the character lattice $\LL^\vee$ of $K$. Set $e_{x_0}=0$. For each arrow $a \in L(Q)_1$, denote by $D_a \in \LL^\vee$ the weight $D_a=-e_{s(a)}+e_{t(a)}.$ The weights are given by the $D_a$ for all arrows $a$ in $L(Q)$.

Inherited from the ladder diagrams of the two legs of the $Y$-shaped quiver is the notion of \emph{external vertices}, and source vertices (which are the sources in the ladder quiver of each leg). There is an external vertex $v_i \in L(Q)_0$ for each $i \in Q_1-\{0\}$. Denote the source vertices from the two legs of the quiver as $O_1$ and $O_2$ (the blue and green dots above). Note that $O_2$ (in green) is not a source vertex in $L(Q)$ -- it is the original source vertex from the ladder quiver of the second leg. Then $O_1$ is $x_0$.  The vertex $O_2$ is also external vertex associated to $1 \in Q_1$; that is $v_1$ and $O_2$ are the same vertex.  

 \begin{mydef}\label{def:wQ}
 The \emph{ladder stability condition} is 
 \[w_Q:=\sum_{i \in S_1}(-e_{O_1}+ e_{v_i})-\sum_{i \in S_2}(-e_{O_2}+ e_{v_i}).\]
  The change in signs between $S_1$ and $S_2$ takes into account that ladder diagram for the $S_2$ leg has been reflected, but arrows still move up and to the right. 
 \end{mydef}
 The associated toric variety is denoted $X(Q)$.
\begin{lem} There are $\dim(M_\theta(Q,\br))$ boxes in the ladder diagram of $Y$-shaped quiver flag variety.
\end{lem}
\begin{proof}
The quiver flag variety is a tower of Grassmannian bundles of the form $\Gr(s_i,r_i)$, so the lemma is stating that there are $\sum_{i=1}^\rho r_i(s_i-r_i)$ boxes in the ladder diagram. We first assume that $Q$ is a $Y$-shaped quiver with just one leg. When building the ladder diagram, for each vertex $i$, we add an 
\[r_i \times (k_{i}-k_{i-1})\]
grid of boxes,
where $k_i=\tilde{s}_i-r_i.$ In this case,
\[\tilde{s}_i=\tilde{s}_{i-1}+n_{0i}, \hspace{5mm} s_i=r_{i-1}+n_{0i}.\]
So 
\[k_{i}-k_{i-1}=n_{0i}+r_{i-1}-r_{i}=s_i-r_i.\]
So the lemma holds in the special case of a one-legged $Y$-shaped quiver. By comparing the dimensions of each leg of general $Y$-shaped quiver and the associated ladder diagram, this proves the lemma. 
\end{proof}
\begin{cor} The dimension of $X(Q)$ is the number of boxes in $L(Q)$.
\end{cor}
\begin{proof} 
The dimension of $X(Q)$ is 
\[|L(Q)_1|-|L(Q)_0|+1.\]
We'll show that this is equal to the number of boxes in the ladder quiver. For each vertex $v$ that is not $O_1$ in $L(Q)_0$, consider (if it exists) the box $B_v$ that whose upper right corner $v$. Consider the strip of boxes to the left and below $B_v$, and in the same row or column of $B_v$ that do not have a vertex on their upper right corner. That is, go along to the left and stop right before the first box that has a vertex in the upper right corner, then go back to $B_v$ and do the same thing going downwards. Let $c_v$ be the number of boxes in this strip. It's easy to see that the number of arrows into $v$ is exactly $c_v+1$, and that each box appears in exactly one such strip. This proves that 
\[|L(Q)_1|-|L(Q)_0|+1\]
is equal to the number of boxes in $L(Q)$. 
\end{proof}
\begin{cor} The dimension of $X(Q)$ equals the dimension of $\dim(M_\theta(Q,\br))$.
\end{cor}

We can associate weights $D_i \in \LL^\vee$ for each vertex $i \in Q_0-\{0\}$. For $i \in S_j$, $D_i$ is the weight associated to a path in $L_Q$ between $O_j$ and the external vertex $v_i$. Note that if $i \in S_1$, the path goes from $O_1$ to $v_i$, whereas if $i \in S_2'$, the path goes from $v_i$ to $O_2$. In particular, the ladder stability condition $w_Q$ corresponds to the weight
\begin{equation}\label{eq:ladderstability}
\sum_{i=1}^\rho D_i.\end{equation}

Let $S_Q$ be the homogeneous coordinate ring of $L(Q)$, so that 
\[S_Q:=\CC[x_a: a \in L(Q)_1].\]
This is a multi-graded algebra. It is graded by $\LL^\vee$: the grading of $x_a$ is $D_a$. For a path $p$ in the ladder quiver, let $m_p:=\prod_{a \in p} x_a \in S_Q.$ By Proposition 5.3.7 in \cite{CLS}, there is an isomorphism between $\Gamma(X(Q),\mathcal{O}(D_\alpha))$ and the $D_\alpha$ graded piece of $\LL^\vee$. 

\begin{lem} Suppose $m \in S_Q$ is a monomial in $S_Q$, with degree $D_i$. Then $m=m_p$ for a path between $O_j$ and $v_i$, where $i \in S_j.$
\end{lem}
\begin{proof}
Suppose that $i \in S_1$. Then since $m$ is $D_i$ graded, $x_{a_1}$ must divide $m$ for some $a_1 \in L(Q_1)$ with target $v_i$. By assumption, the component of the grading of $m$ at $s(a_1)$ is $0$. So either $s(a_1)=O_1$, or we can find another arrow $a_2$ with $t(a_2)=s(a_1)$ and $x_{a_1} x_{a_2}$ divides $m$. We can keep arguing in this way until we construct a path $p$ between $O_1$ and $v_i$. If $i \in S_2$, the same proof holds. 
\end{proof}

\begin{thm}\label{thm:generators} The sub-algebra of $S_Q$ 
\[ \CC[m: \deg(m) = \sum_{i=1}^\rho c_i D_i, c_i \geq 0],\]
is generated by monomials of degree $D_i$ for  $1 \leq i \leq \rho$.  
\end{thm}
\begin{proof}
To prove this, we temporarily define another, non-negative grading of $S_Q$. A similar grading plays a role in \cite{flagdegenerations}. Let $w$ be the grading defined by setting $w(x_a)$ to be the length of the arrow $a$ in the ladder quiver (recalling that it is a grid of $1 \times 1$ boxes). Then  for a monomial $m$ of degree $D_i$, $w(m)=\tilde{s}_i.$ Next, we show an intermediate claim:
\begin{claim} Suppose that $m$ is a monomial of degree 
\[(-\sum_{j \in S_2} c_j e_{v_j})+(\sum_{j \in S_1} c_j e_{v_j})\]
 for some $c_j \geq 0$. Then $m$ is a product $m=\prod_{p \in K} m_p,$ where each path in $p$ is a path from a vertex in  $\{O_1\} \cup \{v_j: j \in S_2'\}$ to a vertex in $\{O_1\} \cup \{v_j: j \in S_1\}$.
\end{claim}
\begin{proof}[Proof of claim]
Let $A>0$ be the minimum $w$-grading of $m_p$ where $p$ is a non-trivial path of the form in the claim.

Case 1: Suppose that $c_i=0$ for all $i \in S_1$. If there exist $c_j>0, j \in S_2$, then there is some arrow $a_0$ with source $v_j$ such that $x_{a_0}$ appears in $m$.  Since the usual grading of $m$ has no positive components, either $t(a_0)=O_1$ or there is another arrow $a_1$ such that $x_{a_0} x_{a_1}$ divides $m$ and $t(a_0)=s(a_1)$. We can continue in this way until we arrive at $O_1$ (if no such path exists then we get a contradiction). So there is monomial $m_p$ corresponding to this path such that $m_p$ divides $m$.

Case 2: Suppose that there is an $i \in S_1$ such that $c_i>0$. Then there is some arrow $a_0$ with target $v_i$ such that $x_{a_0}$ appears in $m$. Since the usual grading of $m$ is negative only at the components corresponding to external vertices from $S_2$, either $s(a)$ is such a vertex, or it is $O_1$, or we can find another arrow $a_1$ such that $t(a_i)=s(a_0)$. We can continue in this way until we arrive at an external vertex $v_j$ for $j \in S'_2$ such that $c_j>0$ or $O_1$.  In any case, we have constructed a path  $p$ from $O_1$ or an external vertex $v_j$  to $v_i$ and a monomial $m_p$ such that $m_p |m$. 

In both cases, $m_p$ is a factor of the form required in the lemma, and $w(m) \geq w(m_p)\geq A.$ If $w(m)=w(m_p)$, then $m=m_p$ and we are done. Otherwise, let $m'$ be the monomial such that $m_p m'=m.$ It is clear that $m'$ is again a monomial of the form given in the claim, so the same argument shows that $w(m')\geq A$. Since $w(m')<w(m)$, this process will eventually terminate. 
\end{proof}

Now suppose $m$ is a monomial of degree $\sum_{i=1}^\rho c_i D_i$. This line bundle is of the form in the claim, so we can write $m$ as a product of $m_p$ where the $p$ are paths between vertices from $S_2 \cup \{O_1\}$ and $S_1$. If $c_1>0$, there is necessarily at least one $m_p$ with target $1$ and source either $O_1$ or $v_j, j \in S_2$. If the first case, $m_p$ has degree $D_1$ and in the second case $c_j>0$ and $m_p$ has degree $D_j$.  If $c_1=0$, then by considering the definition of the $D_i$, we see that $c_i=0$ for all $i \in S_2$. In particular, any $m_p$ appearing in $m$ must correspond from a path from $O_1$ to an external vertex from $S_1$.  So in all cases, we can find $m_p$ graded by $D_k$ for some $k$ such that $m_p | m$ and $c_k>0$. As in the proof of the claim, dividing out by $m_p$ strictly decreases the $w$-weight, and the theorem follows. 
\end{proof}

We can reformulate this into a statement about divisors. 
\begin{cor}\label{cor:partition} Let $T$ be a multi-set of arrows in $L(Q)$ such that $\sum_{a \in T} D_a=\sum_{i=1}^\rho c_i D_i, c_i \geq 0.$ Then there is a partition 
\[T= \sqcup_{i=1}^\rho (T^i_1 \sqcup \cdots \sqcup T^i_{c_i})\]
such that for all $i,k$, 
\[\sum_{a \in T^i_k} D_a =D_i.\]
\end{cor}
\begin{proof} Translating Theorem \ref{thm:generators} into a statement about divisors, we see that given such a $T$ there is partition of $T$ such that each subset sums to some $D_i$. The fact that there are exactly $c_i$ such subsets for each $D_i$ is because the $D_i$ are linearly independent in $\LL^\vee$. 
\end{proof}

We now consider the anti-cones of $X(Q)$. Recall that, by definition, an  anti-cone is a subset  $S \subset L(Q)_1$ such that the $D_a, a \in S$ span a cone $C_S$ in $\LL^\vee$ containing $w_Q$ in the interior.  Anti-cones correspond to cones in the fan description of the toric variety (see \cite{crepantconjecture} for how to obtain a fan from the GIT construction of a toric variety).  

The ample cone is the intersection 
\[ \bigcap C_S^\circ,\]
where $S$ ranges over all anti-cones, and $C_S^\circ$ is the open cone spanned by the $D_a$.

\begin{lem}\label{lem:ample} Suppose $S$ is an anti-cone for the ladder quiver. Then for all $i=1,\dots,\rho$, $D_i \in C_S$.
\end{lem}
\begin{proof}
By assumption, there are rational numbers $d_a >0$ such that
\[ \sum_{a \in S} d_a D_a = w_Q.\]
By multiplying by a common denominator, we can assume that
\[ \sum_{a \in S} c_a D_a = c w_Q,\]
where $c, c_a$ are positive integers. By Corollary \ref{cor:partition}, we can partition the left-hand sum into smaller sums, each which add to $D_i$ for some $i$, and each $D_i$ appears $c>0$ times. So for each $i$ there are some $p_a \geq 0$ such that
\[\sum_{a \in S} p_a D_a=D_i.\]
Thus $D_i$ is in the cone over the $D_a, a\in S,$ as claimed. 
\end{proof}

\begin{lem} For every arrow $a$, there is a collection of paths $p_i$ not containing $a$ such that 
\[\sum_{i=1}^\rho \sum_{a \in p_i} D_a=w_Q.\]
\end{lem}
\begin{proof} For an external vertex $v_i$, there is a maximum and a minimum path between $O_j$ and $v_i$ corresponding to the further up/left and the furthest right/down path respectively. These two paths do not share any arrows. So for any arrow $a$ we can choose a set of paths between $O_j$ and $v_i$,  $i \in S_j$ such that $a$ is not contained in any of the paths. \end{proof}

\begin{rem}
 The lemma shows that the condition
\[ \bigcap_{S \in A_w} S = \emptyset\]
holds for $X(Q)$, so $\LL^\vee$ is identified with the class group. 
\end{rem}

The anti-canonical divisor is $\sum_{a \in L(Q)_1} D_a,$ using the isomorphism $Cl(X(Q)) \cong \LL^\vee$. At the coordinate corresponding to $x_j, j \in L(Q)_0$, the coefficient of $-K_{X(Q)} \in \LL^\vee$ is the number of the arrows out minus the number of arrows into $x_j$ \cite{hille}. Notice that all vertices except the external vertices and the source have the same number of arrows in and out. Therefore $-K_{X(Q)}$ is in the sub-space generated by the vectors associated to the the external vertices. 
\begin{lem} If $-K_{M_Q}=\sum_{i=1}^\rho c_i c_1(W_i)$, then
\begin{equation}\label{anticanonical}
-K_{X(Q)}=\sum_{i=1}^\rho c_i D_i.
\end{equation}
\end{lem}
\begin{proof}
First, we observe that 
\[\sum_{i=1}^\rho c_i D_i=\sum_{i \in S_1'} c_i e_{v_i}-\sum_{i \in S_2'} c_i e_{v_i} +(\sum_{i \in S_2'} c_i+c_1) e_{v_1}.\]
We first consider the subquiver corresponding to one branch of the ladder quiver, then describe what happens when they are overlaid. Consider a vertex $i \in S'_j$. Suppose $j_1$ is the vertex just before $i$ in $S_j$, and $j_2$ the one just after. By \cite{Craw2011}, 
\[c_i= n_{0i}+r_{j_1}-r_{j_2}.\]
In the ladder quiver of this branch, if $r_{j_2}<r_i,$ there are
\[n_{0i}+r_{j_1}-r_i+r_i-r_{j_2}=c_i\]
arrows into the external vertex $v_i$, and no arrows out. 
If $r_{j_2} \geq r_i,$ then there are 
\[n_{0i}+r_{j_1}-r_i\]
arrows into the vertex $v_i$, and $r_{j_2} -r_i$ arrows out, so the difference is again $c_i$. If $i \in S_1'$, this shows that the $e_{v_i}$ component of the right and left hand side of \eqref{anticanonical} agree. If $i \in S_2'$, then when the ladder quiver is reflected, arrows in and out are reversed, so again, the $e_{v_i}$ component of the right and left hand side of \eqref{anticanonical} agree. It remains to consider the $e_{v_1}$ component.

Label the two vertices with arrows from vertex $1$, $a \in S_1$ and $b \in S_2$. Then, in the ladder quiver corresponding to the $S_1$ component of the quiver,
\[n_{01}-r_a=\#\{\text{arrows into vertex } 1\} -\#\{\text{arrows out of vertex } 1\}.\]
When we move to the full ladder quiver, there are an additional 
\[\sum_{i \in S_2'} s_i-r_i=-r_b+\sum_{i \in S_2'} c_i.\]
arrows into vertex $v_1$.  The $v_1$ component of the left-hand side of \eqref{anticanonical} is the sum of the last two equations:
\[n_{01}-r_a-r_b+\sum_{i \in S_2'} c_i=c_1+\sum_{i \in S_2'} c_i,\]
which is the right-hand side of the same equation. 
\end{proof}

\begin{prop} Let $i \in Q_0$. Then $D_i$ is a Cartier divisor. In particular, by \eqref{anticanonical}, $X(Q)$ is a Gorenstein toric variety.
\end{prop}
\begin{proof}
We use Kempf's descent lemma. We can define the trivial line bundle
\[L=\CC \times U^{ss}_{w_Q} \to U^{ss}_{w_Q}\]
on the semi-stable locus $U^{ss}_{w_Q} \subset \CC^{|L(Q)_1|}$, and endow it with a $K$-action using the weight $D_i$. This trivial line bundle descends to the quotient if for every $z=(z_a)_{a \in L(Q)_1}  \in U^{ss}_{w_Q}$, the stabilizer of $z$ acts trivially on the fiber of $L$. Recall that if $z$ is semi-stable, then 
\[ S=\{a: z_a \neq 0\}\]
is an anti-cone. So by the lemma above, we can find $a_1,\dots,a_l \in S$ such that
\[D_i= \sum_{j=1}^l e_j D_{a_j}.\]
The weights $D \in \chi(K)$ are characters, so they can be viewed as maps $\chi_D: K \to \CC^*$. Let $t \in K$ be in the stabilizer of $z$. The action of $t$ on the $a^{th}$ coordinate is given by multiplication by $\chi_{D_a}(t)$.  So for all $a \in S$, $\chi_{D_a}(t)=1.$ The action of $t$ on the fiber of $L$ is given by multiplication by $\chi_{D_i}(t)$, and
\[\chi_{D_i}(t)=\chi_{\sum_{j=1}^l D_{a_j}}(t)=\prod_{j=1}^l \chi_{D_{a_j}}(t)=1.\]
Therefore $L$ descends to the quotient, and we obtain the line bundle $L_i=\mathcal{O}(D_i)$. So $D_i$ is Cartier. 
\end{proof}

Let $L_i$ be the line bundle associated to the Cartier divisor $D_i$. 
\begin{cor} The $L_i$ are nef and globally generated.
\end{cor} 
\begin{proof}
Lemma \ref{lem:ample} show that $L_i$ is nef. $X(Q)$ is a toric variety, so $L_i$ is globally generated.
\end{proof}

\begin{cor} The line bundle $L_{w_Q}$ is very ample.
\end{cor}
\begin{proof} By definition of the GIT quotient giving $X(Q)$, we know that
\[X(Q)=\Proj(\oplus_{n\geq 0} \Gamma(L_{w_Q}^{\otimes n})).\]
Recall that 
\[L_{w_Q}=\otimes_{i=1}^\rho L_i,\]
so $ \Gamma(L_{w_Q}^{\otimes n})$ is the $n \sum D_i$ graded piece of $S_Q$. By Theorem \ref{thm:generators}, a monomial element of  $\Gamma(L_{w_Q}^{\otimes n})$ is a product of monomials, where there are some number $k_i$ of each degree $D_i$.  But since the divisors corresponding to the $L_i$ are linearly independent, in fact, there are $n$ monomials of each degree $D_i$. So we can write the monomial section of $\Gamma(L_{w_Q}^{\otimes n})$ as a product of $n$ monomial sections of $\Gamma(L_{w_Q})$ This shows that the on the right hand side in the statement of the corollary is generated in degree one. Therefore $L_{w_Q}$ is very ample. 
\end{proof}

There are exactly ${\tilde{s}_i \choose k_i}$ paths between $O_j$ and $v_i$ for a vertex $i \in S_j$ in the \emph{extended} ladder diagram. We can index such a path by size $k_i$ subset of $\{1,\dots,\tilde{s}_i\}$ by recording the horizontal steps from \emph{right to left}. Some of the paths are deleted when truncating to $L(Q)$. For a path between $O_j$ and $v_i$ indexed by a subset $J$ in the extended ladder diagram, we set $s^i_J\in \Gamma(L_i)$ to be associated section if $p$ is a path in $L(Q)$, and set $s_p=0$ otherwise. Then for each $i$, the set of sections $\{s^i_J: p \text{ a path in $L(Q)$}\}$ is a basis of $\Gamma(L_i)$. 
\begin{prop}\label{prop:ce} There is a closed embedding 
\[X(Q) \to \prod_{i=1}^\rho \PP^{\binom{\tilde{s}_i}{k_i}-1},\]
given by the map
\[\psi:=\CC[p] \to \oplus_{(a_i) \in \ZZ_{\geq 0}^{\rho}} \Gamma(\bigotimes_{i=1}^\rho L_i^{\otimes a_i}), \hspace{3mm} p^i_J \mapsto s^i_J.\]
\end{prop}
\begin{proof}
Let $A_i$ be the number of paths from $O_j$ to $v_i$ in the ladder diagram. 

Since each $L_i$ is globally generated, there is a map 
\[X(Q) \to \prod_{i=1}^\rho \PP^{A_i-1},\]
corresponding to the map of graded rings above. We need to show that it is a closed embedding. We can compose this map with the Segre embedding to get a map
\[X(Q) \to \PP^{A-1}.\]
where $A=\prod_{i=1}^\rho A_i$. 
 Since $L_{w_Q}$ is very ample, there is a closed embedding given by $L_{w_Q}$
 \[X(Q) \to \prod_{i=1}^\rho \PP^{A_i-1}.\]
 Since the natural map
 \[\bigotimes_{i=1}^\rho \Gamma(L_i) \to \Gamma(L_{w_Q})\]
 is surjective, the two maps
 \[X(Q) \to \prod_{i=1}^\rho \PP^{A_i-1},\]
are the same. 

We can use the zero sections to extend this to a closed embedding
\[X(Q) \to \prod_{i=1}^\rho \PP^{\binom{\tilde{s}_i}{k_i}-1}.\]
By construction, the morphism of graded rings is as described in the proposition. 
\end{proof}

There is a natural identification between the sections of the $L_i$ and initial terms of the monomials of matrix $A_i$ (see the construction  of the $A_i$ in \S\ref{sec:coords}). We illustrate this for the quiver flag variety in Example \ref{eg:Yshaped2}. The matrices which give the minors which define the embedding of the quiver flag variety in $\prod_{i=1}^\rho \mathbb{P}^{{\tilde{s}_i \choose k_i} -1}$ are
\[ A_1=\begin{bmatrix}
x^{(1)}_{11} & x^{(1)}_{12} & x^{(1)}_{13} & x^{(1)}_{14} & x^{(1)}_{15} & x^{(1)}_{16} \\
x^{(1)}_{21} & x^{(1)}_{22} & x^{(1)}_{23} & x^{(1)}_{24} & x^{(1)}_{25} & x^{(1)}_{26}\\
x^{(1)}_{31} & x^{(1)}_{32} & x^{(1)}_{33} & x^{(1)}_{34} & x^{(1)}_{35} & x^{(1)}_{36} \\
\end{bmatrix},\]
\[A_2=\begin{bmatrix}
x^{(1)}_{11} & x^{(1)}_{12} & x^{(1)}_{13} & x^{(1)}_{14} & x^{(1)}_{15} & x^{(1)}_{16} \\
x^{(1)}_{21} & x^{(1)}_{22} & x^{(1)}_{23} & x^{(1)}_{24} & x^{(1)}_{25} & x^{(1)}_{26}\\
x^{(1)}_{31} & x^{(1)}_{32} & x^{(1)}_{33} & x^{(1)}_{34} & x^{(1)}_{35} & x^{(1)}_{36} \\
x^{(2)}_{11} & x^{(2)}_{12} & x^{(2)}_{13} & x^{(2)}_{14} & x^{(2)}_{15} &x^{(2)}_{16} \\
x^{(2)}_{21} & x^{(2)}_{22} & x^{(2)}_{23} & x^{(2)}_{24} & x^{(2)}_{25} &x^{(2)}_{26} \\
\end{bmatrix}, \]
\[A_3=\begin{bmatrix}
x^{(3)}_{11} & x^{(3)}_{12} & x^{(3)}_{13} & x^{(3)}_{14} & x^{(3)}_{15} & x^{(3)}_{16} & x^{(3)}_{17} & x^{(3)}_{18}\\
 0 & 0& x^{(1)}_{11} & x^{(1)}_{12} & x^{(1)}_{13} & x^{(1)}_{14} & x^{(1)}_{15} & x^{(1)}_{16} \\
0&0&x^{(1)}_{21} & x^{(1)}_{22} & x^{(1)}_{23} & x^{(1)}_{24} & x^{(1)}_{25} & x^{(1)}_{26}\\
0&0& x^{(1)}_{31} & x^{(1)}_{32} & x^{(1)}_{33} & x^{(1)}_{34} & x^{(1)}_{35} & x^{(1)}_{36} \\
\end{bmatrix},\]
\[A_4=\begin{bmatrix}

x^{(4)}_{11} & x^{(4)}_{12} & x^{(4)}_{13} & x^{(4)}_{14} & x^{(4)}_{15} & x^{(4)}_{16} & x^{(4)}_{17} &x^{(4)}_{18}&x^{(4)}_{19}\\
x^{(4)}_{21} & x^{(4)}_{22} & x^{(4)}_{23} & x^{(4)}_{24} & x^{(4)}_{25} & x^{(4)}_{26} & x^{(4)}_{27} &x^{(4)}_{28}&x^{(4)}_{29}\\
x^{(4)}_{31} & x^{(4)}_{32} & x^{(4)}_{33} & x^{(4)}_{34} & x^{(4)}_{35} & x^{(4)}_{36} & x^{(4)}_{37} &x^{(4)}_{38}&x^{(4)}_{39}\\
x^{(4)}_{41} & x^{(4)}_{42} & x^{(4)}_{43} & x^{(4)}_{44} & x^{(4)}_{45} & x^{(4)}_{46} & x^{(4)}_{47} &x^{(4)}_{48}&x^{(4)}_{49}\\
0&x^{(3)}_{11} & x^{(3)}_{12} & x^{(3)}_{13} & x^{(3)}_{14} & x^{(3)}_{15} & x^{(3)}_{16} & x^{(3)}_{17} & x^{(3)}_{18}\\
 0&0 & 0& x^{(1)}_{11} & x^{(1)}_{12} & x^{(1)}_{13} & x^{(1)}_{14} & x^{(1)}_{15} & x^{(1)}_{16} \\
0&0&0&x^{(1)}_{21} & x^{(1)}_{22} & x^{(1)}_{23} & x^{(1)}_{24} & x^{(1)}_{25} & x^{(1)}_{26}\\
0&0&0& x^{(1)}_{31} & x^{(1)}_{32} & x^{(1)}_{33} & x^{(1)}_{34} & x^{(1)}_{35} & x^{(1)}_{36} \\
\end{bmatrix}\]

The identification is given by labelling the ladder diagram of the quiver as follows:
\[\begin{tikzpicture}[scale=1]
\draw (0,0) rectangle (1,1);
\draw (1,1) rectangle (2,2);
\draw (0,1) rectangle (1,2);
\draw (1,0) rectangle (2,1);
\draw (2,1) rectangle (3,2);
\draw (2,0) rectangle (3,1);
\draw (3,0) rectangle (4,1);
\draw (4,0) rectangle (5,1);
\draw (5,0) rectangle (6,1);
\draw (6,0) rectangle (7,1);
\draw (0,2) rectangle (1,3);
\draw (1,2) rectangle (2,3);
\draw (2,2) rectangle (3,3);
\draw (-1,0) rectangle (0,1);
\draw (-1,1) rectangle (0,2);
\draw (-1,2) rectangle (0,3);
\draw (-2,2) rectangle (-1,3);
\draw (-3,2) rectangle (-2,3);

\draw (2,3) rectangle (3,4);
\draw[fill,blue] (-1,0) circle (3pt);
\draw[fill] (1,1) circle (3pt);
\draw[fill] (2,1) circle (3pt);
\draw[fill] (3,1) circle (3pt);
\draw[fill] (7,1) circle (3pt);
\draw[fill] (1,2) circle (3pt);
\draw[fill] (2,2) circle (3pt);
\draw[fill,green] (2,3) circle (3pt);
\draw[fill] (3,4) circle (3pt);
\draw[fill] (0,1) circle (3pt);
\draw[fill] (-1,2) circle (3pt);
\draw[fill] (-3,2) circle (3pt);
\draw[fill] (0,2) circle (3pt);
\node[] at (1.5,3.2) {\tiny $x^{(1)}_{11}$};
\node[] at (1.5,2.2) {\tiny $x^{(1)}_{12}$};
\node[] at (1.5,1.2) {\tiny $x^{(1)}_{13}$};
\node[] at (1.5,0.2) {\tiny $x^{(1)}_{14}$};
\node[] at (0.5,3.2) {\tiny $x^{(1)}_{22}$};
\node[] at (0.5,2.2) {\tiny $x^{(1)}_{23}$};
\node[] at (0.5,1.2) {\tiny $x^{(1)}_{24}$};
\node[] at (0.5,0.2) {\tiny $x^{(1)}_{25}$};
\node[] at (-0.5,3.2) {\tiny $x^{(1)}_{33}$};
\node[] at (-0.5,2.2) {\tiny $x^{(1)}_{34}$};
\node[] at (-0.5,1.2) {\tiny $x^{(1)}_{35}$};
\node[] at (-0.5,0.2) {\tiny $x^{(1)}_{36}$};
\node[] at (-1.5,3.2) {\tiny $x^{(2)}_{14}$};
\node[] at (-1.5,2.2) {\tiny $x^{(2)}_{15}$};
\node[] at (-2.5,3.2) {\tiny $x^{(2)}_{25}$};
\node[] at (-2.5,2.2) {\tiny $x^{(2)}_{26}$};
\node[] at (2.5,4.2) {\tiny $x^{(3)}_{11}$};
\node[] at (2.5,3.2) {\tiny $x^{(3)}_{12}$};
\node[] at (2.5,2.2) {\tiny $x^{(3)}_{13}$};
\node[] at (2.5,1.2) {\tiny $x^{(3)}_{14}$};
\node[] at (2.5,0.2) {\tiny $x^{(3)}_{15}$};
\node[] at (6.5,1.2) {\tiny $x^{(4)}_{11}$};
\node[] at (6.5,0.2) {\tiny $x^{(4)}_{12}$};
\node[] at (5.5,1.2) {\tiny $x^{(4)}_{22}$};
\node[] at (5.5,0.2) {\tiny $x^{(4)}_{23}$};
\node[] at (4.5,1.2) {\tiny $x^{(4)}_{33}$};
\node[] at (4.5,0.2) {\tiny $x^{(4)}_{34}$};
\node[] at (3.5,1.2) {\tiny $x^{(4)}_{44}$};
\node[] at (3.5,0.2) {\tiny $x^{(4)}_{45}$};
\end{tikzpicture}.\]
For example, sections of $L_3$ correspond to paths from $(0,0)$ (the blue vertex) to $(4,4)$. For each such path, we identify it with the monomial which is the product of all the variables in the path. So, for example, the path marked in red below corresponds to $x^{(3)}_{13} x^{(1)}_{12} x^{(1)}_{23} x^{(1)}_{34}$, which is the initial term of the minor of $A_3$ given by the choice of columns $3,4,5,6$. 
\[\begin{tikzpicture}[scale=1]
\draw (0,0) rectangle (1,1);
\draw (1,1) rectangle (2,2);
\draw (0,1) rectangle (1,2);
\draw (1,0) rectangle (2,1);
\draw (2,1) rectangle (3,2);
\draw (2,0) rectangle (3,1);
\draw (3,0) rectangle (4,1);
\draw (4,0) rectangle (5,1);
\draw (5,0) rectangle (6,1);
\draw (6,0) rectangle (7,1);
\draw (0,2) rectangle (1,3);
\draw (1,2) rectangle (2,3);
\draw (2,2) rectangle (3,3);
\draw (-1,0) rectangle (0,1);
\draw (-1,1) rectangle (0,2);
\draw (-1,2) rectangle (0,3);
\draw (-2,2) rectangle (-1,3);
\draw (-3,2) rectangle (-2,3);

\draw (2,3) rectangle (3,4);
\draw[fill,blue] (-1,0) circle (3pt);
\draw[fill] (1,1) circle (3pt);
\draw[fill] (2,1) circle (3pt);
\draw[fill] (3,1) circle (3pt);
\draw[fill] (7,1) circle (3pt);
\draw[fill] (1,2) circle (3pt);
\draw[fill] (2,2) circle (3pt);
\draw[fill,green] (2,3) circle (3pt);
\draw[fill] (3,4) circle (3pt);
\draw[fill] (0,1) circle (3pt);
\draw[fill] (-1,2) circle (3pt);
\draw[fill] (-3,2) circle (3pt);
\draw[fill] (0,2) circle (3pt);
\draw[red, line width = 1mm] (-1,0) -- (-1,2) -- (3,2) -- (3,4);
\end{tikzpicture}.\]
\begin{prop} Let $p$ be a path between $O_i$ and $v_j$ for some external vertex $v_j, j \in S_i$. Let $J\subset \{1,\dots,\tilde{s}_i\}$ be the subset corresponding to this path. Then the leading term of the minor $p^i_J$ of $A_i$ is the product of the labels on the path $p$. 
\end{prop}
\begin{proof} Fix $j$ and $J$, and assume that $j \in S_1$ -- a very similar argument gives the claim when $j \in S_2$. The columns from right to left of the $r_i \times k_i$ block of the grid corresponding to the $i^{th}$ step of the quiver are labeled by entries of the $k_i$ rows of $A_i$. The $j^{th}$ column starts at the top with the $j^{th}$ entry of the row, so the $k^{th}$ label going down is $[A_i]_{j,j+k-1}$. If the path $p$ contains the $k^{th}$ horizontal arrow in the $j^{th}$ column, then $j+k-1 \in J$. Since the leading term of the minor $p^i_J$ of $A_i$ is the diagonal term, this shows the claim. 
\end{proof}
We can identify a partial order on paths between $O$ and $v_j$ for all $j$, given by $p_1 \leq p_2$ if $p_1$ is always below and to the right of $p_2$. Then this partial order is precisely the partial order given on the variables of $\CC[p]$. In fact, the paths have the structure of a distributive lattice. 

Recall the map 
\[\phi: \CC[p] \to \Pl(Q,\br), \hspace{3mm} p^i_J \mapsto m_J(A_i).\]
Define \[\overline{\phi}:\CC[p] \to \Pl(Q,\br), \hspace{3mm} p^i_J \mapsto lt(m_J(A_i))\]
where $lt(m_J(A_i))$ is the leading term of the minor $m_J(A_i)$.
\begin{mydef} Let $Y(Q)$ denote the toric variety given by the SAGBI degeneration of a $Y$-shaped quiver flag variety $Q$. That is, $Y(Q)$ is the closure of the morphism 
\[(\CC^*)^{\dim(M(Q,\br))} \to \prod_{i=1}^\rho \PP^{\binom{\tilde{s}_i}{k_i}},\]
defined by $\overline{\phi}$.  
\end{mydef}

 Note that $Y(Q)$ is cut out of $\prod_{i=1}^\rho \PP^{\binom{\tilde{s}_i}{k_i}}$ by $\ker(\overline{\phi}).$

\begin{thm}\label{thm:ladder} Let $Q$ be a $Y$-shaped quiver flag variety.  The toric variety $Y(Q)$ is isomorphic to $X(Q)$.
\end{thm}
\begin{proof}
$X(Q)$ is a subvariety of  $\prod_{i=1}^\rho \PP^{\binom{\tilde{s}_i}{k_i}}$ by Proposition \ref{prop:ce}. The equations defining $X(Q)$ are given by the kernel of the map $\psi$ defined in the proposition. It therefore suffices to show that $\ker(\overline{\phi})=\ker(\psi).$ Both of these ideals are generated by binomials, as both $X(Q)$ and $Y(Q)$ are toric varieties. A monomial in such a binomial is a product of monomials arising from paths in the ladder quiver.  In either case, a binomial $m_1-m_2$ lies in the kernel  if and only if the union of the supports of the paths in $m_1$ is the same as that for $m_2$. So  $\ker(\overline{\phi})=\ker(\psi).$
\end{proof}
\begin{rem} When applying this theorem to flag varieties, we obtain that in particular, that the degenerate toric variety described by \cite{GL} is the toric quiver moduli space $X(n,r_1,\dots,r_\rho)$. 
\end{rem}

\section{An extended Przyjalkowski method}\label{sec:przy}
The goal of the remainder of the paper is to exploit the toric degeneration produced above for $Y$-shaped quiver flag varieties to find Laurent polynomial mirrors. Given a quiver flag zero locus $Z$ in a Fano $Y$-shaped quiver flag variety $X$, we can look at what happens to $Z$ under the toric degeneration of $X$. If $Z$ degenerates to a complete intersection in $X(Q)$, we can apply the Przyjalkowski method to find a Laurent polynomial mirror of the complete intersection, which is a candidate mirror to $Z$.  We'll actually need a slight generalization of the Przyjalkowski method, which we describe here. 

The original Przyjalkowski \cite{laurentinversion} method has limited applications to complete intersections in flag varieties and Grassmannians, as there often is not a convex partition for the required line bundles. The generalization allows one produce a Laurent polynomial for all Fano complete intersections on flag varieties. In this section, we will use the notation introduced in \S \ref{subsection:GIT} for the GIT data of a toric variety.

\begin{mydef}\label{defn:convexpartition}
Let $\mathbb{L}$ be the co-character lattice of a torus $K$, and $D_1,\dots,D_m \in \mathbb{L}^\vee, \omega \in \LL^\vee_\RR$ the weight data for a toric stack. Let $L_1,\dots,L_k$ be divisors on $X$ corresponding to characters $L_i \in \LL^\vee$. An \emph{extended convex partition} for $L_1,\dots,L_k$ is the following data:
\begin{enumerate}
\item A subset $B$ of $\{1,\dots,m\}$ such that the $D_i: i \in B$ are a basis for $\LL^\vee$. A partition of $B$ into $B_0$ and $B_1$. 
\item A partition of $\{1,\dots,m\} - B_0$ into $k+1$ subsets $S_0,S_1,\dots,S_{k}$ such that for $i>0$, $L_i=\sum_{j \in S_i} D_j$ and each $L_i$ are in the positive span of $D_j, j \in B_0.$ 
\item A distinguished element $j_i \in S_i, i>0$. If $S_i \cap B_1 \neq S_i$, take $j_i \not \in B_1$. 
\end{enumerate} 
\end{mydef}
Choosing the basis corresponding to $B$, define $M$ to be an $r \times m$ matrix $M=[m_{ij}]$ with the first $r$ columns corresponding to the divisors in $B$ in this basis (in other words, the $r \times r$ identity matrix), and remaining columns corresponding to the remaining divisors $D_i$ written in the basis. 

Introduce variables $x_i$ for each $i \in \{1,\dots,m\}$. For each $S_i,i>0$ of size $n_i$, introduce $n_i$ new variables $y_{il}, l \in S_i$, but set $y_{i j_i}=1$. Then we impose the following equations: for $j \in S_i$, 
\[x_j=\frac{y_{ij}}{\sum_{k \in S_i} y_{ik}},\]
and
\[\prod_{j=1}^m x_{i}^{m_{ij}}=1.\]
We now use these equations to write the polynomial
$$W=\sum_{i=1}^m x_i$$
as a rational function in the variables $x_j, j \in S_0$ and $y_{il}$, such that $l \not \in B_0$, $l \neq i_j$. Call these the un-eliminated variables. Finally, delete the constant term.
\begin{lem} The resulting rational function is a Laurent polynomial in the un-eliminated variables. 
\end{lem}
\begin{proof}
The $r$ equations
\[\prod_{j=1}^m x_{i}^{m_{ij}}=1\]
each contain exactly one of the $r$ $x_i$ where $D_i \in B$. We therefore first solve for these $x_i$, expressing them as a Laurent monomial in the remaining $x_j$. Now, use 
\[x_j=\frac{y_{ij}}{\sum_{k \in S_i} y_{ik}},\]
to re-write these $r$ equations.

We first consider the equation corresponding to some $x_i, D_i \in B_0$. We obtain an equation for $x_i$, which we claim is a Laurent polynomial in the $y_{il}$, $i=1,\dots,k$, $l \neq j_i, l \not \in B_1$. This is because the factors of the form $\sum_{k \in S_t} y_{tk}$ all appear in the numerator; the degree of the factor $\sum_{k \in S_t} y_{tk}$ is the coefficient of $D_i$ in the expression $L_t:=\sum_{j \in B_0} c_j D_j,$ which is non-negative by the convexity assumption. 

 For $x_j, j \in B_1$, at first glance it seems like the situation is more complicated, as we also substitute $x_j$ as a rational function in the $y_{kl}$. However, as in the previous case, the degree of the factor $\sum_{k \in S_t} y_{tk}$ is the coefficient of $D_i$ in the expression $L_t:=\sum_{j \in B_0} c_j D_j,$, which is $0$ since $i \in B_1$. Therefore we obtain an expression for $y_{ij}$ as a monomial in the un-eliminated variables. We can substitute it into the equations for $x_i, i \in B_0$ without problem. 

Substituting the result into $W$, we get a Laurent polynomial as claimed. 
\end{proof}
\begin{eg} We present an example of a toric complete intersection where the extended Pryzjalkowski method allows us to produce a Laurent polynomial mirror, where the original one does not. The ambient toric variety is the Gelfand--Cetlin toric degeneration of $\Gr(6,3)$, represented by the ladder diagram: 
\[\begin{tikzpicture}[scale=0.6]
\draw (0,0) rectangle (1,1);
\draw (1,0) rectangle (2,1);
\draw (2,0) rectangle (3,1);
\draw (0,1) rectangle (1,2);
\draw (1,1) rectangle (2,2);
\draw (2,1) rectangle (3,2);
\draw (0,2) rectangle (1,3);
\draw (1,2) rectangle (2,3);
\draw (2,2) rectangle (3,3);

\draw[fill] (0,0) circle (3pt);
\draw[fill] (1,1) circle (3pt);
\draw[fill] (2,1) circle (3pt);
\draw[fill] (1,2) circle (3pt);
\draw[fill] (2,2) circle (3pt);
\draw[fill] (3,3) circle (3pt);
\draw[blue, line width = 1mm] (1,1) -- (1,2);
\draw[blue, line width = 1mm] (0,0) --(0,1) -- (1,1);
\draw[green, line width = 1mm] (1,1) -- (2,1);
\draw[blue, line width = 1mm] (1,2) -- (2,2);
\draw[blue, line width = 1mm] (2,2) --(2,3) -- (3,3);

\end{tikzpicture}.\]
The union of the blue and green arrows are the choice of basis $B$, with the green arrow the only element in $B_1$ and the blue arrows the set $B_0$. We label the arrows in $B_0$, $a^0_1, a^0_2, a^0_3,a^0_4$, ordered from bottom to top (i.e. in their appearance in the path). The corresponding variables we call
\[x^0_1,x^0_2,x^0_3,x^0_4.\]
For line bundles, take $L_1^{\oplus 5}$. 

Any choice of path from the bottom right to top left gives a representation of the associated divisor as a sum of arrow divisors. We choose the paths:

\[\begin{tikzpicture}[scale=0.6]
\draw (0,0) rectangle (1,1);
\draw (1,0) rectangle (2,1);
\draw (2,0) rectangle (3,1);
\draw (0,1) rectangle (1,2);
\draw (1,1) rectangle (2,2);
\draw (2,1) rectangle (3,2);
\draw (0,2) rectangle (1,3);
\draw (1,2) rectangle (2,3);
\draw (2,2) rectangle (3,3);

\draw[fill] (0,0) circle (3pt);
\draw[fill] (1,1) circle (3pt);
\draw[fill] (2,1) circle (3pt);
\draw[fill] (1,2) circle (3pt);
\draw[fill] (2,2) circle (3pt);
\draw[fill] (3,3) circle (3pt);
\draw[red, line width = 1mm] (0,0) -- (0,3)-- (3,3);
\end{tikzpicture},
\hspace{5mm}
\begin{tikzpicture}[scale=0.6]
\draw (0,0) rectangle (1,1);
\draw (1,0) rectangle (2,1);
\draw (2,0) rectangle (3,1);
\draw (0,1) rectangle (1,2);
\draw (1,1) rectangle (2,2);
\draw (2,1) rectangle (3,2);
\draw (0,2) rectangle (1,3);
\draw (1,2) rectangle (2,3);
\draw (2,2) rectangle (3,3);
\draw[fill] (0,0) circle (3pt);
\draw[fill] (1,1) circle (3pt);
\draw[fill] (2,1) circle (3pt);
\draw[fill] (1,2) circle (3pt);
\draw[fill] (2,2) circle (3pt);
\draw[fill] (3,3) circle (3pt);
\draw[yellow, line width = 1mm] (0,0) --(0,2) -- (1,2) -- (1,3) --(3,3);
\end{tikzpicture},
\hspace{5mm}
\begin{tikzpicture}[scale=0.6]
\draw (0,0) rectangle (1,1);
\draw (1,0) rectangle (2,1);
\draw (2,0) rectangle (3,1);
\draw (0,1) rectangle (1,2);
\draw (1,1) rectangle (2,2);
\draw (2,1) rectangle (3,2);
\draw (0,2) rectangle (1,3);
\draw (1,2) rectangle (2,3);
\draw (2,2) rectangle (3,3);

\draw[fill] (0,0) circle (3pt);
\draw[fill] (1,1) circle (3pt);
\draw[fill] (2,1) circle (3pt);
\draw[fill] (1,2) circle (3pt);
\draw[fill] (2,2) circle (3pt);
\draw[fill] (3,3) circle (3pt);

\draw[green, line width = 1mm] (0,0) -- (1,0) -- (1,1) -- (2,1) --(2,2)--(3,2)--(3,3);

\end{tikzpicture},
\hspace{5mm}
\begin{tikzpicture}[scale=0.6]
\draw (0,0) rectangle (1,1);
\draw (1,0) rectangle (2,1);
\draw (2,0) rectangle (3,1);
\draw (0,1) rectangle (1,2);
\draw (1,1) rectangle (2,2);
\draw (2,1) rectangle (3,2);
\draw (0,2) rectangle (1,3);
\draw (1,2) rectangle (2,3);
\draw (2,2) rectangle (3,3);

\draw[fill] (0,0) circle (3pt);
\draw[fill] (1,1) circle (3pt);
\draw[fill] (2,1) circle (3pt);
\draw[fill] (1,2) circle (3pt);
\draw[fill] (2,2) circle (3pt);
\draw[fill] (3,3) circle (3pt);

\draw[orange, line width = 1mm] (0,0) -- (2,0) -- (2,1) -- (3,1) -- (3,3);

\end{tikzpicture},
\hspace{5mm}
\begin{tikzpicture}[scale=0.6]
\draw (0,0) rectangle (1,1);
\draw (1,0) rectangle (2,1);
\draw (2,0) rectangle (3,1);
\draw (0,1) rectangle (1,2);
\draw (1,1) rectangle (2,2);
\draw (2,1) rectangle (3,2);
\draw (0,2) rectangle (1,3);
\draw (1,2) rectangle (2,3);
\draw (2,2) rectangle (3,3);

\draw[fill] (0,0) circle (3pt);
\draw[fill] (1,1) circle (3pt);
\draw[fill] (2,1) circle (3pt);
\draw[fill] (1,2) circle (3pt);
\draw[fill] (2,2) circle (3pt);
\draw[fill] (3,3) circle (3pt);
\draw[pink, line width = 1mm] (0,0) --(3,0) -- (3,3);

\end{tikzpicture}.
\]
We label the arrows appearing the $i^{th}$ path above as $a^i_j$, where the arrows are ordered as in the order in the path. Let $n_i$ be the length of the $i^{th}$ path (so that the paths are of lengths $1,2,4,2,1$ respectively).  The set $S_i$ is the set of divisors $\{D_{a^i_1},\dots, D_{a^i_{n_i}}\}$. For each $a^i_j$, we introduce variables $x^i_j$ and $y^i_j$. 
Note that as required, the set of arrows in each path are disjoint from each other and from $B_0$. The green path contains the arrow in $B_1$, which is permitted by the extended Przyjalkowski method (as opposed to the original version). The condition of convexity is easy to see from the diagram. 

We set $W=\sum_{a} x_a$. The Przyjalkowski method gives the following equations:
\[y^i_1=1,\]
\[x^i_j=\frac{y^i_j}{\sum_{j=1}^{n_i} y^i_j},\]
and
\[x^0_1=\frac{1}{x^1_1 x^2_1 x^3_1 x^4_1 x^5_1},\hspace{2mm} x^0_2=\frac{1}{x^1_1 x^2_1 x^3_3 x^4_2 x^5_1}, \hspace{2mm}
x^0_3=\frac{1}{x^1_1 x^2_2 x^3_3 x^4_2 x^5_1}, \]
 \[x^0_4=\frac{1}{x^1_1 x^2_2  x^3_4 x^4_2 x^5_1}, \hspace{2mm} x^3_2=\frac{x^4_2 x^3_3}{x^4_1}.\]
We immediately see that $x^1_1=x^5_1=y^1_1=y^5_1=1.$

Substituting the first set of equations into the second set of equations above gives:
\[x^0_1=\frac{(1+y^2_2)(1+y^3_2+y^3_3+y^3_4)(1+y^4_2)}{1},\hspace{2mm} x^0_2=\frac{(1+y^2_2)(1+y^3_2+y^3_3+y^3_4)(1+y^4_2)}{y^3_3 y^4_2},\]
\[x^0_3=\frac{(1+y^2_2)(1+y^3_2+y^3_3+y^3_4)(1+y^4_2)}{y^2_2 y^3_3 y^4_2}, \hspace{2mm}x^0_4=\frac{(1+y^2_2)(1+y^3_2+y^3_3+y^3_4)(1+y^4_2)}{y^2_2 y^3_4 y^4_2}, \]
\[ \frac{y^3_2}{1+y^3_2+y^3_3+y^3_4}=\frac{y^3_3 y^4_2}{1+y^3_2+y^3_3+y^3_4}.\]
Simplifying the last equation, we see that
\[y^3_2=y^3_3 y^4_2.\]
We can thus eliminate $x^0_1,x^0_2,x^0_3,x^0_4$ and $y^3_2$. The Laurent polynomial produced from the Przyjalkowski method is:
\[(1+y^2_2)(1+y^3_3 y^4_2+y^3_3+y^3_4)(1+y^4_2)(1+\frac{1}{y^3_3 y^4_2}+\frac{1}{y^2_2 y^3_3 y^4_2}+\frac{1}{y^2_2 y^3_4 y^4_2})-6.\]
This is a Laurent polynomial in 4 variables, which is the dimension of the toric complete intersection.  

\end{eg}

The first 20 terms of the period sequence of the Laurent polynomial and the quantum period of the four dimensional Fano variety cut out of $\Gr(6,3)$ by a generic section of $\mathcal{O}(1)^{\oplus 5}$ agree. This is expected because under the Gelfand--Cetlin toric degeneration, this Fano subvariety degenerates to the complete intersection on the singular toric variety cut out by $L_1^{\oplus 5}$.   

\begin{rem} Suppose $Z \subset Y$ is a Fano toric complete intersection in a toric variety, and $f$ is a Laurent polynomial mirror produced by the Przyjalkowski method for $Z$. Then Doran--Harder \cite{doranharder} show that the Przyjalkowski method gives the data of a toric degeneration of the toric complete intersection; the spanning fan of the Newton polytope of the mirror Laurent polynomial is the fan of the toric degeneration.  
\end{rem}

\section{Mirrors of quiver flag zero loci}\label{sec:mirrors}
Let $Z \subset Y$ be a Fano subvariety in a Fano variety $Y$. Suppose there is a toric degeneration of $Y$ to a singular toric variety $X$ such that $Z$ degenerates to a complete intersection in $X$. If this complete intersection has a convex partition, then one can apply the extended Przyjalkowski method to produce a candidate Laurent polynomial mirror to $Z$.  This approach has been used to find mirrors to complete intersections in flag varieties \cite{flagdegenerations}. The degeneration of $Y$-shaped ladder quivers introduced in this paper allows this approach to be extended to quiver flag zero loci in $Y$-shaped quivers, provided the vector bundles split into line bundles. That is, if $Z$ is a quiver flag zero locus in a $Y$-shaped quiver flag variety $M_\theta(Q,\br)$ cut out by line bundles (that is, a quiver flag complete intersection) then $Z$ degenerates to a complete intersection in $X(Q)$.
However, if $Z$ is a quiver flag locus cut out by sections of higher rank vector bundles $E_G$, the situation is more complicated. In this section, we give the first (to the author's knowledge) suggestion of how to find Laurent polynomial mirrors to these subvarieties, both in the flag and quiver flag cases. The difficulty is that it is unclear how to find the image of $Z$ under the degeneration, as it is not cut out by Pl\"ucker coordinates. Instead of taking this approach, we focus on the representations of $G$ that gives the vector bundles $E_G$.  

\subsection{The construction}
Let $L(Q)$ be a ladder diagram for some Fano $Y$-shaped quiver $Q$. If $X$ is a complete intersection in a $Y$-shaped quiver flag variety, then it is cut out by sections of line bundles given by tensor powers of the $\det(W_i)$. Using the identification of $\det(S)^*=\det(F)$ on the Grassmannian (here $S$ is the tautological sub-bundle and $F$ the tautological quotient bundle), we can see that the $\det(W_i)$ are the pullback of the tautological line bundles via the map
\[M_Q \to P^\vee=\prod_{i=1}^\rho \Gr^{sub}(k_i,\tilde{s}_i)\to \prod_{i=1}^\rho \PP^{{\tilde{s}_i \choose k_i}-1}.\]
As the degeneration of $M_Q$ is within this product of projective spaces, we see that such a complete intersection degenerates to a complete intersection in $X(Q)$.  The pullback of the $j^{th}$ tautological line bundle is $L_i$ on $X(Q)$.

The situation for quiver flag zero loci which are not complete intersection is less straightforward, essentially because the sections of $S^\alpha W_i$ cannot generally be expressed with Pl\"ucker coordinates. Thus there is no natural way to degenerate a quiver flag zero locus.  Instead we find direct sums of rank one reflexive sheaves whose tensor product is $L_i$ and which correspond to the same characters as appear in the Schur representation $S^\alpha W_i.$

Recall the GIT construction of a quiver flag variety $M_\theta(Q,\br)$ as $V/\!/G$, where $G=\prod_{i=1}^\rho \GL(r_i)$. The ladder quiver $L(Q)$ is also a GIT quotient:
\[ \CC^{|L(Q)_1|}/K, \hspace{5mm} K \cong (\CC^*)^{|L(Q)_0|-1}.\]
The group $G$ does not act on $\CC^{|L(Q)_1|}$. However, for each $i$, we can define an action of the diagonal torus of $\GL(r_i)$ on this vector space. 

We first explain this in the simplest example, that of a Grassmannian. The ladder quiver of $\Gr(n,r)$ is an $r \times (n-r)$ grid. An arrow $a$ in the ladder quiver moves from a vertex $(i_1,j_1)$ to a vertex $(i_2,j_2)$, where $i_1 \leq i_2$ and $j_1 \leq j_2 \leq r.$ If $t=(t_1,\dots,t_r) \in T_{ab}$,  we define the action of $t$ on the coordinate of a vector in  $\CC^{|L(Q)_1|}$ corresponding to the arrow $a$ to be
\[\prod_{j=j_1+1}^{j_2} t_j.\]
\begin{eg} Consider the Grassmannian $\Gr(5,2)$ with ladder quiver 
\[\begin{tikzpicture}[scale=0.6]
\draw (0,0) rectangle (1,1);
\draw (1,1) rectangle (2,2);
\draw (0,1) rectangle (1,2);
\draw (1,0) rectangle (2,1);
\draw (2,1) rectangle (3,2);
\draw (2,0) rectangle (3,1);
\draw[fill] (0,0) circle (3pt);
\draw[pink, line width = 1mm] (0,0) --(1,0) -- (1,1);
\draw[green, line width = 1mm] (2,1) -- (1,1);
\draw[blue, line width = 1mm] (0,0) -- (0,2)--(3,2);

\end{tikzpicture}.\]
Using the identification $\chi((\CC^*)^2) = \ZZ^2$, the character of the $(\CC^*)^2$ on the coordinate corresponding to the pink arrow is $(1,0)$, for the green arrow $(0,0)$, and for the blue arrow $(1,1)$. 
\end{eg}

For a $Y$-shaped quiver flag variety, the ladder diagram has an $r_i \times (s_i-r_i)$ block for each vertex $i=1,\dots,\rho$. Using this breakdown, we can define an action on $(\CC^*)^{r_i}$ on $\CC^{|L(Q)|_1}$: it acts as in the Grassmannian case on the $r_i \times (s_i-r_i)$ block, and trivially elsewhere. Given a set of arrows $S=\{a_1,\dots,a_s\}$, we define the character of $(\CC^*)^{r_i}$ associated to this set to be the sum of the characters corresponding to each $a_j$. 

\begin{mydef} Let $E_G$ be a rank $k$ representation theoretic vector bundle on a $Y$-shaped quiver flag variety $M_\theta(Q,\br)$, derived from the $G$-representation $E$.  For each $i$, the induced representation of $(\CC^*)^i$ on $E$ splits, say into characters $\beta^i_1,\dots, \beta^i_k.$

 Let $S_1,\dots, S_k$ be $k$ sets of arrows in $L(Q)$. 

We say that $\oplus_{j=1}^k \mathcal{O}(\sum_{a \in S_j} D_a)$ \emph{represents $E_G$} on $X(Q)$ if for each $i$, the characters associated to the collections $S_1,\dots,S_k$ are $\beta^i_1,\dots, \beta^i_k.$
\end{mydef}

 For example, taking $S_1,S_2, S_3$ to be the set of red, green, and yellow arrows respectively. This gives a sheaf that represents $W_1$ on $X(6,3)$: 
\[\begin{tikzpicture}[scale=0.6]
\draw (0,0) rectangle (1,1);
\draw (1,0) rectangle (2,1);
\draw (2,0) rectangle (3,1);
\draw (0,1) rectangle (1,2);
\draw (1,1) rectangle (2,2);
\draw (2,1) rectangle (3,2);
\draw (0,2) rectangle (1,3);
\draw (1,2) rectangle (2,3);
\draw (2,2) rectangle (3,3);

\draw[fill] (0,0) circle (3pt);
\draw[fill] (1,1) circle (3pt);
\draw[fill] (2,1) circle (3pt);
\draw[fill] (1,2) circle (3pt);
\draw[fill] (2,2) circle (3pt);
\draw[fill] (3,3) circle (3pt);
\draw[red, line width = 1mm] (0,0) -- (1,0) -- (1,1);
\draw[green, line width = 1mm] (1,1) -- (1,2);
\draw[yellow, line width = 1mm] (1,2) -- (1,3) -- (3,3);

\end{tikzpicture}.\]
There are multiple choices of such sheaves.

In this way, we can often associate to a quiver flag zero locus given by a vector bundle $E_G$ on a Fano $Y$-shaped quiver a subvariety given by a section of a sheaf representing $E_G$ on $X(Q)$. One can then apply the extended Przyjalkowski method to find a Laurent polynomial mirror. In many -- but not all -- cases, this Laurent polynomial has the same period sequence (up to 20 terms) as the quiver flag zero locus. 

 Consider the following example of a quiver flag zero locus (for which a candidate Laurent polynomial mirror previously wasn't known). 
\begin{eg} Consider the quiver flag zero locus on $\Gr(8,6)$ with bundles
 \[\wedge^5 W_1 \oplus \det(W_1) \oplus \det(W_1). \]
The six arrows in the path highlighted below represent $W_1$ on $X(8,6)$:
\[\begin{tikzpicture}[scale=0.6]
\draw (0,0) rectangle (1,1);
\draw (1,0) rectangle (2,1);
\draw (0,1) rectangle (1,2);
\draw (1,1) rectangle (2,2);
\draw (0,2) rectangle (1,3);
\draw (1,2) rectangle (2,3);
\draw (0,3) rectangle (1,4);
\draw (1,3) rectangle (2,4);
\draw (0,4) rectangle (1,5);
\draw (1,4) rectangle (2,5);
\draw (0,5) rectangle (1,6);
\draw (1,5) rectangle (2,6);

\draw[fill] (0,0) circle (3pt);
\draw[fill] (1,1) circle (3pt);
\draw[fill] (1,2) circle (3pt);
\draw[fill] (1,3) circle (3pt);
\draw[fill] (1,4) circle (3pt);
\draw[fill] (1,5) circle (3pt);
\draw[fill] (2,6) circle (3pt);
\draw[blue, line width = 1mm] (0,0) -- (0,1) -- (1,1) --(1,6)--(2,6);
\end{tikzpicture}.\]
The rank $6$ vector bundle  $\wedge^5 W_1$ is represented by arrows given in the following six diagrams: 
\[\begin{tikzpicture}[scale=0.6]
\draw (0,0) rectangle (1,1);
\draw (1,0) rectangle (2,1);
\draw (0,1) rectangle (1,2);
\draw (1,1) rectangle (2,2);
\draw (0,2) rectangle (1,3);
\draw (1,2) rectangle (2,3);
\draw (0,3) rectangle (1,4);
\draw (1,3) rectangle (2,4);
\draw (0,4) rectangle (1,5);
\draw (1,4) rectangle (2,5);
\draw (0,5) rectangle (1,6);
\draw (1,5) rectangle (2,6);

\draw[fill] (0,0) circle (3pt);
\draw[fill] (1,1) circle (3pt);
\draw[fill] (1,2) circle (3pt);
\draw[fill] (1,3) circle (3pt);
\draw[fill] (1,4) circle (3pt);
\draw[fill] (1,5) circle (3pt);
\draw[fill] (2,6) circle (3pt);
\draw[blue, line width = 1mm] (2,6) -- (2,1) -- (1,1);
\end{tikzpicture}
\hspace{5mm}
\begin{tikzpicture}[scale=0.6]
\draw (0,0) rectangle (1,1);
\draw (1,0) rectangle (2,1);
\draw (0,1) rectangle (1,2);
\draw (1,1) rectangle (2,2);
\draw (0,2) rectangle (1,3);
\draw (1,2) rectangle (2,3);
\draw (0,3) rectangle (1,4);
\draw (1,3) rectangle (2,4);
\draw (0,4) rectangle (1,5);
\draw (1,4) rectangle (2,5);
\draw (0,5) rectangle (1,6);
\draw (1,5) rectangle (2,6);

\draw[fill] (0,0) circle (3pt);
\draw[fill] (1,1) circle (3pt);
\draw[fill] (1,2) circle (3pt);
\draw[fill] (1,3) circle (3pt);
\draw[fill] (1,4) circle (3pt);
\draw[fill] (1,5) circle (3pt);
\draw[fill] (2,6) circle (3pt);
\draw[red, line width = 1mm] (0,0) -- (0,5) -- (1,5);
\end{tikzpicture}
\hspace{5mm}
\begin{tikzpicture}[scale=0.6]
\draw (0,0) rectangle (1,1);
\draw (1,0) rectangle (2,1);
\draw (0,1) rectangle (1,2);
\draw (1,1) rectangle (2,2);
\draw (0,2) rectangle (1,3);
\draw (1,2) rectangle (2,3);
\draw (0,3) rectangle (1,4);
\draw (1,3) rectangle (2,4);
\draw (0,4) rectangle (1,5);
\draw (1,4) rectangle (2,5);
\draw (0,5) rectangle (1,6);
\draw (1,5) rectangle (2,6);

\draw[fill] (0,0) circle (3pt);
\draw[fill] (1,1) circle (3pt);
\draw[fill] (1,2) circle (3pt);
\draw[fill] (1,3) circle (3pt);
\draw[fill] (1,4) circle (3pt);
\draw[fill] (1,5) circle (3pt);
\draw[fill] (2,6) circle (3pt);

\draw[green, line width = 1mm] (2,6) -- (2,2) -- (1,2);
\draw[green, line width = 1mm] (0,0) -- (0,1) -- (1,1);

\end{tikzpicture}
\hspace{5mm}
\begin{tikzpicture}[scale=0.6]
\draw (0,0) rectangle (1,1);
\draw (1,0) rectangle (2,1);
\draw (0,1) rectangle (1,2);
\draw (1,1) rectangle (2,2);
\draw (0,2) rectangle (1,3);
\draw (1,2) rectangle (2,3);
\draw (0,3) rectangle (1,4);
\draw (1,3) rectangle (2,4);
\draw (0,4) rectangle (1,5);
\draw (1,4) rectangle (2,5);
\draw (0,5) rectangle (1,6);
\draw (1,5) rectangle (2,6);

\draw[fill] (0,0) circle (3pt);
\draw[fill] (1,1) circle (3pt);
\draw[fill] (1,2) circle (3pt);
\draw[fill] (1,3) circle (3pt);
\draw[fill] (1,4) circle (3pt);
\draw[fill] (1,5) circle (3pt);
\draw[fill] (2,6) circle (3pt);

\draw[yellow, line width = 1mm] (2,6) -- (2,3) -- (1,3);
\draw[yellow, line width = 1mm] (0,0) -- (0,2) -- (1,2);

\end{tikzpicture}
\hspace{5mm}
\begin{tikzpicture}[scale=0.6]
\draw (0,0) rectangle (1,1);
\draw (1,0) rectangle (2,1);
\draw (0,1) rectangle (1,2);
\draw (1,1) rectangle (2,2);
\draw (0,2) rectangle (1,3);
\draw (1,2) rectangle (2,3);
\draw (0,3) rectangle (1,4);
\draw (1,3) rectangle (2,4);
\draw (0,4) rectangle (1,5);
\draw (1,4) rectangle (2,5);
\draw (0,5) rectangle (1,6);
\draw (1,5) rectangle (2,6);

\draw[fill] (0,0) circle (3pt);
\draw[fill] (1,1) circle (3pt);
\draw[fill] (1,2) circle (3pt);
\draw[fill] (1,3) circle (3pt);
\draw[fill] (1,4) circle (3pt);
\draw[fill] (1,5) circle (3pt);
\draw[fill] (2,6) circle (3pt);

\draw[orange, line width = 1mm] (2,6) -- (2,4) -- (1,4);
\draw[orange, line width = 1mm] (0,0) -- (0,3) -- (1,3);

\end{tikzpicture}
\hspace{5mm}
\begin{tikzpicture}[scale=0.6]
\draw (0,0) rectangle (1,1);
\draw (1,0) rectangle (2,1);
\draw (0,1) rectangle (1,2);
\draw (1,1) rectangle (2,2);
\draw (0,2) rectangle (1,3);
\draw (1,2) rectangle (2,3);
\draw (0,3) rectangle (1,4);
\draw (1,3) rectangle (2,4);
\draw (0,4) rectangle (1,5);
\draw (1,4) rectangle (2,5);
\draw (0,5) rectangle (1,6);
\draw (1,5) rectangle (2,6);

\draw[fill] (0,0) circle (3pt);
\draw[fill] (1,1) circle (3pt);
\draw[fill] (1,2) circle (3pt);
\draw[fill] (1,3) circle (3pt);
\draw[fill] (1,4) circle (3pt);
\draw[fill] (1,5) circle (3pt);
\draw[fill] (2,6) circle (3pt);

\draw[pink, line width = 1mm] (2,6) -- (2,5) -- (1,5);
\draw[pink, line width = 1mm] (0,0) -- (0,4) -- (1,4);

\end{tikzpicture}\]
The ladder diagram determines a weight matrix for $X(8,6)$:
\setcounter{MaxMatrixCols}{18}
\[\begin{bmatrix}
1 & 1 &-1 & 0 & 0 & 0 & 0 & 0 & 0 & 0 & 0 & 0 & 0 & 0 &-1 &-1& 0 & 0 \\
0 & 0 & 1 & 1 &-1 & 0 & 0 & 0 & 0 & 0 & 0 & 0 & 0 &-1 & 0 & 0& 0 & 0 \\
0 & 0 & 0 & 0 & 1 & 1 &-1 & 0 & 0 & 0 & 0 & 0 &-1 & 0 & 0 & 0& 0 & 0 \\
0 & 0 & 0 & 0 & 0 & 0 & 1 & 1 &-1 & 0 & 0 &-1 & 0 & 0 & 0 & 0& 0 & 0 \\
0 & 0 & 0 & 0 & 0 & 0 & 0 & 0 & 1 & 1 &-1 & 0 & 0 & 0 & 0 & 0& 0 & 0 \\
0 & 0 & 0 & 0 & 0 & 0 & 0 & 0 & 0 & 0 & 1 & 1 & 1 & 1 & 1 & 1& 1 & 1\\
\end{bmatrix}.\]
The weights for the line bundles (or rather divisors) are given by
\[\begin{bmatrix} 
0 & 0 &-1 & 0 & 1 & 0 & 0 & 0\\
0 & 0 & 0 & 0 &-1 & 1 & 0 & 0\\
0 & 0 & 0 & 0 & 0 &-1 & 1 & 0\\
0 & 0 & 0 & 0 & 0 & 0 &-1 & 1\\
0 & 0 & 0 & 1 & 0 & 0 & 0 &-1\\
1 & 1 & 1 & 0 & 1 & 1 & 1 & 1\\
\end{bmatrix}.\]
The ladder diagrams give a way of picking a convex partition for these bundles. One can formally follow the Przyjalkowski method (see \cite{laurentinversion}) to produce a Laurent polynomial. This Laurent polynomial is \emph{not} rigid maximally mutable, but it also does not have the correct period sequence. There is a unique rigid maximally mutable Laurent polynomial on the Newton polytope of this polynomial. It has the correct classical period, up to 10 terms. The Laurent polynomial is
\[\begin{aligned}&x y z w + x y z + x y w + 2 x y + x y/w + x z w + x z \\&+ x w + 2 x + x/w + x/z +
x/(z w) + y z w + y z + y w \\&+ 2 y + y/w + z w + z + w + 2/w + 2/z + 2/(z w) +
1/y \\& + 1/(y w) + 1/(y z) + 1/(y z w) + 1/x + 1/(x w) + 1/(x z) + 1/(x z w) +
1/(x y) \\&+ 1/(x y w) + 1/(x y z) + 1/(x y z w).\end{aligned}\]

\end{eg}

\subsection{More examples}
Below, we give some more examples of this method applied to various Fano fourfolds. Some examples give candidate Laurent polynomial mirrors in situations slightly beyond the scope of the methods discussed above, that might suggest generalizations of the method. 
\begin{eg}
Consider the quiver flag zero locus given by the quiver
\begin{center}
\includegraphics[scale=0.5]{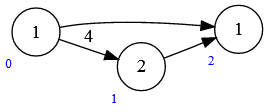}
\end{center}
with bundles $W_1 \otimes W_2$. This Fano variety has PID 115 (see the tables in the appendicies of \cite{kalashnikov}). The paths on the ladder diagram which give the divisors suggested by the above method is
\[\begin{tikzpicture}[scale=0.6]
\draw (0,0) rectangle (1,1);
\draw (1,1) rectangle (2,2);
\draw (0,1) rectangle (1,2);
\draw (1,0) rectangle (2,1);
\draw (2,0) rectangle (3,1);
\draw (3,0) rectangle (4,1);

\draw[fill] (0,0) circle (3pt);
\draw[fill] (0,1) circle (3pt);
\draw[fill] (1,1) circle (3pt);
\draw[fill] (2,1) circle (3pt);
\draw[fill] (2,2) circle (3pt);
\draw[fill] (4,1) circle (3pt);
\draw[red, line width = 1mm] (0,0) -- (0,1) -- (1,1);
\draw[red, line width = 1mm] (0,-0.1) -- (4,-0.1) -- (4,1);
\draw[yellow, line width = 1mm] (1,1) -- (1,2) -- (2,2);
\draw[yellow, line width = 1mm] (0,0) -- (3,0) -- (3,1) --(4,1);
\end{tikzpicture}.\]
Again, to find a mirror with the correct period sequence, one must find a rigid maximally mutable Laurent polynomial supported on the resulting Newton polytope. This is
 \[x + y w + y + z + w + 1/x + 1/(x w) + 1/(x z) + z/(x y w) + 1/(x y) + 1/(x y w)
+ 1/(x y z)
 \]
\end{eg}
\begin{eg} Consider the quiver flag zero locus given by the quiver flag variety
\begin{center}
\includegraphics[scale=0.5]{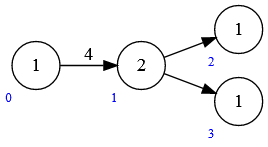}
\end{center}
and bundle $W_1 \oplus W_2$. It corresponds to PID 20. The toric degeneration is given by the following ladder diagram:
\[\begin{tikzpicture}[scale=0.6]
\draw (1,0) rectangle (2,1);
\draw (2,0) rectangle (3,1);
\draw (3,0) rectangle (4,1);
\draw (2,1) rectangle (3,2);
\draw (1,1) rectangle (2,2);
\draw (0,1) rectangle (1,2);
\draw[red, line width = 1mm] (1,0) -- (4,0);
\draw[red, line width = 1mm] (4,0) -- (4,1);
\draw[yellow, line width = 1mm] (0,1) -- (0,2);
\draw[yellow, line width = 1mm] (0,2) -- (3,2);

\draw[fill] (0,1) circle (3pt);
\draw[fill] (1,1) circle (3pt);
\draw[fill] (2,1) circle (3pt);
\draw[fill] (3,1) circle (3pt);
\draw[fill] (4,1) circle (3pt);
\draw[fill] (1,0) circle (3pt);
\draw[fill] (3,2) circle (3pt);
\end{tikzpicture}.\]
The mirror produced is
$$x + y + z + w + z/y + 1/(y w) + w/x + 1/(x z).$$
\end{eg}
\begin{eg}
Consider the quiver flag zero locus with PID 232 given by the quiver flag variety
\begin{center}
\includegraphics[scale=0.5]{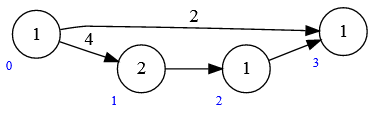}
\end{center}
and the quiver flag bundles $\det(W_1) \oplus W_1 \otimes W_3$. The ladder diagram and paths given by the prescribed method are
\[\begin{tikzpicture}[scale=0.6]
\draw (0,0) rectangle (1,1);
\draw (1,1) rectangle (2,2);
\draw (0,1) rectangle (1,2);
\draw (1,0) rectangle (2,1);
\draw (2,0) rectangle (3,1);
\draw (3,0) rectangle (4,1);
\draw (4,0) rectangle (5,1);

\draw[red, line width = 1mm] (0,0) -- (0,1) -- (1,1);
\draw[red, line width = 1mm] (0,-0.1) -- (5,-0.1) -- (5,1);
\draw[yellow, line width = 1mm] (1,1) -- (1,2) -- (2,2);
\draw[yellow, line width = 1mm] (0,0) -- (4,0) -- (4,1) --(5,1);
\draw[green, line width = 1mm] (-0.1,0) -- (-0.1,2.1) -- (2,2.1);

\draw[fill] (0,0) circle (3pt);
\draw[fill] (0,1) circle (3pt);
\draw[fill] (1,1) circle (3pt);
\draw[fill] (2,1) circle (3pt);
\draw[fill] (2,2) circle (3pt);
\draw[fill] (3,1) circle (3pt);
\draw[fill] (5,1) circle (3pt);
\end{tikzpicture}\]
The mirror is
\[\begin{split}x + y + z + w + w/y + 1/y + 1/x + 1/(x w) \\ + 1/(x z) + 1/(x z w) + 2/(x y) +
1/(x y w) + 1/(x y z) + 1/(x y z w) + 1/(x^2 z w) \\+ 1/(x^2 y w) + 2/(x^2 y z w)
+ 1/(x^2 y z w^2) + 1/(x^3 y z w^2)
.\end{split}\]
\end{eg}
Of the 141 Fano fourfolds found in \cite{kalashnikov}, all can be modeled as quiver flag zero loci in a $Y$-shaped quiver. For 99 of these, the methods above produce a candidate Laurent polynomial mirror. Many of the remaining did not have a convex partition supporting the choice of line bundles. The next example is an example of this; in this case (but not usually), one is able to find a degeneration of the complete intersection to a toric variety; we find the Laurent polynomial associated to this toric variety and after taking the rigid maximally mutable Laurent polynomial on its polytope, find a mirror.
\begin{eg}[PID 104] Consider the quiver flag variety obtained from the quiver flag variety with PID 104 in the tables by grafting:
\begin{center}
\includegraphics[scale=0.5]{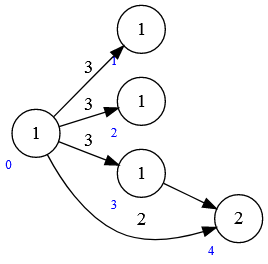}.
\end{center}
The quiver flag zero locus is given by bundles $W_1 \otimes W_4 \oplus W_2 \otimes W_4.$
The toric degeneration is given by the product of the 3 toric varieties given by the three ladder diagrams below:
\[\begin{tikzpicture}[scale=0.6]
\draw (0,0) rectangle (1,1);
\draw (1,0) rectangle (2,1);
\draw (2,1) rectangle (3,2);
\draw (2,0) rectangle (3,1);

\draw[fill] (0,0) circle (3pt);
\draw[fill] (2,1) circle (3pt);
\draw[fill] (3,2) circle (3pt);
\draw[green, line width = 1mm] (0,0) -- (2,0) --(2,1);
\draw[red, line width = 1mm] (2,1) -- (3,1) -- (3,2);
\draw[yellow, line width = 1mm] (0,0) -- (0,1) -- (2,1);
\draw[blue, line width = 1mm] (2,1) -- (2,2) -- (3,2);
\end{tikzpicture}
\hspace{5mm}
\begin{tikzpicture}[scale=0.6]
\draw (0,0) rectangle (1,1);
\draw (1,0) rectangle (2,1);
\draw[fill] (0,0) circle (3pt);
\draw[fill] (2,1) circle (3pt);
\draw[green, line width = 1mm] (0,0) -- (2,0) -- (2,1);
\draw[red, line width = 1mm] (0,0) -- (0,1) -- (2,1);
\end{tikzpicture}
\hspace{5mm}
\begin{tikzpicture}[scale=0.6]
\draw (0,0) rectangle (1,1);
\draw (1,0) rectangle (2,1);
\draw[fill] (0,0) circle (3pt);
\draw[fill] (2,1) circle (3pt);
\draw[yellow, line width = 1mm] (0,0) -- (2,0) -- (2,1);
\draw[blue, line width = 1mm] (0,0) -- (0,1) -- (2,1);
\end{tikzpicture}.\]
Notice that there is no convex partition which will give these bundles, because there is no choice of basis of divisors in the first ladder diagram such that all chosen divisors are in the positive span. We instead construct a toric degeneration, using similar ideas to that of \cite{doranharder}. Suppose the fan of toric variety is in the lattice $N_\mathbb{R}$. I find $v_1,\dots,v_4 \in (N^\vee)_\mathbb{R}$ such that they define binomial sections of the four line bundles; the associated toric subvariety has the following Laurent polynomial mirror, with matching period sequence. 
$$x + y + z + w + y/(x w) + 1/x + 1/(x w) + w/(x z) + 1/(x z) + 1/(x y) + w/(x y z) + 1/(x y z)$$
\end{eg}
In other examples, the degeneration of the ambient quiver flag variety has too low Picard rank. Consider the quiver flag variety which appeared in one of the factors in the previous example:
\begin{center}
\includegraphics[scale=0.5]{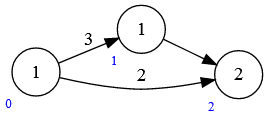}
\end{center}
with toric degeneration given by the ladder diagram 
\[\begin{tikzpicture}[scale=0.6]
\draw (0,0) rectangle (1,1);
\draw (1,0) rectangle (2,1);
\draw (2,1) rectangle (3,2);
\draw (2,0) rectangle (3,1);

\draw[fill] (0,0) circle (3pt);
\draw[fill] (2,1) circle (3pt);
\draw[fill] (3,2) circle (3pt);
\end{tikzpicture}.\]
Notice that from the quiver flag variety, we would expect to have a class group of rank at least 3 (generated by two Weil divisors coming from $W_2$ and one Cartier divisor from $W_1$), but the toric degeneration has rank only 2. In the previous example, we are still able to find a mirror, because the bundles only involve $W_2$. However, in the case of PID 15, where the quiver flag variety is
\begin{center}
\includegraphics[scale=0.5]{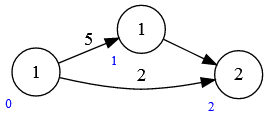}
\end{center}
and the bundles are $W_1 \otimes W_2$, this method fails to produce a mirror. The ladder diagram is
\[\begin{tikzpicture}[scale=0.6]
\draw (0,0) rectangle (1,1);
\draw (1,0) rectangle (2,1);
\draw (4,1) rectangle (5,2);
\draw (2,0) rectangle (3,1);
\draw (3,0) rectangle (4,1);
\draw (4,0) rectangle (5,1);

\draw[fill] (0,0) circle (3pt);
\draw[fill] (4,1) circle (3pt);
\draw[fill] (5,2) circle (3pt);
\end{tikzpicture},\]
so we see that one of the summands of $W_2$ and $W_1$ degenerate to the same bundle.

In the theorem below, we show that for a Fano complete intersection in a Fano $Y$-shaped quiver, there is always a convex partition for the degenerate toric variety and line bundles. 
\begin{thm}\label{extendprz} For any Fano complete intersection in a $Y$-shaped quiver flag variety, the extended Przyjalkowski method can be applied to produce a Laurent polynomial.
\end{thm}
\begin{proof}
Let $Q$ be a $Y$-shaped quiver, and $i \in S_j$. First, we show that there are $s_i-s_i'$ non-overlapping paths between $O_j$ and $i$. We could translate this statement to the statement that there is a length $\sum_{i} s_i-s_i'$ chain under the partial order $\leq$ on the $p^l_\sigma$, such that $s_i-s_i'$ of the paths are between $O_j$ and $v_i$.

Consider first the ladder quiver of a Grassmannian, which is from an $n-r \times r$ grid of boxes. There are clearly $n$ distinct paths from the source to single external vertex $v_1$. For example, for $Gr(4,2)$, consider the following four paths:
 \[\begin{tikzpicture}[scale=0.6]
\draw (0,0) rectangle (1,1);
\draw (1,0) rectangle (2,1);
\draw (0,1) rectangle (1,2);
\draw (1,1) rectangle (2,2);

\draw[fill] (0,0) circle (3pt);
\draw[fill] (1,1) circle (3pt);
\draw[fill] (2,2) circle (3pt);
\draw[yellow, line width = 0.8mm] (0,0) -- (2,0) -- (2,2);
\draw[red, line width = 0.8mm] (0,0.1) -- (1,0.1) -- (1,1) -- (1.9,1) -- (1.9,2);
\draw[green, line width = 0.8mm] (0.1,0) -- (0.1,1) -- (1,1) -- (1,1.9) -- (2,1.9);
\draw[blue, line width = 0.8mm] (0,0) -- (0,2) -- (2,2);
\end{tikzpicture}.\]
In general, we can choose them in this way (starting with the path that goes as east for $n-r$ steps, the north for $r$ steps, then the one that goes east $n-r-1$ steps, then north, then east, then north until it reaches $v_i$, and so on). Label the $n$ paths chosen in this way from $1,\dots,n$ from bottom/right to top/left. 

Now consider a general Fano $Y$-shaped quiver. We first show the paths in an example. Consider the ladder diagram of the quiver \ref{eg:Yshaped2}:
\[\begin{tikzpicture}[scale=0.6]
\draw (-1,0) rectangle (0,1);
\draw (-1,1) rectangle (0,2);
\draw (-1,2) rectangle (0,3);
\draw (0,0) rectangle (1,1);
\draw (1,1) rectangle (2,2);
\draw (0,1) rectangle (1,2);
\draw (1,0) rectangle (2,1);
\draw (2,1) rectangle (3,2);
\draw (2,0) rectangle (3,1);
\draw (3,0) rectangle (4,1);
\draw (4,0) rectangle (5,1);
\draw (5,0) rectangle (6,1);
\draw (6,0) rectangle (7,1);
\draw (0,2) rectangle (1,3);
\draw (1,2) rectangle (2,3);
\draw (2,2) rectangle (3,3);
\draw (2,3) rectangle (3,4);
\draw (-2,2) rectangle (-1,3);
\draw (-3,2) rectangle (-2,3);

\draw[fill,blue] (-1,0) circle (3pt);
\draw[fill,green] (2,3) circle (3pt);
\draw[fill] (-1,2) circle (3pt);
\draw[fill] (-3,2) circle (3pt);
\draw[fill] (0,1) circle (3pt);
\draw[fill] (0,2) circle (3pt);
\draw[fill] (1,1) circle (3pt);
\draw[fill] (2,1) circle (3pt);
\draw[fill] (3,1) circle (3pt);
\draw[fill] (7,1) circle (3pt);
\draw[fill] (1,2) circle (3pt);
\draw[fill] (2,2) circle (3pt);
\draw[fill] (3,4) circle (3pt);

\draw[yellow, line width = 0.8mm] (-1,0) -- (7,0) -- (7,1);
\draw[yellow, line width = 0.8mm] (-1,0) -- (3,0) -- (3,1) -- (7,1);
\draw[yellow, line width = 0.8mm] (4,0) -- (4,1);
\draw[yellow, line width = 0.8mm] (5,0) -- (5,1);
\draw[yellow, line width = 0.8mm] (6,0) -- (6,1);
\draw[red, line width = 0.8mm] (-1,0.1) -- (2,0.1) -- (2,1) -- (3,1) -- (3,4);
\draw[red, line width = 0.8mm] (1,0.1) -- (1,1) -- (2,1) -- (2,2) -- (3,2);
\draw[red, line width = 0.8mm] (0,0.1) -- (0,1) -- (1,1) -- (1,2) -- (2,2) -- (2,3) -- (3,3);
\draw[red, line width = 0.8mm] (-1,0) -- (-1,1) -- (0,1) -- (0,2) -- (1,2) -- (1,3) -- (2,3) -- (2,4) -- (3,4);
\draw[blue, line width = 0.8mm] (-1.08,0) -- (-1.08,2) -- (0,2) -- (0,3.08) -- (2,3.08);
\draw[green, line width = 0.8mm] (-3,2) -- (-1,2) -- (-1,2.92) --  (2,2.92);
\draw[green, line width = 0.8mm] (-3,2) -- (-3,2.92) -- (-1,2.92);
\draw[green, line width = 0.8mm] (-2,2) -- (-2,2.92);
\end{tikzpicture}.\]
The colors give partition of the paths: there are five paths in the yellow region (which cover the yellow region) from $O_1$ to $v_4$, four paths in the red region from $O_1$ to $v_3$, one path in the blue region from $O_1$ to $v_1$, and three paths in the green region from $v_2$ to $O_2$. In each case, the number of paths in the prescribed region is $s_i-s_i'$. 

In the general case, consider one vertex $i$ which is not $1$ or the source. Suppose we have labeled the vertices such that next vertex on the leg is $i+1$, if it is exists, and the previous one is $i-1$.  The ladder quiver contains (part) of the ladder quiver of $\Gr(\tilde{s}_i,r_i)$. Select the paths corresponding to $r_{i+1}+1,\dots,r_{i},r_{i}+1,\dots, s_i-(s_{i-1}-r_{i-1})$, a total of $s_i-s_{i-1}+r_{i-1}-r_{i+1}=s_i-s_i'$ paths on the Grassmannian ladder (and also on the sub-ladder quiver of $Q$ corresponding to $i$). The way in which the grids are overlayed ensure that we choose in this way distinct paths for all $i \neq 1$. The paths chosen for $i=1$ is precisely what remains -- it is easy to see that there are $s_1-s_1'$ remaining paths.

Now suppose that we are given the data of a Fano complete intersection in a Fano $Y$-shaped quiver flag variety. The complete intersection is given by a direct sum of line bundles each of the form $\det(W_1)^{\otimes a_{1}} \otimes \cdots \otimes \det(W_{\rho})^{\otimes a_\rho}.$ 
We now translate the combinatorial requirements of the extended Przyjalkowski method to the quiver. Arrows on the quiver correspond to weights of the toric variety. Paths (or collections of paths) correspond to sums of the weights associated to each of the arrows. Notice that the paths described above partition the arrows of the ladder quiver into $\sum_{i} s_i-s_i'$ subsets. 

First, we choose a subset of the arrows corresponding to a basis. For each $i$, let $b_i$ be a path between $O_j$ and $v_i$ which can be broken into a maximal number of steps. Let $B_0$ be the maximal set of arrows in all of these paths such that they are linearly independent. Extend $B_0$ to a basis by choosing arrows into each vertex -- this is $B_1$. Considering the arrows not contained in $B_0$, we see there is at least enough to build $s_i-s_i'-1$ paths between $O_j$ and $v_i$ for each $i$. Each of these correspond to a copy of $\det(W_i)$. There fact that we are given a Fano toric complete intersection ensures that each $\det(W_i)$ appears no more than $s_i-s_i'-1$ times as a factor in all of the line bundles. This means that we can choose a convex partition for these line bundles, and hence apply the extended Przyjalkowsi method. 
\end{proof}

To find mirrors of quiver flag zero loci more generally, we need to find good degenerations of quiver flag varieties beyond $Y$-shaped quivers. We give one example of such a degeneration; a SAGBI basis degeneration of the sections of the $\det(S_i^*)$ (see \ref{sec:coords}) which cannot be represented as a ladder diagram: instead it is represented by  a \emph{bound} ladder diagram. Bound quivers were introduced in \cite{CrawSmith} as certain subvarieties of toric quiver flag varieties.

Consider the quiver flag variety $M_Q$
\begin{center}
\includegraphics[scale=0.5]{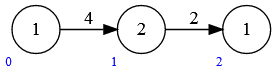}. 
\end{center}
The coordinates on $M_Q$ given by \S \ref{sec:coords} are the maximal minors of the matrices
\[\begin{bmatrix}
x_{11} & x_{12} & x_{13} & x_{14} \\
x_{21} & x_{22} & x_{23} & x_{24}\\
\end{bmatrix},
\hspace{5mm}
\begin{bmatrix}
x_{11} & x_{12} & x_{13} & x_{14} & 0 & 0 & 0 &0 \\
x_{21} & x_{22} & x_{23} & x_{24} & 0 & 0 & 0 &0\\
0 & 0 & 0 & 0 &x_{11} & x_{12} & x_{13} & x_{14} \\
0 & 0 & 0 & 0 & x_{21} & x_{22} & x_{23} & x_{24} \\
z_{11} & z_{12} & z_{13} & z_{14} & z_{15} & z_{16} & z_{17} & z_{18} \\
z_{21} & z_{22} & z_{23} & z_{24} & z_{25} & z_{26} & z_{27} & z_{28} \\
z_{31} & z_{32} & z_{33} & z_{34} & z_{35} & z_{36} & z_{37} & z_{38}\\
\end{bmatrix}.\]
These minors define $M_Q$ as a subvariety of $\mathbb{P}^5 \times \mathbb{P}^6$. Define a monomial ordering induced by
$$z_{1 i_1}>x_{1 i_2}>z_{2 i_3}>x_{2 i_4}>z_{3 i_5}$$
and the lex ordering within the rows. This defines a toric degeneration of $M_Q$; the associated Laurent polynomial has correct period sequence.  One can also describe the degeneration as a subvariety of the ladder diagram
\[\begin{tikzpicture}[scale=0.6]
\draw (0,0) rectangle (1,1);
\draw (1,1) rectangle (2,2);
\draw (0,1) rectangle (1,2);
\draw (1,0) rectangle (2,1);
\draw (2,0) rectangle (3,1);
\draw (3,0) rectangle (4,1);
\draw (4,0) rectangle (5,1);
\draw (5,0) rectangle (6,1);
\draw (6,0) rectangle (7,1);

\draw[fill] (0,0) circle (3pt);
\draw[fill] (0,1) circle (3pt);
\draw[fill] (1,1) circle (3pt);
\draw[fill] (2,1) circle (3pt);
\draw[fill] (2,2) circle (3pt);
\draw[fill] (7,1) circle (3pt);
\end{tikzpicture}\]
Notice that the ladder diagram is the ladder diagram for the quiver flag variety
\begin{center}
\includegraphics[scale=0.5]{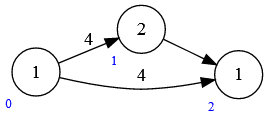}. 
\end{center}
$M_Q$ is a subvariety of this quiver flag variety cut out by a section of $S_1^* \otimes W_2$.
\[\begin{tikzpicture}[scale=0.6]
\draw (0,0) rectangle (1,1);
\draw (1,1) rectangle (2,2);
\draw (0,1) rectangle (1,2);
\draw (1,0) rectangle (2,1);
\draw (2,0) rectangle (3,1);
\draw (3,0) rectangle (4,1);
\draw (4,0) rectangle (5,1);
\draw (5,0) rectangle (6,1);
\draw (6,0) rectangle (7,1);

\draw[fill] (0,0) circle (3pt);
\draw[fill] (0,1) circle (3pt);
\draw[fill] (1,1) circle (3pt);
\draw[fill] (2,1) circle (3pt);
\draw[fill] (2,2) circle (3pt);
\draw[fill] (7,1) circle (3pt);

\end{tikzpicture}\]
To describe the subvariety of the ladder diagram, recall that each arrow in the corresponding ladder quiver determines a variable in $S_Q$. We label the vertices in the ladder diagram by their Cartesian coordinates, so that the source is at $(0,0)$. We draw the relevant arrows on the diagram below, and label them in the text for further clarity. 
\[\begin{tikzpicture}[scale=0.6]
\draw (0,0) rectangle (1,1);
\draw (1,1) rectangle (2,2);
\draw (0,1) rectangle (1,2);
\draw (1,0) rectangle (2,1);
\draw (2,0) rectangle (3,1);
\draw (3,0) rectangle (4,1);
\draw (4,0) rectangle (5,1);
\draw (5,0) rectangle (6,1);
\draw (6,0) rectangle (7,1);

\draw[fill] (0,0) circle (3pt);
\draw[fill] (0,1) circle (3pt);
\draw[fill] (1,1) circle (3pt);
\draw[fill] (2,1) circle (3pt);
\draw[fill] (2,2) circle (3pt);
\draw[fill] (7,1) circle (3pt);
\draw[yellow, line width = 1mm] (0,0) -- (5,0) -- (5,1) -- (7,1);
\node[] at (4.8,0.5) {\tiny $x_1$};
\draw[yellow, line width = 1mm] (0,0) -- (6,0) -- (6,1) -- (7,1);
\node[] at (5.8,0.5) {\tiny $x_2$};
\draw[yellow, line width = 1mm] (0,0) -- (7,0) -- (7,1) -- (7,1);
\node[] at (6.8,0.5) {\tiny $x_3$};
\draw[red, line width = 1mm] (0,1) -- (0,2) -- (2,2);
\node[] at (-0.2,1.5) {\tiny $y_1$};
\draw[green, line width = 1mm] (1,1) -- (1,1.9) -- (2,1.9);
\node[] at (0.8,1.5) {\tiny $y_2$};
\draw[blue, line width = 1mm] (0,1) -- (1,1);
\node[] at (0.5,0.8) {\tiny $y_3$};
\draw[orange, line width = 1mm] (2,1) -- (1,1);
\node[] at (1.5,0.8) {\tiny $y_4$};
\draw[purple, line width = 1mm] (2,1) -- (2,2);
\node[] at (1.8,1.5) {\tiny $y_5$};
\end{tikzpicture}\]
Label three of the paths from $(0,0)$ to $(7,1)$ with  variables $x_1, x_2, x_3$ (these arrows are marked in yellow on the above diagram). Label the arrow from from $(0,1)$ to $(2,2)$ as $y_1$ (in red), the arrow from $(1,1)$ to $(2,2)$ as $y_2$ (in green), the arrow from $(0,1)$ to $(1,1)$ as $y_3$ (in blue), the arrow from $(1,1)$ to $(2,1)$ as $y_4$ (in orange), and the arrow from $(2,1)$ to $(2,2)$ as $y_5$ (in purple). 
Then the ideal determining the toric variety is given by the binomial relations
$$(x_1 y_2 y_3- y_1 x_2, x_1 y_3 y_4 y_5 -y_1 x_3, x_2 y_3 y_4 y_5 - x_3 y_2 y_3).$$
In other words, this identifies the right most two boxes with the uppermost two boxes. 

The quiver flag zero locus $X$ given by $M_Q$ and the bundle $W_2^{\oplus 3}$ has period sequence PID 29. Pulling back the divisors indicated by choosing three distinct paths from $(0,0)$ to $(7,1)$ in the ladder diagram result in the following candidate Laurent polynomial mirror with matching period sequence (up to 20 terms) to $X$:
$$x + y + z + w + w/z + 1/(y z) + z/(x w) + 1/(x w) + 1/(x y) + 1/(x y z).$$

\appendix
\section{Connection to the Pl\"ucker coordinate mirror of the Grasssmannian}\label{ap:plucker}
In this section, we explain how the splitting of vector bundles under the toric degeneration corresponds to a similar phenomenon that arises in the Pl\"ucker coordinate mirror of the Grassmannian (which we call $W_P$) of \cite{MarshRietsch}. They propose a mirror for the Grassmannian using the quantum cohomology ring of the dual Grassmannian. The mirrors produced by the method proposed in this paper for quiver flag zero loci are conjectural; even the fact that Hori--Vafa mirror of $X(Q)$ is a mirror of $M_Q$ when $M_Q$ is a flag variety is conjectural. The only case where this is known is for the Grassmannian; this is proved in \cite{MarshRietsch} as a corollary of statements shown about $W_P$. 

We briefly recall the construction of $W_P$. The quantum cohomology ring of the Grassmannian $\Gr(n,k)$ is generated by Schur polynomials in the Chern roots of the tautological quotient bundle $E$. The partitions allowed are those such that the corresponding Young tableau fit inside a $k\times n-k$ rectangle. For example, $c_1(W)=s_{\yng(1)}$ and $1=s_{\emptyset}$.  Multiplication is given by quantum Schubert calculus. 

A partition is called \emph{rectangular} if the associated Young tableau is a rectangle. There are $n$ rectangular partitions which are either maximally wide or maximally long (i.e. $k \times j$ rectangles or $j \times n-k$ rectangles). Let $\lambda$ be such a partition. If $\lambda$ is not the maximal rectangle, there is exactly one way to add a single box to the rectangle while fitting inside a $k \times n-k$ box: call this partition $\lambda'$. The following identity holds in quantum cohomology:
$$s_{\yng(1)}*s_{\lambda}=s_{\lambda'}.$$
If $\lambda$ is the maximal rectangle, then let $\lambda'$ be the partition coming from the $k-1\times n-k-1$ rectangle. Then we have
$$s_{\yng(1)}*s_{\lambda}=q s_{\lambda'}$$
where $q$ is the quantum parameter in quantum cohomology. 

Each of these $n$ equations gives a description of $s_{\yng(1)}$, up to localising the rectangular Schur polynomials. For example, for $\Gr(4,2)$, the following relations hold:
$$\frac{s_{\yng(1)}}{s_{\emptyset}}=\frac{s_{\yng(2,1)}}{s_{\yng(2)}}=\frac{s_{\yng(2,1)}}{s_{\yng(1,1)}}=\frac{q s_{\yng(1)}}{s_{\yng(2,2)}}.$$
 The Pl\"ucker coordinate mirror $W_P$ is the sum of these $n$ descriptions of $s_{\yng(1)}$, with each Schur polynomial reinterpreted as a Pl\"ucker coordinate. This is done by interpreting the Young tableau as tracing out a path inside the $k\times n-k$ grid of boxes and looking at where the vertical steps are -- this gives a subset of $\{1,\dots,n\}$ which corrresponds to a Pl\"ucker coordinate in the usual way. For example, $s_{\yng(2,1)}$ corresponds to $p_{13}$ on $\Gr(4,2)$. The mirror for $\Gr(4,2)$ is thus
$$W_P=\frac{p_{\yng(1)}}{p_{\emptyset}}+\frac{p_{\yng(2,1)}}{p_{\yng(2)}}+\frac{p_{\yng(2,1)}}{p_{\yng(1,1)}}+\frac{q p_{\yng(1)}}{p_{\yng(2,2)}}.$$

Given a cluster seed (which in particular is a collection of algebraically independent elements of the coordinate ring), one can write $W_P$ as a Laurent polynomial in the elements of the seed. There is a distinguished seed consisting of the rectangle Pl\"ucker coordinates. In this seed, $W_P$ agrees with Eguchi--Hori--Xiong \cite{eguchi} mirror $W_{EHX}$ for the Grassmannian. The EHX mirror is the Laurent polynomial mirror of the Gelfand--Cetlin toric degeneration of the Grassmannian. One could write down the EHX mirror (in a different basis) by considering the Hori--Vafa mirror to $X(Q)$, where $\Gr(n,k)=M(Q,(1,k))$, but the usual description is the head over tails version, which we now explain.

To construct this, draw a dual ladder diagram, pictured below in blue for $\Gr(4,2)$.
\[\begin{tikzpicture}[scale=0.6]
\draw (0,0) rectangle (1,1);
\draw (1,0) rectangle (2,1);
\draw (0,1) rectangle (1,2);
\draw (1,1) rectangle (2,2);
\draw[fill] (0,0) circle (3pt);
\draw[fill] (1,1) circle (3pt);
\draw[fill] (2,2) circle (3pt);
\draw[fill,blue] (1.5,1.5) circle (2pt);
\draw[fill,blue] (0.5,0.5) circle (2pt);
\draw[fill,blue] (1.5,0.5) circle (2pt);
\draw[fill,blue] (0.5,1.5) circle (2pt);
\draw[fill,blue] (0.5,2.5) circle (2pt);
\draw[fill,blue] (2.5,0.5) circle (2pt);

\draw[blue,->] (0.5,0.5) -- (1.4,0.5);
\draw[blue,<-] (0.5,0.6) -- (0.5,1.5);
\draw[blue,<-] (1.5,0.6) -- (1.5,1.5);
\draw[blue,->] (0.5,1.5) -- (1.4,1.5);
\draw[blue,<-] (2.4,0.5) -- (1.5,0.5);
\draw[blue,->] (0.5,2.5) -- (0.5,1.6);
\end{tikzpicture}\]
For a general Grassmannian, this is formed by putting a vertex in each square of the ladder diagram, and drawing vertical and horizontal arrows between the vertices oriented down and left. There are two extra arrows, in the top left corner and the bottom right, oriented as in the example. Notice that each arrow in the ladder quiver is crossed precisely once by an arrow in the dual ladder diagram -- the two sets of arrows are in 1:1 correspondence. Note that the number of internal vertices is precisely the dimension of the Grassmannian. The head over tails mirror is given by assigning a variable to each internal vertex in the quiver, and setting the top external vertex to $1$ and the right one to $q$. The mirror is then 
$$W_{EHX}=\sum_{arrows} \frac{z_{t(a)}}{z_{s(a)}}.$$
\begin{rem} For a Fano $Y$-shaped quiver, a dual ladder diagram can also be drawn; this gives an alternative description of the mirror of $X(Q)$ as the heads over tail mirror of this dual ladder diagram. Note, in particular, that the number of boxes in the ladder diagram is the dimension of $M_Q$.
\end{rem}
We can naturally label the vertices as $z_{ij}$ where $1 \leq i \leq k$ and $1 \leq j \leq n-k$. To see that this construction agrees with the Hori--Vafa mirror (up to a $\GL(\ZZ)$ transformation, and setting $q=1$) is precisely the check that the two descriptions of $X(Q)$ (the quiver description and the one given in \cite{flagdegenerations}) agree.

The identification between the EHX mirror and Pl\"ucker coordinate mirror $W_P$ is then done as follows. Set $z_{i,j}=\frac{p_{i \times j}}{p_{i-1 \times j-1}}$, where $p_{i\times j}$ is the Pl\"ucker coordinate corresponding to the partition associated with the $i\times j$ rectangle, and setting $p_{\emptyset}=1$ as before. In \cite{MarshRietsch}, they show that under this change of coordinates $W_{EHX}$ is precisely the expansion of the Pl\"ucker coordinate mirror in the rectangles chart. In fact, this identification can be done monomial by monomial in $W_P$, by looking at columns of horizontal arrow and rows of vertical arrows (this will be explained in a forthcoming paper by Rietsch and Williams). There are $n-k-1$ columns of horizontal arrows in the dual ladder diagram. For example, in the above example, there is one column: those arrows in red below.
\[\begin{tikzpicture}[scale=0.6]
\draw (0,0) rectangle (1,1);
\draw (1,0) rectangle (2,1);
\draw (0,1) rectangle (1,2);
\draw (1,1) rectangle (2,2);
\draw[fill] (0,0) circle (3pt);
\draw[fill] (1,1) circle (3pt);
\draw[fill] (2,2) circle (3pt);
\draw[fill,blue] (1.5,1.5) circle (2pt);
\draw[fill,blue] (0.5,0.5) circle (2pt);
\draw[fill,blue] (1.5,0.5) circle (2pt);
\draw[fill,blue] (0.5,1.5) circle (2pt);
\draw[fill,blue] (0.5,2.5) circle (2pt);
\draw[fill,blue] (2.5,0.5) circle (2pt);

\draw[red,->] (0.5,0.5) -- (1.4,0.5);
\draw[blue,<-] (0.5,0.6) -- (0.5,1.5);
\draw[blue,<-] (1.5,0.6) -- (1.5,1.5);
\draw[red,->] (0.5,1.5) -- (1.4,1.5);
\draw[blue,<-] (2.4,0.5) -- (1.5,0.5);
\draw[blue,->] (0.5,2.5) -- (0.5,1.6);
\end{tikzpicture}\]
 There are $k-1$ rows of vertical arrows (in the above example, there is only one row, consisting of the two internal blue arrows). The sum of monomials in a row or column corresponding to arrows in a vertical column can be simplified via the Pl\"ucker relations to exactly one of the monomials in $W_P$. The remaining two terms of $W_P$ come from the arrows associated to the external vertices. In this example, summing over the single column of horizontal arrows we see that
$$\frac{p_{\yng(2)}}{p_{\yng(1)}}+\frac{p_{\yng(2,2)}}{p_{\yng(1)}p_{\yng(1,1)}}=\frac{p_{\yng(2,1)}}{p_{\yng(1,1)}}.$$

We have seen that under the Gelfand--Cetlin toric degeneration, the tautological quotient bundle $E$ corresponds to sets of $k$ rank $1$ reflexive sheaves. For example, each path from $0$ to $v_1$ with $k$ consecutive vertical steps through the interior of the ladder diagram gives such a set. There are $n-k-1$ such paths in the ladder diagram, call them $P_1,\dots,P_{n-k-1}$.  The tensor product of these sheaves is $c_1(L_1)$ by construction, which under the degeneration is exactly $c_1(E)$.  Here, we see that in the rectangles chart, precisely $n-k-1$ of the monomials in $W_P$  split into $k$ monomials. That is, we obtain a way of writing  $c_1(E)=s_{\yng(1)}$ as $k$ terms. Moreover, each of splittings of $s_{\yng(1)}$ corresponds a column of horizontal arrows in the dual ladder diagram. Each of these horizontal arrows correspond to an arrow in the ladder quiver, and together these trace out a path in the ladder quiver. The paths which arise in this manner are precisely the $P_i$. 
\section{Laurent polynomial mirrors} \label{ap:table}
In the table below, we record Laurent polynomial mirrors for 99 of the 141 Fano quiver flag zero loci found in \cite{kalashnikov}. The Period ID refers to the indexing system in that paper -- descriptions of the Fano varieties can be found there.
\begin{landscape}
\begin{longtable}{c C{18cm}}
\caption{Mirrors to some four dimensional Fano quiver flag zero loci}
\label{tab:laurent}\\
\toprule
\multicolumn{1}{c}{Period ID}&\multicolumn{1}{C{18cm}}{Laurent polynomial mirror}\\
\midrule
\endfirsthead
\multicolumn{2}{l}{\tiny Continued from previous page.}\\
\addlinespace[1.7ex]
\midrule
\multicolumn{1}{c}{Period ID}&\multicolumn{1}{c}{Laurent polynomial mirror}\\
\midrule
\endhead
\midrule
\multicolumn{2}{r}{\tiny Continued on next page.}\\
\endfoot
\bottomrule
\endlastfoot
\oddrow 20&$x + y + w + z + z/y + 1/(y w) + w/x + 1/(x z)$\\
\evnrow 25&$x + y + w + z + z/w + 1/(x z) + w/(x y) + w/(x y z) + 2/(x y) + w/(x^2 y^2 z)$\\
\oddrow 29&$x + y + w + w/z + z + 1/(y z) + z/(x w) + 1/(x w) + 1/(x y) + 1/(x y z)$\\
\evnrow 101&$x + x/y + y + w + z + 1/(w z) + 1/(y z) + 1/x$\\
\oddrow 102&$x + x/y + x/(y w z) + y + w + z + w/y + 2/(y z) + 1/x + w/(x y z)$\\
\evnrow 104&$x + y + w + z + y/(x w) + w/(x z) + 1/x + 1/(x z) + 1/(x w) + w/(x y z) + 1/(x y) + 1/(x y z)$\\
\oddrow 109&$x + y + w + z + z/w + w/y + z/y + 1/x + 1/(x z) + 1/(x w) + w/(x y z) + 1/(x y)$\\
\evnrow 115&$x + y w + y + w + z + 1/x + 1/(x z) + 1/(x w) + 1/(x y) + 1/(x y z) + z/(x y w) + 1/(x y w)$\\
\oddrow 116&$x + x/w + x/y + x/(y w) + y + w + z + w/(y z) + 2/(y z) + 1/(y w z) + 1/x$\\
\evnrow 126&$x^2/(y^2 w) + x + 2 x/y + x/(y w) + x/(y^2 w z) + y + w + z + 2/(y z) + 1/(y w z) + w/(x z) + 1/x + 1/(x z)$\\
\oddrow 142&$x + y + w + z + 1/y + 1/(y z) + 1/(y^2 w z) + 1/x + 1/(x y w z)$\\
\evnrow 154&$x + y + w + z + 1/y + 1/(y z) + z/x + 1/x + 1/(x y w) + 1/(x y w z)$\\
\oddrow 173&$x y/(w z) + x + y + w + z + 1/(w z) + 1/y + 1/x + 1/(x y)$\\
\evnrow 178&$x + y + y/z + w + z + 1/(w z) + 1/y + z/x + 1/x$\\
\oddrow 188&$x + y + w + z + w/y + 1/y + 1/x + 2/(x z) + 1/(x w z) + 2/(x y z) + 1/(x y w z) + 1/(x^2 w z^2) + 1/(x^2 y w z^2)$\\
\evnrow 190&$x + y + w + z + z/y + z^2/(y w) + z/(y w) + w/x + w/(x z) + z/x + 2/x + 1/(x z) + z/(x y) + 1/(x y) + z^2/(x y w) + 2 z/(x y w) + 1/(x y w)$\\
\oddrow 195&$x + y + w + z + w/y + y/(x z) + 2/x + 1/(x z) + z/(x w) + 2/(x y) + 1/(x^2 w) + 1/(x^2 y w)$\\
\evnrow 201&$x + y + w + z + z/y + 1/y + 1/(y w) + w/x + 1/x + 1/(x z) + 1/(x y) + 1/(x y w z)$\\
\oddrow 202&$x + y + w + z + 1/(w z) + 1/y + 1/(y w z) + y/x + 1/x + 1/(x z) + 1/(x y z)$\\
\evnrow 211&$x + y + w z + w + z + z/y + 1/y + w z/x + w/x + 1/x + 1/(x z) + 2 z/(x y) + 2/(x y) + 1/(x y w) + 1/(x y w z) + z/(x y^2 w) + 1/(x y^2 w)$\\
\oddrow 212&$x + y + w + z + w/(y z) + 2/(y z) + 1/(y w z) + w/x + 2/x + 1/(x w) + w/(x y z) + 2/(x y z) + 1/(x y w z)$\\
\evnrow 219&$x + y + w + w/z + z + w/(y z) + 1/y + y/(x w) + 1/x + 1/(x z) + 2/(x w) + 1/(x y z) + 1/(x y w)$\\
\oddrow 224&$x + x/y + y + w + z + 1/y + 1/(y z) + 1/x + 1/(x z) + 1/(x w) + 1/(x y w z) + 1/(x^2 w z)$\\
\evnrow 227&$x + y + w z + w + z + z/y + 1/y + 1/(y w) + 1/x + 1/(x z) + 1/(x w z) + z/(x y) + 1/(x y) + 1/(x y w)$\\
\oddrow 230&$x + y + w + z + 1/y + y/(x z) + y/(x w) + 1/x + 2/(x z) + 2/(x w) + 1/(x y z) + 1/(x y w) + y/(x^2 w z) + 2/(x^2 w z) + 1/(x^2 y w z)$\\
\evnrow 232&$x + y + w + z + w/y + 1/y + 1/x + 1/(x z) + 1/(x w) + 1/(x w z) + 2/(x y) + 1/(x y z) + 1/(x y w) + 1/(x y w z) + 1/(x^2 w z) + 1/(x^2 y w) + 2/(x^2 y w z) + 1/(x^2 y w^2 z) + 1/(x^3 y w^2 z)$\\
\oddrow 238&$x w z + x + y + w + z + 1/y + y/(x z) + y/(x w) + 1/x + 2/(x z) + 2/(x w) + 1/(x y z) + 1/(x y w)$\\
\evnrow 244&$x + x/z + x/w + x/(y w) + y w/z + y + w + w/z + z + 1/y + y w/(x z) + y/x + 1/x$\\
\oddrow 253&$x + y + w + z + w z/y + w/y + y z/(x w) + 2 y/(x w) + y/(x w z) + z/x + 2/x + 1/(x z) + z/(x w) + 2/(x w) + 1/(x w z) + z/(x y) + 2/(x y) + 1/(x y z)$\\
\evnrow 288&$x + y + w + w/z + z + 1/y + 1/(y z) + 2/x + 2/(x z) + 1/(x y w) + 1/(x y w z) + 1/(x^2 w) + 1/(x^2 w z)$\\
\oddrow 294&$x + y + w + w/z + z + z/y + 1/y + y/x + 2/x + 2/(x z) + y/(x^2 w) + 1/(x^2 w) + 1/(x^2 w z)$\\
\evnrow 297&$x + y + z + 1/z + 1/y + 1/(y w) + w z/x + w/x + z/x + 1/x + w z/(x y) + w/(x y) + 2 z/(x y) + 2/(x y) + z/(x y w) + 1/(x y w)$\\
\oddrow 303&$x + y + y/z + w + w/z + 1/z + 1/w + 1/y + 1/(y w) + w z/x + w/x + z/x + 1/x$\\
\evnrow 318&$x + y + w + 1/z + 1/w + 1/(w z) + z/y + 2/y + 1/(y z) + z/(y w) + 2/(y w) +1/(y w z) + w z/(x y) + w/(x y) + z/(x y) + 1/(x y)$\\
\oddrow 321&$x + y + w/z + 1/z + z/w + 1/w + w/(y z) + 2/y + 2/(y z) + z/(y w) + 2/(y w) + 1/(y w z) + w^2/(x y z) + 2 w/(x y) + 2 w/(x y z) + z/(x y) + 2/(x y) + 1/(x y z)$\\
\evnrow 327&$x + y + w z + w + z + 1/y + 1/(y w) + w z/x + 2 w/x + w/(x z) + z/x + 2/x + 1/(x z) + w z/(x y) + 2 w/(x y) + w/(x y z) + 2 z/(x y) + 3/(x y) + 1/(x y z) + z/(x y w) + 1/(x y w)$\\
\oddrow 341&$x + y + w/z + 1/z + w^2/(y z) + 2 w/y + 2 w/(y z) + z/y + 2/y + 1/(y z) + 1/x + z/(x w) + 1/(x w) + w^2/(x y z) + 2 w/(x y) + 3 w/(x y z) + z/(x y) + 4/(x y) + 3/(x y z) + z/(x y w) + 2/(x y w) + 1/(x y w z)$\\
\evnrow 342&$x + y + z + 1/z + w z/y + w/y + z/y + 1/y + 1/x + 1/(x z) + 1/(x w) + 1/(x w z) + w z/(x y) + 2 w/(x y) + w/(x y z) + 2 z/(x y) + 4/(x y) + 2/(x y z) + z/(x y w) + 2/(x y w) + 1/(x y w z)$\\
\oddrow 350&$x + w + z + w/y + z/y + 1/y + y w/(x z) + 2 y/x + y/(x z) + y z/(x w) + y/(x w) + w/(x z) + 3/x + 1/(x z) + 2 z/(x w) + 2/(x w) + 1/(x y) + z/(x y w) + 1/(x y w)$\\
\evnrow 361&$x + y w/z + 2 y + y/z + y z/w + y/w + w/z + 1/z + z/w + 1/w + w^2/(x z) + 2 w/x + 2 w/(x z) + z/x + 2/x + 1/(x z) + w^2/(x y z) + 2 w/(x y) + 2 w/(x y z) + z/(x y) + 2/(x y) + 1/(x y z)$\\
\oddrow 385&$x + y + 1/z + z/y + 2/y + 1/(y z) + z/(y w) + 2/(y w) + 1/(y w z) + w z/x + w/x + 2 z/x + 2/x + z/(x w) + 1/(x w)$\\
\evnrow 403&$x + y + w/z + 1/z + z/w + 1/w + 1/y + y w/x + y z/x + 2 y/x + y z/(x w) + y/(x w) + w/x + z/x + 2/x + z/(x w) + 1/(x w)$\\
\oddrow 404&$x + y + w/z + 1/z + z/w + 1/w + w/y + z/y + 2/y + z/(y w) + 1/(y w) + w/x + z/x + 1/x$\\
\evnrow 406&$x + y w + y z + y + w + z + 1/z + 1/w + 1/(w z) + 1/y + z/x + 1/x + z/(x y) + 1/(x y)$\\
\oddrow 413&$x + y w + y + w + z + 1/z + 1/y + z/x + 2/x + 1/(x z) + z/(x w) + 2/(x w) + 1/(x w z) + z/(x y) + 2/(x y) + 1/(x y z)$\\
\evnrow 421&$x + y w + y + w + z + 1/z + y w/x + y w/(x z) + 2 y/x + 2 y/(x z) + y/(x w) + y/(x w z) + w/x + w/(x z) + 3/x + 3/(x z) + 2/(x w) + 2/(x w z) + 1/(x y) + 1/(x y z) + 1/(x y w) + 1/(x y w z)$\\
\oddrow 422&$x + y z + y + z + 1/z + 1/w + 1/y + 1/(y w) + y z/x + 2 y/x + y/(x z) + w z/x + 2 w/x + w/(x z) + z/x + 2/x + 1/(x z)$\\
\evnrow 423&$x + y + y/z + y/w + y/(w z) + 1/z + 1/w + 1/(w z) + 1/y + y w z/x + y w/x + 2 y z/x + 3 y/x + y/(x z) + y z/(x w) + 2 y/(x w) + y/(x w z) + w z/x + w/x + 2 z/x + 3/x + 1/(x z) + z/(x w) + 2/(x w) + 1/(x w z)$\\
\oddrow 424&$x + w + w/z + z + 1/z + z/w + 1/w + w/(y z) + 2/y + 1/(y z) + z/(y w) + 1/(y w) + y w/(x z) + y/x + y/(x z) + w/(x z) + 1/x + 1/(x z)$\\
\evnrow 430&$x + y z + y + z + 1/z + 1/w + 1/y + 1/(y w) + w/x + w/(x z) + 2/x + 2/(x z) + 1/(x w) + 1/(x w z) + w/(x y) + w/(x y z) + 2/(x y) + 2/(x y z) + 1/(x y w) + 1/(x y w z)$\\
\oddrow 439&$x + y w z + y w + y z + y + w z + w + z + y w z/x + y w/x + y z/x + y/x + 2 w z/x + 2 w/x + 3 z/x + 4/x + 1/(x z) + z/(x w) + 2/(x w) + 1/(x w z) + w z/(x y) + w/(x y) + 2 z/(x y) + 3/(x y) + 1/(x y z) + z/(x y w) + 2/(x y w) + 1/(x y w z)$\\
\evnrow 464&$x + y + y/w + w + z + w z/y + 2 w/y + w/(y z) + z/y + 2/y + 1/(y z) + 1/x + 1/(x z) + 1/(x w) + 1/(x w z)$\\
\oddrow 473&$x + z/w + 1/w + 1/y + z/(y w) + 1/(y w) + y w/(x z) + y/x + y/(x z) + w^2/(x z) + 3 w/x + 3 w/(x z) + 3 z/x + 5/x + 2/(x z) + z^2/(x w) + 2 z/(x w) + 1/(x w) + w^2/(x y z) + 3 w/(x y) + 2 w/(x y z) + 3 z/(x y) + 4/(x y) + 1/(x y z) + z^2/(x y w) + 2 z/(x y w) + 1/(x y w)$\\
\evnrow 478&$x + y + y/z + w z + w + 2 z + 1/z + z/w + 1/w + 1/y + y/x + y/(x z) + y/(x w) + y/(x w z) + 1/x + 1/(x z) + 1/(x w) + 1/(x w z)$\\
\oddrow 483&$x + y z + 2 y + y/z + w z + 2 w + w/z + z + 1/z + 1/y + y/(x w) + y/(x w z) + 2/x + 2/(x z) + 1/(x w) + 1/(x w z) + w/(x y) + w/(x y z) + 1/(x y) + 1/(x y z)$\\
\evnrow 508&$x + y w z + 2 y w + y w/z + y z + 3 y + 2 y/z + y/w + y/(w z) + w z + 2 w + w/z + z + 2/z + 1/w + 1/(w z) + 1/x + 1/(x y)$\\
\oddrow 509&$x + y + z/w + 1/w + w^2/(y z) + 3 w/y + 2 w/(y z) + 3 z/y + 4/y + 1/(y z) + z^2/(y w) + 2 z/(y w) + 1/(y w) + w^2/(x z) + 2 w/x + 2 w/(x z) + z/x + 2/x + 1/(x z)$\\
\evnrow 512&$x + y + y/w + w + z + 1/z + 1/w + y z/x + 2 y/x + y/(x z) + w z/x + 2 w/x + w/(x z) + 2 z/x + 4/x + 2/(x z) + w z/(x y) + 2 w/(x y) + w/(x y z) + z/(x y) + 2/(x y) + 1/(x y z)$\\
\oddrow 516&$x + y w + y z + 2 y + y z/w + y/w + w + w/z + z + 1/z + z/w + 1/w + 1/y + 1/x + z/(x w) + 1/(x w) + 1/(x y) + z/(x y w) + 1/(x y w)$\\
\evnrow 517&$x + y + y/z + w z + w + z + 1/z + y w z/x + 2 y w/x + y w/(x z) + y z/x + 2 y/x + y/(x z) + 2 w z/x + 4 w/x + 2 w/(x z) + 3 z/x + 5/x + 2/(x z) + z/(x w) + 1/(x w) + w z/(x y) + 2 w/(x y) + w/(x y z) + 2 z/(x y) + 3/(x y) + 1/(x y z) + z/(x y w) + 1/(x y w)$\\
\oddrow 519&$x + y + y/w + y/(w z) + w + z + 1/z + 1/w + 1/(w z) + y^2/(x w) + y^2/(x w z) + 2 y/x + 2 y/(x z) + 3 y/(x w) + 3 y/(x w z) + w/x + w/(x z) + 4/x + 4/(x z) + 3/(x w) + 3/(x w z) + w/(x y) + w/(x y z) + 2/(x y) + 2/(x y z) + 1/(x y w) + 1/(x y w z)$\\
\evnrow 521&$x + y + w + z + z/w + 1/w + 1/y + z/(y w) + 1/(y w) + w^2/(x z) + 3 w/x + 2 w/(x z) + 3 z/x + 4/x + 1/(x z) + z^2/(x w) + 2 z/(x w) + 1/(x w) + w^2/(x y z) + 3 w/(x y) + 2 w/(x y z) + 3 z/(x y) + 4/(x y) + 1/(x y z) + z^2/(x y w) + 2 z/(x y w) + 1/(x y w)$\\
\oddrow 524&$x + y + 2 w + z + 1/z + w^2/y + 2 w/y + 1/y + y^2 z/(x w) + y^2/(x w) + 3 y z/x + 3 y/x + 2 y z/(x w) + 2 y/(x w) + 3 w z/x + 3 w/x + 4 z/x + 4/x + z/(x w) + 1/(x w) + w^2 z/(x y) + w^2/(x y) + 2 w z/(x y) + 2 w/(x y) + z/(x y) + 1/(x y)$\\
\evnrow 527&$x + y w z + 2 y w + y w/z + y z + 3 y + 2 y/z + y/w + y/(w z) + w z + 2 w + w/z + z + 2/z + 1/w + 1/(w z) + 1/x + 1/(x z) + 1/(x w) + 1/(x w z) + 1/(x y) + 1/(x y z) + 1/(x y w) + 1/(x y w z)$\\
\oddrow 529&$x + y + y z/w + y/w + 2 z/w + 2/w + 1/y + z/(y w) + 1/(y w) + y w^2/(x z) + 3 y w/x + 2 y w/(x z) + 3 y z/x + 4 y/x + y/(x z) + y z^2/(x w) + 2 y z/(x w) + y/(x w) + w^2/(x z) + 3 w/x + 2 w/(x z) + 3 z/x + 5/x + 1/(x z) + z^2/(x w) + 3 z/(x w) + 2/(x w) + 1/(x y) + z/(x y w) + 1/(x y w)$\\
\evnrow 549&$x + y w + y + w + z + 1/z + y w z/x + 2 y w/x + y w/(x z) + 2 y z/x + 4 y/x + 2 y/(x z) + y z/(x w) + 2 y/(x w) + y/(x w z) + w z/x + 2 w/x + w/(x z) + 3 z/x + 6/x + 3/(x z) + 2 z/(x w) + 4/(x w) + 2/(x w z) + z/(x y) + 2/(x y) + 1/(x y z) + z/(x y w) + 2/(x y w) + 1/(x y w z)$\\
\oddrow 552&$x y w z + x y w + x y z + x y + x w z + x w + x z + x + y w z + y w + y z + y + w z + w + z + 1/z + 1/w + 1/y + 1/(y z) + 1/x + 1/(x z) + 1/(x w)$\\
\evnrow 555&$x + y + y/z + y/w + w + w/z + z + 1/z + z/w + 1/w + y^2/(x z) + 2 y w/(x z) + 2 y/x + 3 y/(x z) + w^2/(x z) + 2 w/x + 4 w/(x z) + z/x + 4/x + 3/(x z) + w^2/(x y z) + 2 w/(x y) + 2 w/(x y z) + z/(x y) + 2/(x y) + 1/(x y z)$\\
\oddrow 558&$x + y + w + 2 z + z^2/w + 2 z/w + 1/w + w^2/(y z) + 3 w/y + 2 w/(y z) + 3 z/y + 4/y + 1/(y z) + z^2/(y w) + 2 z/(y w) + 1/(y w) + w^2/(x z) + 2 w/x + 2 w/(x z) + z/x + 2/x + 1/(x z)$\\
\evnrow 559&$x + y w + y z + y + w + w/z + z + 1/z + y w^2/(x z) + 2 y w/x + 3 y w/(x z) + y z/x + 4 y/x + 3 y/(x z) + y z/(x w) + 2 y/(x w) + y/(x w z) + w^2/(x z) + 2 w/x + 4 w/(x z) + z/x + 6/x + 5/(x z) + 2 z/(x w) + 4/(x w) + 2/(x w z) + w/(x y z) + 2/(x y) + 2/(x y z) + z/(x y w) + 2/(x y w) + 1/(x y w z)$\\
\oddrow 562&$x + y + y z/w + y/w + w + 2 z + z^2/w + 3 z/w + 2/w + w/y + 2 z/y + 2/y + z^2/(y w) + 2 z/(y w) + 1/(y w) + w^2/(x z) + 2 w/x + 2 w/(x z) + z/x + 2/x + 1/(x z) + w^2/(x y z) + 2 w/(x y) + 2 w/(x y z) + z/(x y) + 2/(x y) + 1/(x y z)$\\
\evnrow 566&$x + y^2/w + 2 y + 2 y z/w + 2 y/w + w + 2 z + z^2/w + 2 z/w + 1/w + y^2/(x z) + y^2/(x w) + 2 y w/(x z) + 4 y/x + 3 y/(x z) + 2 y z/(x w) + 2 y/(x w) + w^2/(x z) + 3 w/x + 4 w/(x z) + 3 z/x + 6/x + 3/(x z) + z^2/(x w) + 2 z/(x w) + 1/(x w) + w^2/(x y z) + 2 w/(x y) + 2 w/(x y z) + z/(x y) + 2/(x y) + 1/(x y z)$\\
\oddrow 582&$x + y w + 2 y + y/w + w + z + 1/z + 1/w + z/y + 2/y + 1/(y z) + z/x + 2/x + 1/(x z) + z/(x w) + 2/(x w) + 1/(x w z) + z/(x y) + 2/(x y) + 1/(x y z) + z/(x y w) + 2/(x y w) + 1/(x y w z)$\\
\evnrow 583&$x y + x y/z + x y/w + x + x/z + x/w + y + y/z + y/w + w z + w + z + 1/z + 1/w + w z/y + w/y + z/y + 1/y + w z/x + w/x + z/x + 1/x$\\
\oddrow 593&$x + y w + y z + 2 y + y z/w + y/w + w + w/z + z + 1/z + 2 z/w + 2/w + 1/y + z/(y w) + 1/(y w) + w/(x z) + 1/x + 1/(x z)$\\
\evnrow 597&$x + y w z + y w + y z + y + w z + w + z + 1/z + 1/w + 1/(w z) + y w z/x + 2 y w/x + y w/(x z) + 2 y z/x + 4 y/x + 2 y/(x z) + y z/(x w) + 2 y/(x w) + y/(x w z) + 2 w z/x + 4 w/x + 2 w/(x z) + 3 z/x + 6/x + 3/(x z) + z/(x w) + 2/(x w) + 1/(x w z) + w z/(x y) + 2 w/(x y) + w/(x y z) + z/(x y) + 2/(x y) + 1/(x y z)$\\
\oddrow 600&$x + y + z + 1/z + z/w + 2/w + 1/(w z) + z/y + 2/y + 1/(y z) + z/(y w) + 2/(y w) + 1/(y w z) + y z/x + 2 y/x + y/(x z) + y z/(x w) + 2 y/(x w) + y/(x w z) + w z/x + 2 w/x + w/(x z) + 3 z/x + 6/x + 3/(x z) + 2 z/(x w) + 4/(x w) + 2/(x w z) + w z/(x y) + 2 w/(x y) + w/(x y z) + 2 z/(x y) + 4/(x y) + 2/(x y z) + z/(x y w) + 2/(x y w) + 1/(x y w z)$\\
\evnrow 603&$x + y w z + 2 y w + y w/z + y z + 3 y + 2 y/z + y/w + y/(w z) + w z + 2 w + w/z + z + 2/z + 1/w + 1/(w z) + 1/y + 1/x + 1/(x z) + 1/(x w) + 1/(x w z) + 1/(x y) + 1/(x y z) + 1/(x y w) + 1/(x y w z)$\\
\oddrow 604&$x + y w + y z + 2 y + y z/w + y/w + w + w/z + z + 1/z + z/w + 1/w + w/(y z) + 1/y + 1/(y z) + w/(x z) + 2/x + 2/(x z) + z/(x w) + 2/(x w) + 1/(x w z) + w/(x y z) + 2/(x y) + 2/(x y z) + z/(x y w) + 2/(x y w) + 1/(x y w z)$\\
\evnrow 616&$x y + x w + x + 2 y + 2 w + z + 1/z + z/w + 2/w + 1/(w z) + z/y + 2/y + 1/(y z) + z/(y w) + 2/(y w) + 1/(y w z) + y/x + w/x + 1/x$\\
\oddrow 617&$x y + x y/z + x y/w + 2 x + 2 x/z + x/w + x/y + x/(y z) + 2 y + 2 y/z + 2 y/w + w z + w + z + 3/z + 2/w + 1/y + 1/(y z) + y/x + y/(x z) + y/(x w) + 1/x + 1/(x z) + 1/(x w)$\\
\evnrow 623&$x w + x z + 2 x + x z/w + x/w + x/y + y w + y z + 2 y + y z/w + y/w + w + w/z + z + 1/z + 2 z/w + 2/w + 1/y + y/x + y z/(x w) + y/(x w) + 1/x + z/(x w) + 1/(x w)$\\
\oddrow 625&$x w z + x w + x z + x + y w z + y w + y z + y + w z + w + z + 1/z + 1/w + 1/(w z) + 1/y + 1/(y z) + 1/(y w) + 1/(y w z) + 1/x + 1/(x z) + 1/(x w) + 1/(x w z) + 1/(x y) + 1/(x y z) + 1/(x y w) + 1/(x y w z)$\\
\evnrow 626&$x + y + y z/w + 2 y/w + y/(w z) + w + z + 1/z + z/w + 2/w + 1/(w z) + y^2 z/(x w) + 2 y^2/(x w) + y^2/(x w z) + 2 y z/x + 4 y/x + 2 y/(x z) + 3 y z/(x w) + 6 y/(x w) + 3 y/(x w z) + w z/x + 2 w/x + w/(x z) + 4 z/x + 8/x + 4/(x z) + 3 z/(x w) + 6/(x w) + 3/(x w z) + w z/(x y) + 2 w/(x y) + w/(x y z) + 2 z/(x y) + 4/(x y) + 2/(x y z) + z/(x y w) + 2/(x y w) + 1/(x y w z)$\\
\oddrow 637&$x y + x w + x z + 2 x + x w/y + x z/y + x/y + y + y/z + y/w + w + w/z + z + 1/z + z/w + 1/w + 2 w/y + 2 z/y + 2/y + 1/x + w/(x y) + z/(x y) + 1/(x y)$\\
\evnrow 642&$x + y^2/z + y^2/w + 2 y w/z + 4 y + 2 y/z + 2 y z/w + 2 y/w + w^2/z + 3 w + 2 w/z + 3 z + 1/z + z^2/w + 2 z/w + 1/w + y/(x z) + 2 w/(x z) + 2/x + 2/(x z) + w^2/(x y z) + 2 w/(x y) + 2 w/(x y z) + z/(x y) + 2/(x y) + 1/(x y z)$\\
\oddrow 653&$x w + x z + 2 x + x z/w + x/w + x w/(y z) + x/y + x/(y z) + y w + y z + 2 y + y z/w + y/w + w + w/z + z + 1/z + 2 z/w + 2/w + w/(y z) + 1/y + 1/(y z) + y/x + y z/(x w) + y/(x w) + 1/x + z/(x w) + 1/(x w)$\\
\evnrow 662&$x y w + x y z + x y + x w + 2 x z + 2 x + x z/y + x/y + 2 y w + 2 y z + 2 y + 2 w + 4 z + 1/z + 1/w + 1/(w z) + 2 z/y + 2/y + y w/x + y z/x + y/x + w/x + 2 z/x + 2/x + z/(x y) + 1/(x y)$\\
\oddrow 664&$x + y^2/(w z) + y + 2 y/z + 2 y/w + 2 y/(w z) + w + w/z + z + 2/z + z/w + 2/w + 1/(w z) + y^3/(x w z) + 3 y^2/(x z) + 3 y^2/(x w) + 4 y^2/(x w z) + 3 y w/(x z) + 6 y/x + 9 y/(x z) + 3 y z/(x w) + 9 y/(x w) + 6 y/(x w z) + w^2/(x z) + 3 w/x + 6 w/(x z) + 3 z/x + 12/x + 9/(x z) + z^2/(x w) + 6 z/(x w) + 9/(x w) + 4/(x w z) + w^2/(x y z) + 3 w/(x y) + 3 w/(x y z) + 3 z/(x y) + 6/(x y) + 3/(x y z) + z^2/(x y w) + 3 z/(x y w) + 3/(x y w) + 1/(x y w z)$\\
\evnrow 665&$x + y + y/w + w + z + 2 z/w + 2/w + w/y + 2 z/y + 2/y + z^2/(y w) + 2 z/(y w) + 1/(y w) + y^2/(x z) + y^2/(x w) + 3 y w/(x z) + 6 y/x + 3 y/(x z) + 3 y z/(x w) + 3 y/(x w) + 3 w^2/(x z) + 9 w/x + 6 w/(x z) + 9 z/x + 12/x + 3/(x z) + 3 z^2/(x w) + 6 z/(x w) + 3/(x w) + w^3/(x y z) + 4 w^2/(x y) + 3 w^2/(x y z) + 6 w z/(x y) + 9 w/(x y) + 3 w/(x y z) + 4 z^2/(x y) + 9 z/(x y) + 6/(x y) + 1/(x y z) + z^3/(x y w) + 3 z^2/(x y w) + 3 z/(x y w) + 1/(x y w)$\\
\oddrow 669&$x y/z + x w/z + 2 x + x/z + x w/y + x z/y + x/y + y^2/w + 2 y + 2 y/z + 2 y z/w + 2 y/w + w + 2 w/z + 2 z + 2/z + z^2/w + 2 z/w + 1/w + w/y + z/y + 1/y + y/(x z) + w/(x z) + 1/x + 1/(x z)$\\
\evnrow 679&$x w + 2 x z + 2 x + x z^2/w + 2 x z/w + x/w + x w/(y z) + x/y + x/(y z) + 2 y w + 3 y z + 3 y + y z^2/w + 2 y z/w + y/w + 2 w + w/z + 3 z + 1/z + z^2/w + 2 z/w + 1/w + w/(y z) + 1/y + 1/(y z) + y^2 w/x + y^2 z/x + y^2/x + 2 y w/x + 2 y z/x + 2 y/x + w/x + z/x + 1/x$\\
\oddrow 699&$x y w/z + 2 x y + x y/z + x y z/w + x y/w + x w/z + 3 x + x/z + 2 x z/w + 2 x/w + x/y + x z/(y w) + x/(y w) + 2 y w/z + 4 y + 2 y/z + 2 y z/w + 2 y/w + w + 2 w/z + z + 2/z + 3 z/w + 3/w + 1/y + z/(y w) + 1/(y w) + y w/(x z) + 2 y/x + y/(x z) + y z/(x w) + y/(x w) + w/(x z) + 2/x + 1/(x z) + z/(x w) + 1/(x w)$\\
\evnrow 717&$x y w/z + 2 x y + x y/z + x y z/w + x y/w + 2 x w/z + 3 x + 2 x/z + x z/w + x/w + x w/(y z) + x/y + x/(y z) + 2 y w/z + 4 y + 2 y/z + 2 y z/w + 2 y/w + w + 4 w/z + z + 4/z + 2 z/w + 2/w + 2 w/(y z) + 2/y + 2/(y z) + y w/(x z) + 2 y/x + y/(x z) + y z/(x w) + y/(x w) + 2 w/(x z) + 3/x + 2/(x z) + z/(x w) + 1/(x w) + w/(x y z) + 1/(x y) + 1/(x y z)$\\
\oddrow 723&$x y^2/(w z) + 3 x y/z + 2 x y/w + 2 x y/(w z) + 3 x w/z + 4 x + 4 x/z + x z/w + 2 x/w + x/(w z) + x w^2/(y z) + 2 x w/y + 2 x w/(y z) + x z/y + 2 x/y + x/(y z) + 2 y^2/(w z) + y + 5 y/z + 4 y/w + 4 y/(w z) + w + 4 w/z + z + 6/z + 2 z/w + 4/w + 2/(w z) + w^2/(y z) + 2 w/y + 2 w/(y z) + z/y + 2/y + 1/(y z) + y^2/(x w z) + 2 y/(x z) + 2 y/(x w) + 2 y/(x w z) + w/(x z) + 2/x + 2/(x z) + z/(x w) + 2/(x w) + 1/(x w z)$\\
\evnrow 730&$x^2/y + x^2 z/(y w) + x^2/(y w) + x w/z + 4 x + x/z + 3 x z/w + 3 x/w + x w/y + 2 x z/y + 3 x/y + x z^2/(y w) + 3 x z/(y w) + 2 x/(y w) + 2 y w/z + 5 y + 2 y/z + 3 y z/w + 3 y/w + w + 2 w/z + 2 z + 2/z + z^2/w + 5 z/w + 4/w + w/y + 2 z/y + 2/y + z^2/(y w) + 2 z/(y w) + 1/(y w) + y^2 w/(x z) + 2 y^2/x + y^2/(x z) + y^2 z/(x w) + y^2/(x w) + 2 y w/(x z) + 4 y/x + 2 y/(x z) + 2 y z/(x w) + 2 y/(x w) + w/(x z) + 2/x + 1/(x z) + z/(x w) + 1/(x w)$\\
\oddrow 734&$x y w z + 2 x y w + x y w/z + x y z + 3 x y + 2 x y/z + x y/w + x y/(w z) + x w z + 2 x w + x w/z + x z + 4 x + 3 x/z + 2 x/w + 2 x/(w z) + x/y + x/(y z) + x/(y w) + x/(y w z) + 2 y w z + 4 y w + 2 y w/z + 2 y z + 6 y + 4 y/z + 2 y/w + 2 y/(w z) + 2 w z + 4 w + 2 w/z + 2 z + 6/z + 4/w + 4/(w z) + 2/y + 2/(y z) + 2/(y w) + 2/(y w z) + y w z/x + 2 y w/x + y w/(x z) + y z/x + 3 y/x + 2 y/(x z) + y/(x w) + y/(x w z) + w z/x + 2 w/x + w/(x z) + z/x + 4/x + 3/(x z) + 2/(x w) + 2/(x w z) + 1/(x y) + 1/(x y z) + 1/(x y w) + 1/(x y w z)$\\
\evnrow 740&$x y^2/z + x y^2/w + 2 x y w/z + 4 x y + 3 x y/z + 2 x y z/w + 2 x y/w + x w^2/z + 3 x w + 4 x w/z + 3 x z + 6 x + 3 x/z + x z^2/w + 2 x z/w + x/w + x w^2/(y z) + 2 x w/y + 2 x w/(y z) + x z/y + 2 x/y + x/(y z) + 2 y^2/z + 2 y^2/w + 4 y w/z + 8 y + 6 y/z + 4 y z/w + 4 y/w + 2 w^2/z + 6 w + 8 w/z + 6 z + 6/z + 2 z^2/w + 4 z/w + 2/w + 2 w^2/(y z) + 4 w/y + 4 w/(y z) + 2 z/y + 4/y + 2/(y z) + y^2/(x z) + y^2/(x w) + 2 y w/(x z) + 4 y/x + 3 y/(x z) + 2 y z/(x w) + 2 y/(x w) + w^2/(x z) + 3 w/x + 4 w/(x z) + 3 z/x + 6/x + 3/(x z) + z^2/(x w) + 2 z/(x w) + 1/(x w) + w^2/(x y z) + 2 w/(x y) + 2 w/(x y z) + z/(x y) + 2/(x y) + 1/(x y z)$\\
\oddrow 742&$x^3/(y z) + 4 x^2/z + x^2/w + 3 x^2 w/(y z) + 3 x^2/y + 3 x^2/(y z) + 6 x y/z + 3 x y/w + 9 x w/z + 12 x + 9 x/z + 3 x z/w + 3 x/w + 3 x w^2/(y z) + 6 x w/y + 6 x w/(y z) + 3 x z/y + 6 x/y + 3 x/(y z) + 4 y^2/z + 3 y^2/w + 9 y w/z + 15 y + 9 y/z + 6 y z/w + 6 y/w + 6 w^2/z + 15 w + 12 w/z + 12 z + 6/z + 3 z^2/w + 6 z/w + 3/w + w^3/(y z) + 3 w^2/y + 3 w^2/(y z) + 3 w z/y + 6 w/y + 3 w/(y z) + z^2/y + 3 z/y + 3/y + 1/(y z) + y^3/(x z) + y^3/(x w) + 3 y^2 w/(x z) + 6 y^2/x + 3 y^2/(x z) + 3 y^2 z/(x w) + 3 y^2/(x w) + 3 y w^2/(x z) + 9 y w/x + 6 y w/(x z) + 9 y z/x + 12 y/x + 3 y/(x z) + 3 y z^2/(x w) + 6 y z/(x w) + 3 y/(x w) + w^3/(x z) + 4 w^2/x + 3 w^2/(x z) + 6 w z/x + 9 w/x + 3 w/(x z) + 4 z^2/x + 9 z/x + 6/x + 1/(x z) + z^3/(x w) + 3 z^2/(x w) + 3 z/(x w) + 1/(x w)$\\
\end{longtable}

\end{landscape}
\bibliographystyle{amsplain}
\bibliography{bibliography}
\end{document}